\documentclass[nonblindrev]{informs3noheader}

\OneAndAHalfSpacedXI
\usepackage[colorlinks,linkcolor=blue,citecolor=blue,urlcolor=magenta,linktocpage,plainpages=false]{hyperref}
\usepackage{color}
\usepackage{bbm}
\usepackage{tabularx}
\usepackage{booktabs}
\usepackage{multirow}
\usepackage{siunitx}
\def\sym#1{\ifmmode^{#1}\else\(^{#1}\)\fi}
\usepackage{algorithm,algpseudocode}
\usepackage{amsmath,amssymb}
\usepackage{subcaption}
\usepackage[flushleft]{threeparttable}
\usepackage[nameinlink]{cleveref}
\usepackage{xcolor}
\usepackage{bm}
\usepackage[normalem]{ulem}
\usepackage{xspace}
\usepackage{thm-restate}
\usepackage{comment}

\usepackage{tikz}
\usetikzlibrary{shapes.geometric}
\usetikzlibrary{positioning}
\usetikzlibrary{decorations.pathreplacing}
\usetikzlibrary{shapes.multipart,fit}

\renewcommand{\arraystretch}{1.0}

\usepackage{braket}

\usepackage{amssymb}
\usepackage{pifont}
\newcommand{\cmark}{\ding{51}}%
\newcommand{\xmark}{\ding{55}}%

\usepackage{natbib}
 \bibpunct[, ]{(}{)}{,}{a}{}{,}%
 \def\bibfont{\small}%

\def\R{\mathbb{R}}

\def\calC{\mathcal{C}}
\def\calD{\mathcal{D}}

\def\calF{\mathcal{F}}

\def\calH{\mathcal{H}}
\def\calI{\mathcal{I}}
\def\calJ{\mathcal{J}}
\def\calK{\mathcal{K}}
\def\calL{\mathcal{L}}
\def\calM{\mathcal{M}}
\def\calN{\mathcal{N}}
\def\calO{\mathcal{O}}
\def\calP{\mathcal{P}}

\def\calR{\mathcal{R}}

\def\calT{\mathcal{T}}

\def\calX{\mathcal{X}}

\def\bA{\boldsymbol{A}}

\def\bS{\boldsymbol{S}}

\def\bV{\boldsymbol{V}}

\def\bb{\boldsymbol{b}}

\def\be{\boldsymbol{e}}

\def\bh{\boldsymbol{h}}
\def\bi{\boldsymbol{i}}

\def\br{\boldsymbol{r}}

\def\bw{\boldsymbol{w}}
\def\bx{\boldsymbol{x}}
\def\by{\boldsymbol{y}}

\def\bo{\boldsymbol{0}}
\def\bone{\boldsymbol{1}}

\def\bh{\boldsymbol{h}}

\def\bbeta{\boldsymbol{\beta}}

\def\bxi{\boldsymbol{\xi}}
\def\bXi{\boldsymbol{\Xi}}

\def\st{\text{s.t.}}

\newcommand{\E}[1]{\ensuremath \mathbb{E} \left[ #1 \right]}
\newcommand{\condE}[2]{\ensuremath \mathbb{E}_{#1} \left[ #2 \right]}
\newcommand{\Emid}[2]{\ensuremath \mathbb{E} \left[\,#1\,\middle|\,#2\,\right]}
\newcommand{\condEmid}[3]{\ensuremath \mathbb{E}_{#1} \left[\,#2\,\middle|\,#3\,\right]}
\newcommand{\Pmid}[2]{\ensuremath \mathbb{P} \left(\,#1\,\middle|\,#2\,\right)}
\newcommand{\Prob}[1]{\ensuremath \mathbb{P} \left( #1 \right)}

\renewcommand{\exp}[1]{\ensuremath \, \mathrm{exp} \left( #1 \right)}
\newcommand{\ind}[1]{\ensuremath \mathbbm{1} \left( #1 \right)}

\newcommand{\IPbp}[3]{\ensuremath \mathrm{IP}^{#1} \left(\,#2\,\middle|\,#3\,\right)}
\newcommand{\OPTbp}[3]{\ensuremath \mathrm{OPT}^{#1} \left(\,#2\,\middle|\,#3\,\right)}

\definecolor{mygreen}    {RGB}{0,90,0}
\definecolor{myblue}     {RGB}{0,51,140}
\definecolor{myorange}   {RGB}{238,118,0}
\definecolor{myred}      {RGB}{126,0,0}
\definecolor{mygray}     {RGB}{100,100,105}
\definecolor{mygrayblue} {RGB}{0,128,128}
\definecolor{mygraygreen}{RGB}{128,128,0}
\definecolor{DarkPurple}     {RGB}{142, 36, 170}
\definecolor{LightPurple}    {RGB}{57, 130, 7}

\newcommand{\ALG}{\mathrm{ALG}\xspace}
\newcommand{\e}{\varepsilon}
\renewcommand{\d}{\textrm{d}}
\newcommand{\OPT}{\mathrm{OPT}\xspace}
\newcommand{\myqed}{\hfill $\square$}
\newcommand{\mset}{\ensuremath \text{\textit{Set}}}

\TheoremsNumberedThrough
\ECRepeatTheorems
\EquationsNumberedThrough

\MANUSCRIPTNO{}

\begin{document}
\RUNAUTHOR{Baxi et al.}
\RUNTITLE{Online Rack Placement in Large-Scale Data Centers}
\TITLE{\Large Online Rack Placement in Large-Scale Data Centers:\\Online Sampling Optimization and Deployment}
\ARTICLEAUTHORS{
\AUTHOR{Saumil Baxi}
\AFF{Cloud Operations and Innovation, Microsoft}
\AUTHOR{Kayla Cummings}
\AFF{Cloud Supply Chain Sustainability Engineering, Microsoft}
\AUTHOR{Alexandre Jacquillat, Sean Lo} 
\AFF{Sloan School of Management, Operations Research Center, Massachusetts Institute of Technology}
\AUTHOR{Rob McDonald}
\AFF{Cloud Operations and Innovation, Microsoft}
\AUTHOR{Konstantina Mellou, Ishai Menache, Marco Molinaro}
\AFF{Machine Learning and Optimization, Microsoft Research}
} 

\ABSTRACT{
This paper optimizes the configuration of large-scale data centers toward cost-effective, reliable and sustainable cloud supply chains. The problem involves placing incoming racks of servers within a data center to maximize demand coverage given space, power and cooling restrictions. We formulate an online integer optimization model to support rack placement decisions. We propose a tractable online sampling optimization (OSO) approach to multi-stage stochastic optimization, which approximates unknown parameters with a sample path and re-optimizes decisions dynamically. We prove that OSO achieves a strong competitive ratio in canonical online resource allocation problems and sublinear regret in the online batched bin packing problem. Theoretical and computational results show it can outperform mean-based certainty-equivalent resolving heuristics. Our algorithm has been packaged into a software solution deployed across Microsoft's data centers, contributing an interactive decision-making process at the human-machine interface. Using deployment data, econometric tests suggest that adoption of the solution has a negative and statistically significant impact on power stranding, estimated at 1--3 percentage point. At the scale of cloud computing, these improvements in data center performance result in significant cost savings and environmental benefits.
}


\pagenumbering{arabic}
\maketitle

\vspace{-24pt}
\section{Introduction}

The cloud computing industry is projected to reach nearly one trillion dollars by 2025, with double-digit annual growth. At the heart of cloud supply chains, data centers are critical to ensuring efficient, reliable, and sustainable operations \citep{cmsz2023}. With surging demand, driven in part by the rapid expansion of artificial intelligence, data centers operations have led to high costs and energy use \citep{iea}. Data centers also face rising risks of service outages. In practice, operators typically build in buffers to mitigate overload; yet these conservative practices also lead to wasted resources and financial costs \citep{nrdc}.

In particular, an emerging problem involves allocating incoming demand within a data center. Cloud demand materializes as requests for \textit{racks}---steel framework hosting servers, cables, and computing equipment, and powering billions of queries annually as well as platform-, software-, and infrastructure-as-a-service functionalities. Each rack needs to be mounted on a tile within the data center; once placed, it becomes immovable. Rack placements are therefore pivotal toward maximizing data center utilization and ensuring reliable cloud computing operations when a power device fails. In practice, data center managers make rack placement decisions based on domain expertise and spreadsheet tools, leading to high mental overload and inefficiencies \citep{uptime}.

In response, this paper studies an online rack placement problem to optimize data center configurations. 
Rack placement features a discrete resource allocation structure to maximize utilization subject to multi-dimensional constraints from resource availability restrictions, reliability requirements, and operational requirements. 
It also features an online optimization structure due to uncertainty regarding future demand, with continuous uncertainty (power and cooling requirements), a continuous state space (power and cooling utilization), a large action space (tile-rack assignments), and a long time horizon. In turn, the paper proposes an end-to-end approach by (i) formulating an online rack placement model; (ii) developing an online algorithm for multi-stage stochastic optimization, with theoretical and computational results; (iii) developing software solutions and deploying them across Microsoft's fleet of data centers; and (iv) evaluating its impact on data center performance.

Specifically, this paper makes the following contributions:
\begin{itemize}
\item[--]	\emph{An optimization model of rack placement.} We formulate a multi-stage stochastic integer optimization model to optimize the placement of incoming rack requests. Multi-dimensional capacity constraints reflect space restrictions, cooling restrictions, and power restrictions within a multi-layer architecture, and reliability requirements to ensure operational continuity in the event of a power device failure. Moreover, coupling constraints ensure that racks from the same reservation are placed in the same row in the data center, which enables higher service levels for the end customers.
\item[--]	\emph{An online sampling optimization (OSO) algorithm, with supporting theoretical and computational results.} The OSO algorithm is designed as an easily-implementable and tractable sampling-based resolving heuristic in multi-stage stochastic optimization under exogenous uncertainty. The algorithm samples future realizations of uncertainty at each decision epoch; solves a tractable approximation of the problem; and re-optimizes decisions dynamically. With one sample path, the algorithm relies on a deterministic approximation of the problem at each iteration; with several sample paths, it relies on a two-stage stochastic approximation at each iteration.

Theoretical results show that OSO yields strong performance guarantees even with a \textit{single} sample path at each iteration. Specifically, it provides a $(1-\varepsilon_{d,T,B})$-approximation of the perfect-information optimum in canonical online resource allocation problems, where $\varepsilon_{d,T,B}$ approximately scales with the number of resources $d$ as $\calO(\sqrt{\log d})$, with the time horizon $T$ as $\calO(\sqrt{\log T})$ and with resource capacities $B$ as $\calO(1/\sqrt{B})$. In particular, OSO is asymptotically optimal if capacities scale with the time horizon as $\Omega(\log^{1+\sigma} T)$ for $\sigma>0$. Single-sample OSO also achieves sub-linear regret in the online batched bin packing problem with large-enough batches. We also prove that OSO can yield unbounded improvements as compared to myopic decision-making and mean-based certainty-equivalent resolving heuristics---thereby showing benefits of sampling and re-optimization. Computational results further demonstrate that OSO retains tractability and can return higher-quality solutions than other resolving heuristics in large-scale online optimization instances that remain intractable with stochastic programming and dynamic programming benchmarks.
\item[--]   \emph{Deployment of the model and algorithm in production.} We have packaged our optimization algorithm into a software solution and deployed it across Microsoft's fleet of data centers. As with many supply chain problems, rack placement involves complex considerations that are hard to elicit in a single optimization model. We have closely collaborated with data center managers to improve the model iteratively and capture practical considerations through phased deployment. Since its launch in March 2022, the software recommendations have been increasingly adopted. This paper constitutes one of the first large-scale deployments of a collaborative decision-support tool in data centers, contributing an interactive decision-making process at the human-machine interface.
\item[--]   \emph{Realized benefits in production across Microsoft's data centers.} We leverage post-deployment data to estimate the effect of adoption of our solution on data center performance. An important challenge in the empirical assessment is that performance can only be assessed at the data center level rather than at the rack placement level, due to combinatorial interdependencies across rack placements. Still, we exploit Microsoft's large fleet of data centers to identify the effect of adoption on power stranding, defined as the amount of wasted power within the data center. Econometric specifications based on ordinary least square regression and propensity score matching indicate a negative and statistically significant impact of adoption, with reductions in power stranding by 1--3 percentage points. At the scale of cloud supply chains, these improvements represent large efficiency gains, translating into multi-million-dollar annual cost savings along with environmental benefits.
\end{itemize}
\section{Literature Review}
\label{sec:literature}

\subsubsection*{Cloud supply chains.} Cloud computing involves many optimization problems, spanning, at the upstream level, server procurement \citep{arbabian2021capacity}, capacity expansion \citep{liu2023efficient}, and the allocation of incoming demand across data centers \citep{xl2013}; and, at the downstream level, the assignment of jobs to virtual machines \citep{schroeder2006closed, gardner2017redundancy,grosof2022optimal} and of virtual machines to servers \citep{ckmz2019,gr2020, pst2022,bfmmn2022,muir2024temporal}. In-between, rack placement focuses on the allocation of physical servers within a data center. \cite{zhang2021flex} proposed a flexible assignment of demand to power devices. \cite{mmz2023} developed a scheduling algorithm to manage power capacities. Our paper contributes a comprehensive optimization approach to data center operations, in order to manage cloud demand and supply given space, power and cooling capacities, regular and failover conditions, and multi-rack reservations.

Our paper relates to the deployment of software tools in large-scale data centers. \cite{carbondc} developed a scheduling software to mitigate the carbon footprint of Google's data centers. \cite{wu2016dynamo} deployed a dynamic power-capping software in Facebook's data centers. \cite{lyu2023hyrax} introduced a fail-in-place operational model for servers with degraded components in Microsoft's data centers. Our paper provides a new solution to support rack placements.

\subsubsection*{Stochastic online optimization.}
Multi-stage stochastic integer programs are typically formulated on a scenario tree, using sample average approximation \citep{kleywegt2002sample} or scenario reduction \citep{romisch2009scenario,bertsimas2022optimization,zhang2023optimized}. Solution methods include branch-and-bound \citep{lulli2004branch}, cutting planes \citep{guan2009cutting}, and progressive hedging \citep{rockafellar1991scenarios,gade2016obtaining}. Still, scenario trees grow exponentially with the planning horizon. Alternatively, stochastic dual dynamic programming leverages stage-wise decomposition and outer approximation \citep{pereira1991multi}; yet, its complexity scales with the number of variables, and it becomes much more challenging with integer variables due to non-convex cost-to-go functions \citep{lohndorf2013optimizing,philpott2020midas,zou2019stochastic}. Another set of methods employ linear decision rules to derive tractable approximations of multi-stage problems \citep{kuhn2011primal,bodur2022two,daryalal2024lagrangian}.

As a stochastic integer program or a dynamic program, the rack placement problem is highly challenging to solve due to continuous uncertainty, the continuous state space, the large action space, and the long time horizon. This paper proposes a tractable OSO approximation, which re-optimizes decisions dynamically based on one (or a few) sample path(s). This approach relates to certainty-equivalent (CE) controls, which approximate uncertain parameters by averages \citep{bertsekas2012dynamic}. \cite{gallego1994optimal,gallego1997multiproduct} showed that static CE heuristics achieve a $\Theta(\sqrt T)$ loss in online revenue management. Subsequent work embedded CE into resolving heuristics \citep{ciocan2012model,chen2013simple} and compared static vs. adaptive CE heuristics \citep{cooper2002asymptotic,maglaras2006dynamic,secomandi2008analysis,jasin2013analysis}. Augmented with probabilistic allocations and thresholding adjustments, CE heuristics can achieve a $o(\sqrt T)$ loss \citep{reiman2008asymptotically} and even a bounded $\calO(1)$ loss \citep{jasin2012re,bumpensanti2020re}. Bounded additive losses have also been obtained in the multi-secretary problem using a budget-ratio policy \citep{arlotto2019uniformly}; in online packing and matching using a CE resolving heuristic with probabilistic allocations and thresholding rules \citep{vera2021bayesian}; and in broader online allocation problems using an empirical CE heuristic with thresholding rules \citep{banerjee2024good}. These results rely on a discrete characterization of uncertainty.

With continuous distributions, CE resolving heuristics achieve logarithmic regret rates \citep{lueker1998average,arlotto2020logarithmic,li2022online,bray2024logarithmic}. \cite{balseiro2024survey} extended this result to a unified \textit{dynamic resource constrained reward collection problem}, which our online resource allocation problem falls into. \cite{besbes2022multi} interpolated between bounded and logarithmic regret bounds, depending on the complexity of the distribution. \cite{jiang2025degeneracy} proved $\calO(\log T)$ and $\calO((\log T)^2)$ regret bounds in network revenue management with continuous rewards. \cite{chen2025beyond} extended these results to an online multi-knapsack problem.

Finally, in online bin packing, \cite{rhee1993line} obtained a $\calO(\sqrt T\cdot \log T)$ regret rate with a resolving heuristic. \cite{gr2020} derived a $\calO(\sqrt{T})$ regret with a regularized resolving heuristic, without requiring distributional knowledge. \cite{liu2021online} obtained $\calO(\sqrt{T})$ regret under i.i.d. and random permutation models. \cite{banerjee2024good} obtained constant regret bounds, albeit with a dependency on an exponential number of action types.

Our OSO algorithm contributes a tractable resolving heuristic based on a \emph{sampled} path. This algorithm relates to sampling-based approaches in revenue management \citep{freund2021overbooking}, online matching \citep{chen2023feature}, and prophet inequalities \citep{prophetLimitedInfo,rubinsteinSingleSample,singleSampleProphet}; and to the dual mirror descent algorithm in resource allocation \citep{balseiro_best_2023}. Our paper reports new performance guarantees of OSO, notably a $\calO(\sqrt{\log T})$ competitive ratio in online resource allocation with continuous uncertainty, which depends weakly on the number of resource types $d$, in $\calO\left(\sqrt{\log d}\right)$, and scales with demand-normalized capacity $B$ in $\calO(1/\sqrt B)$. This result contributes to prior work on competitive ratios for secretary problems \citep{kleinberg,kesselheim2020knapsack}, packing and covering \citep{doi:10.1287/moor.2013.0612,agrawal2014fast,kesselheim,guptaMolinaroMOR}, and advertising \citep{10.5555/1888935.1888957,DevanurHayes09}. We also demonstrate that OSO can provide unbounded benefits versus mean-based CE algorithms, thus showing the potential of sampling and re-optimization.
\section{The Rack Placement Problem}\label{sec:model}

\subsection{Problem Statement and Mathematical Notation}

\subsubsection*{Inputs: demand.}

Data center demand materializes as racks of servers, which need to be mounted onto hardware tiles powered with appropriate power and cooling equipment. Demand requests arrive within a data center sequentially in batches. We index the planning horizon by $\calT$ and denote by $\calI^t$ the set of requests at time $t\in\calT$. Demand requests correspond to different hardware types, which result in heterogeneous resource requirements. Each request $i\in \mathcal{I}^t$ comes with $n_i$ racks, each requiring $\rho_i$ units of power and $\gamma_i$ units of cooling (see distributions in \Cref{F:historicalDemand}). A request is satisfied if all racks are placed.

\begin{figure}[h!]
    \centering
	\subfloat[Space requirements]{\label{subfig:demandSizes}\includegraphics[width=0.32\textwidth]{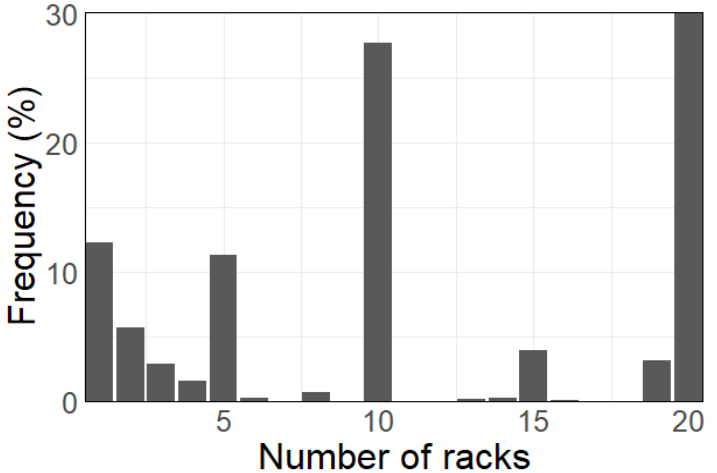}}
	\subfloat[Power requirements]{\label{subfig:demandPowerCDF}\includegraphics[width=0.3\textwidth]{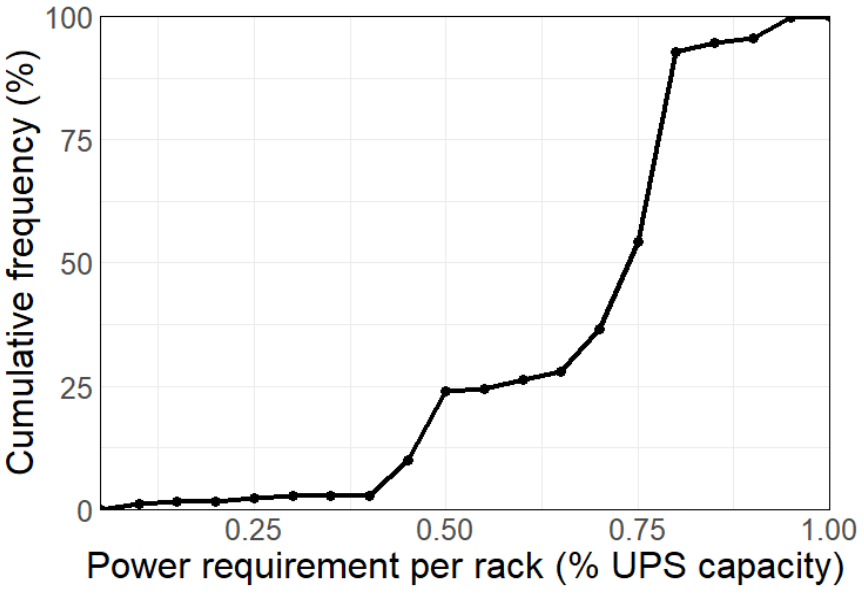}}
	\subfloat[Cooling requirements]{\label{subfig:demandCoolingCDF}\includegraphics[width=0.33\textwidth]{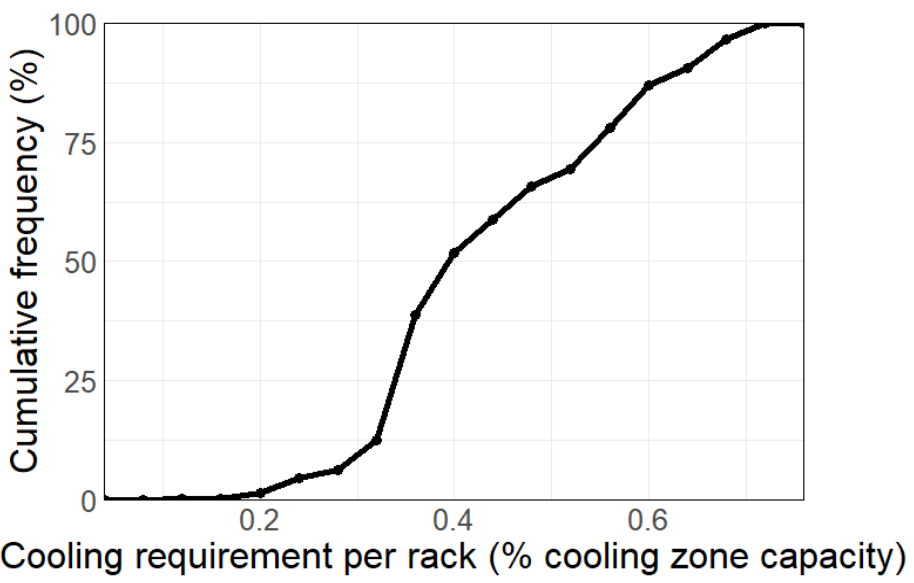}}
    \caption{
    Demand distribution across Microsoft's data centers.}
    \label{F:historicalDemand}
    \vspace{-12pt}
\end{figure}

\subsubsection*{Inputs: data center architecture.}

Data centers comprise server halls that host the main computing equipment, as well as adjacent mechanical and electrical yards that store the primary cooling systems and power generators (\Cref{subfig:space-layout}). 
This architecture creates three bottlenecks:
\begin{enumerate}
\item \emph{Physical space.} Each data center in partitioned into rooms, stored in $\calM$; each room is partitioned into rows, stored in $\calR$; each row comprises tiles which can each fit one rack of servers. All racks from the same demand request must be placed on the same row to be connected to the same networking devices---leading to better network latency. All demand requests have been prepared appropriately by engineering groups, so this constraint does not induce infeasibility by itself---although large demands may need to be rejected once the data center is close to full.

\item \emph{Power equipment.} Each room is connected to a three-level power hierarchy, shown in \Cref{subfig:power-hierarchy}. Let $\calP$ denote the set of power devices, partitioned into (i) upper-level Uninterruptible Power Supplies (UPS) devices that route power from electrical yards to each room ($\calP^{UPS}$); (ii) intermediate-level Power Delivery Units (PDU) devices that route power to the data center floor ($\calP^{PDU}$); and (iii) lower-level Power Supply Units (PSU) devices that distribute power to the tiles ($\calP^{PSU}$). For $p \in \calP^{UPS}$, we denote by $\calL_p\subset\calP$ the subset of connected power devices, including $p$ itself, all PDUs connected to $p$, and all PSUs connected to a PDU in $\calL_p$. This hierarchy defines a tree-based structure that encodes power connectivity, with capacity constraints at each node.

\begin{figure}[h!]
\centering
\subfloat[Layout of a data center room.\label{subfig:space-layout}]{\includegraphics[width=0.495\textwidth]{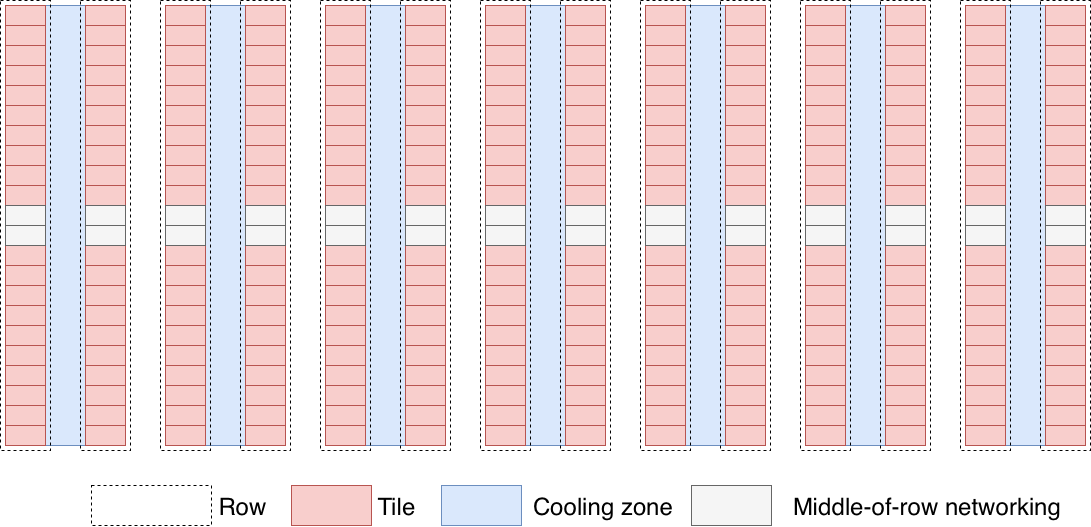}} 
\hfill
\subfloat[Sample power hierarchy.\label{subfig:power-hierarchy}]{\includegraphics[width=0.495\textwidth]{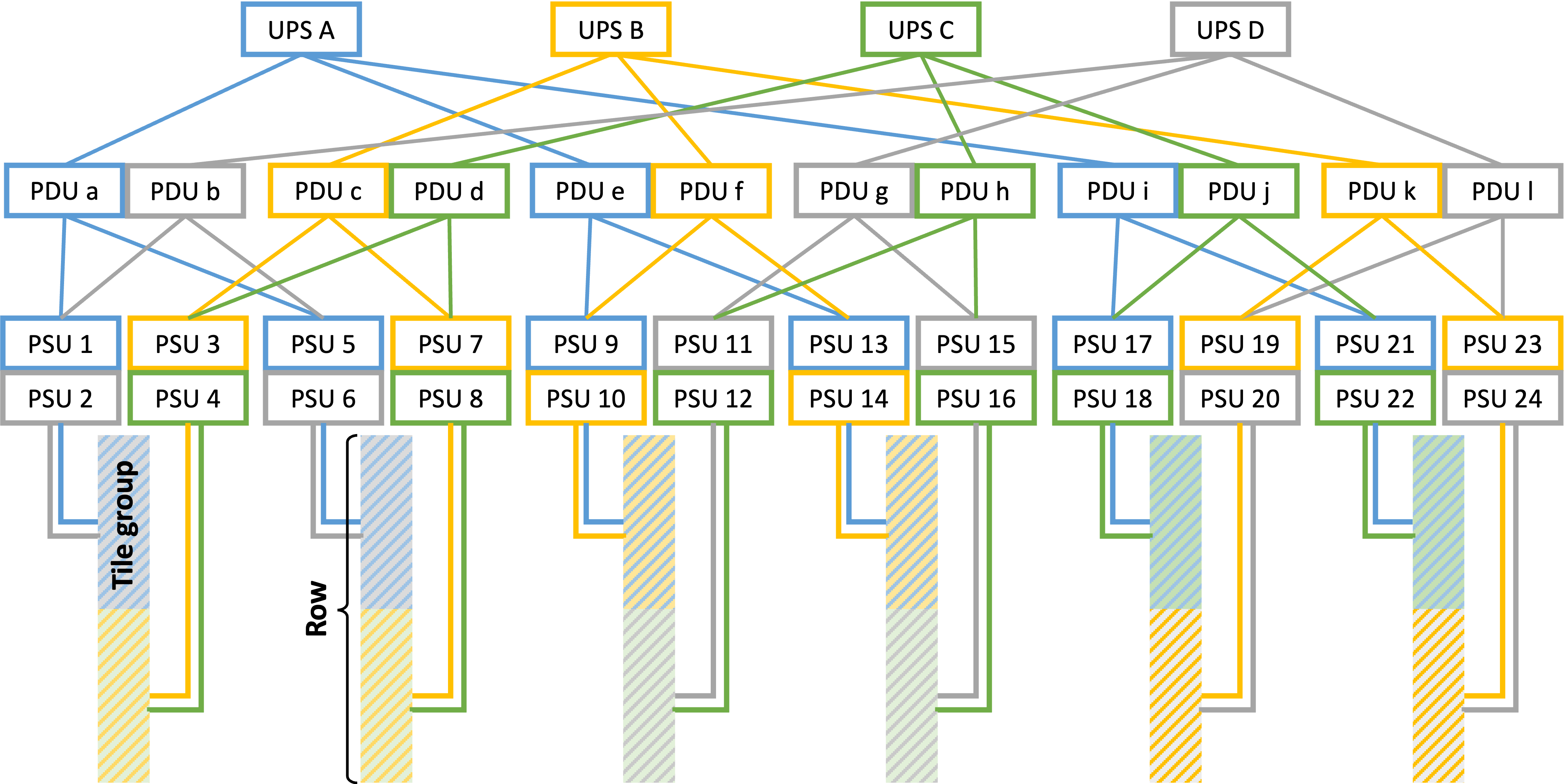}} 
\caption{Visualization of the three main operating bottlenecks in data centers: space, power and cooling.}
\vspace{-12pt}
\end{figure}

Power device failures represent one of the major risks of service outage, so data centers are configured with redundant power architectures \citep{zhang2021flex}. At Microsoft, redundancy is implemented by powering each tile with two leaf-level PSUs, each connected to different mid-level PDU devices and different top-level UPS devices (\Cref{subfig:power-hierarchy}). Under regular conditions, data centers operate in an even-allocation setting, meaning that a tile obtains half of its power from each set of connected power devices. Whenever a power device fails or is taken offline, all affected tiles must derive their power from the surviving devices.

Each device $p\in \calP$ has capacity $P_{p}$ under regular conditions and $F_p> P_{p}$ under failover conditions. The failover capacity can be supplied for a limited amount of time when another device fails. Thus, rack placements need to comply with multiple capacity restrictions across the power hierarchy (UPS, PDU, and PSU) both under regular conditions---so the shared load does not exceed the regular capacities---and under failover conditions---so that the extra load does not exceed the failover capacity of the surviving devices. In our problem, we protect against any one-off UPS failure \citep[as in][]{zhang2021flex}. This trades off efficiency and reliability by protecting against the most impactful failures (protecting against top-level UPS failures also protects against lower-level PDU and PSU failures) but protecting against one failure at a time (in practice, simultaneous failures are extremely rare).

\item \emph{Cooling equipment:} Each room includes several capacitated cooling zones, stored in a set $\calC$, that host necessary equipment to support the computing hardware. Each cooling zone $c\in\calC$ has capacity $C_c$. Each row $r\in\calR$ is connected to one cooling zone, denoted by $\text{cz}(r)\in\calC$. However, multiple rows can share the same cooling zone, creating coupling constraints across rows.

\end{enumerate}

\subsubsection*{Decisions: data center configuration.}
Rack placements determine the data center configuration by assigning each rack to a tile, which links to a cooling zone and two redundant PSU, PDU and UPS devices. Racks then become immovable due to the labor overhead, service interruptions and financial costs associated with changes in data center configurations. Minor changes can come from rack decommissioning and other out-of-scope events; however, these remain rare compared to new rack placements, due to the growing number of datacenters \citep{DC2024,DC2025} and the six-year average lifespan of servers \citep{DecomLifespan}.

To circumvent the many-to-many mapping between rows and power devices, we partition the tiles within each row into \textit{tile groups}, stored in a set $\calJ$. Specifically, each tile group $j\in\calJ$ comprises the set of tiles located within a row, denoted by $\text{row}(j)$, and connected to the same pair of PSU devices; let $\calJ_p \subset \calJ$ denote the subset of tile groups connected to each power device $p\in\calP$ within the three-level power hierarchy. We also denote by $s_j$ the number of tiles in group $j\in\calJ$. By design, all tiles in group $j\in\calJ$ are linked to the same row, the same cooling zone, and the same power devices. Thus, all tiles within the same group are indiscernible in view of the space, cooling and power constraints. To simplify the formulation and reduce model symmetry, we optimize the number of racks assigned to each \textit{tile group}, as opposed to relying on binary rack-tile assignment variables.

\Cref{fig:stranding} highlights three sources of inefficiencies in rack placement: fragmentation, resource unavailability, and failover risk. The example focuses on space and power capacities, which often act as the primary bottlenecks. This example considers two rows of three tiles; each row is powered by two distinct PSU devices connected to two distinct UPS devices (indicated in different colors). We abstract away from the intermediate PDU level in this example. We consider an incoming single-rack request with an 80W requirement. In all cases, at least one row has sufficient space to accommodate it but in neither case it can be added, for the following reasons:
\begin{itemize}
    \item[--] \textit{fragmentation}: in~\Cref{subfig:fragmentation}, available power is spread across power devices, leaving no feasible placement for incoming requests. Whereas the data center has residual capacity of 80W, each row has residual capacity of 40W and cannot handle the incoming 80W request.
    \item[--] \textit{resource unavailability}: in~\Cref{subfig:unavailability}, all placements are infeasible because of unavailable power capacity or unavailable space capacity. Specifically, UPS 1 and UPS 2 have sufficient residual power but are only connected to occupied tiles; vice versa, Row 2 has sufficient space but UPS 3 and UPS 4 do not have sufficient power to handle the incoming 80W request.
    \item[--] \textit{failover risk}: in~\Cref{subfig:redundancy}, the request could be accommodated in Row 2 under regular conditions, with 40W allocated to UPS 1 and UPS3 per the even power allocation. However, UPS 1 does not have sufficient failover capacity to handle it if UPS 2 were to fail. Specifically, UPS 1 would need to handle 200W (160W from Row 1 and 40W from Row 2), in excess of its 180W failover capacity.
\end{itemize}

\begin{figure}[h!]
\centering
\subfloat[Power fragmentation. \label{subfig:fragmentation}]{\includegraphics[width=0.325\textwidth]{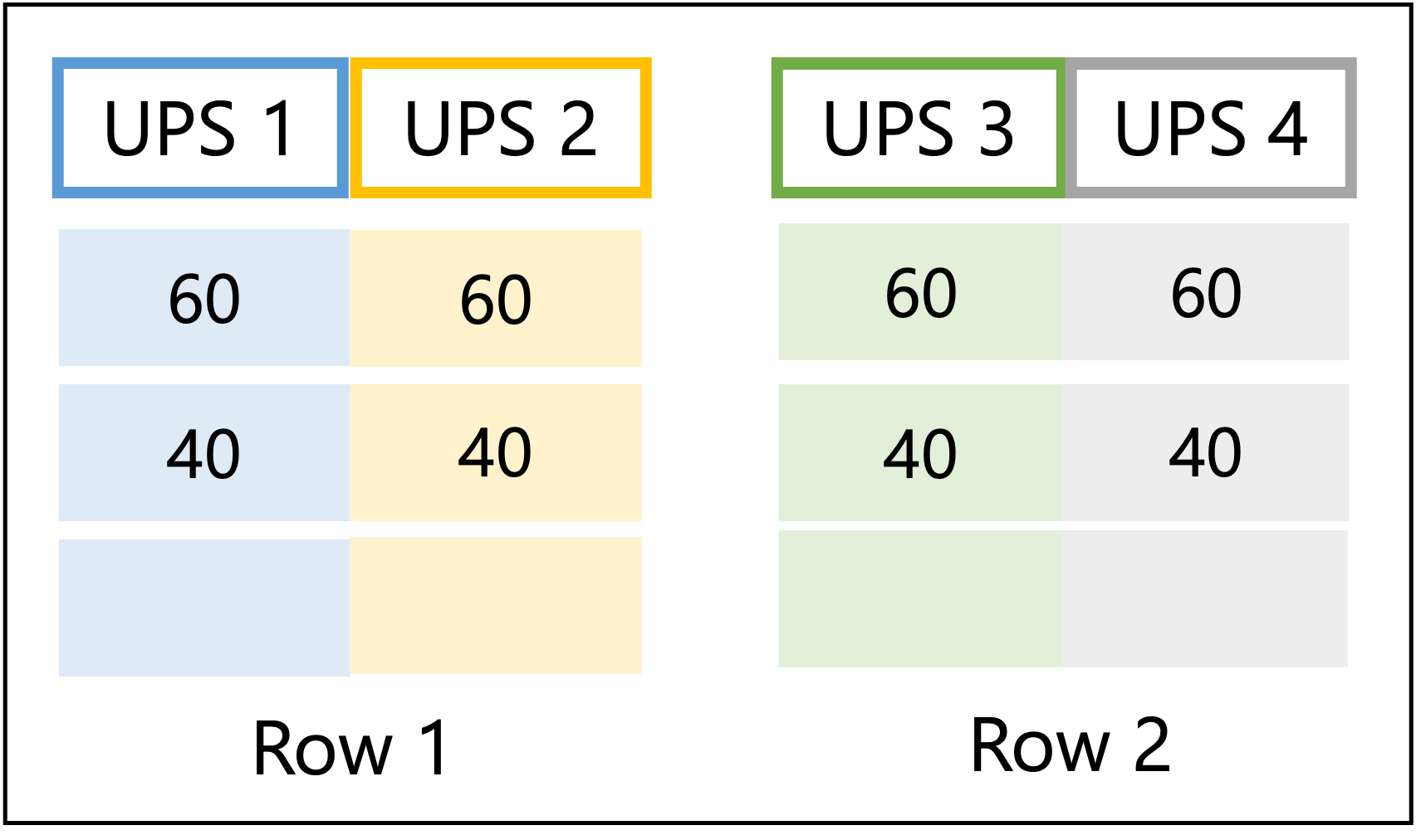}} 
\subfloat[Resource unavailability.
 \label{subfig:unavailability}]{\includegraphics[width=0.325\textwidth]{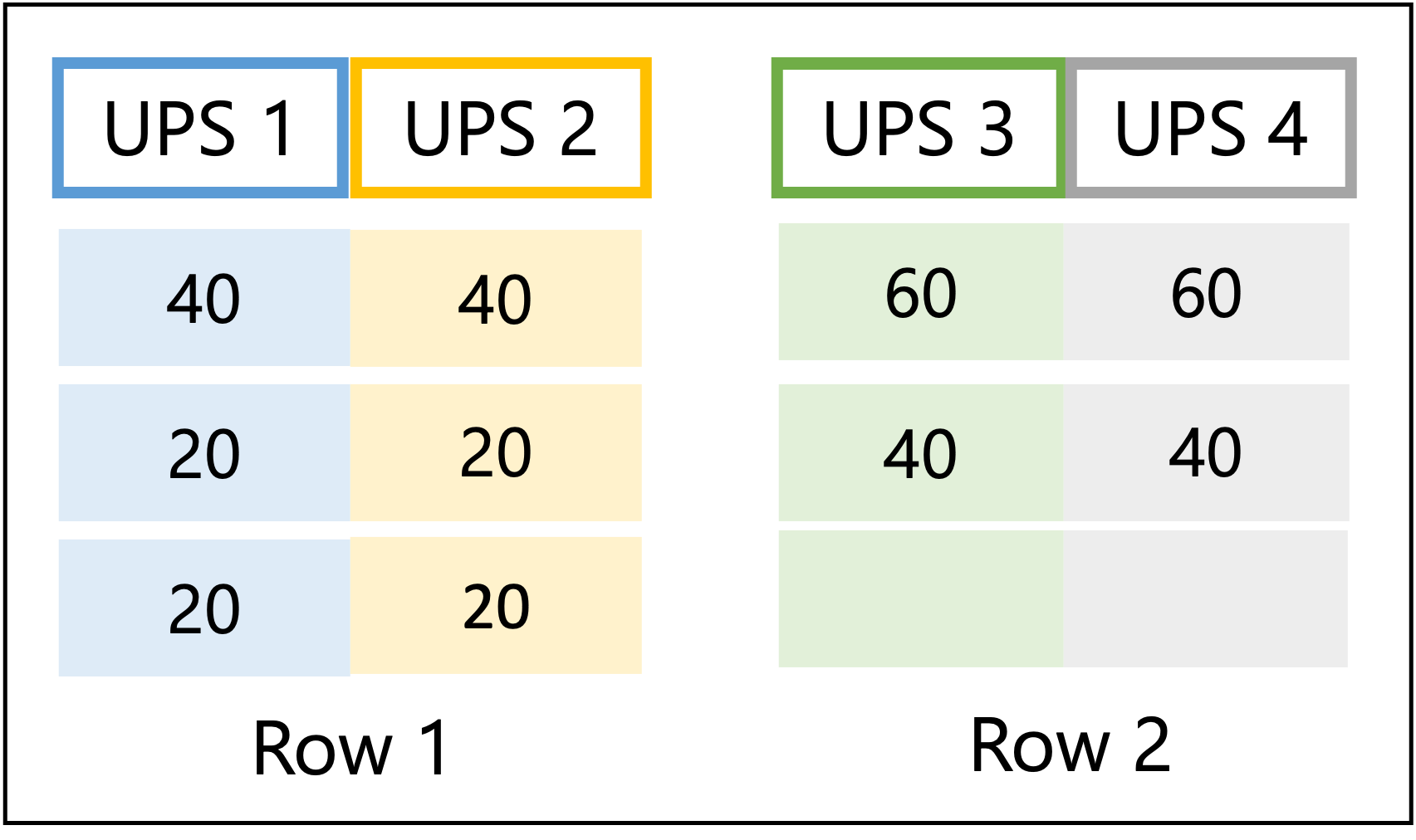}} 
\subfloat[Failover risk.
\label{subfig:redundancy}]{\includegraphics[width=0.325\textwidth]{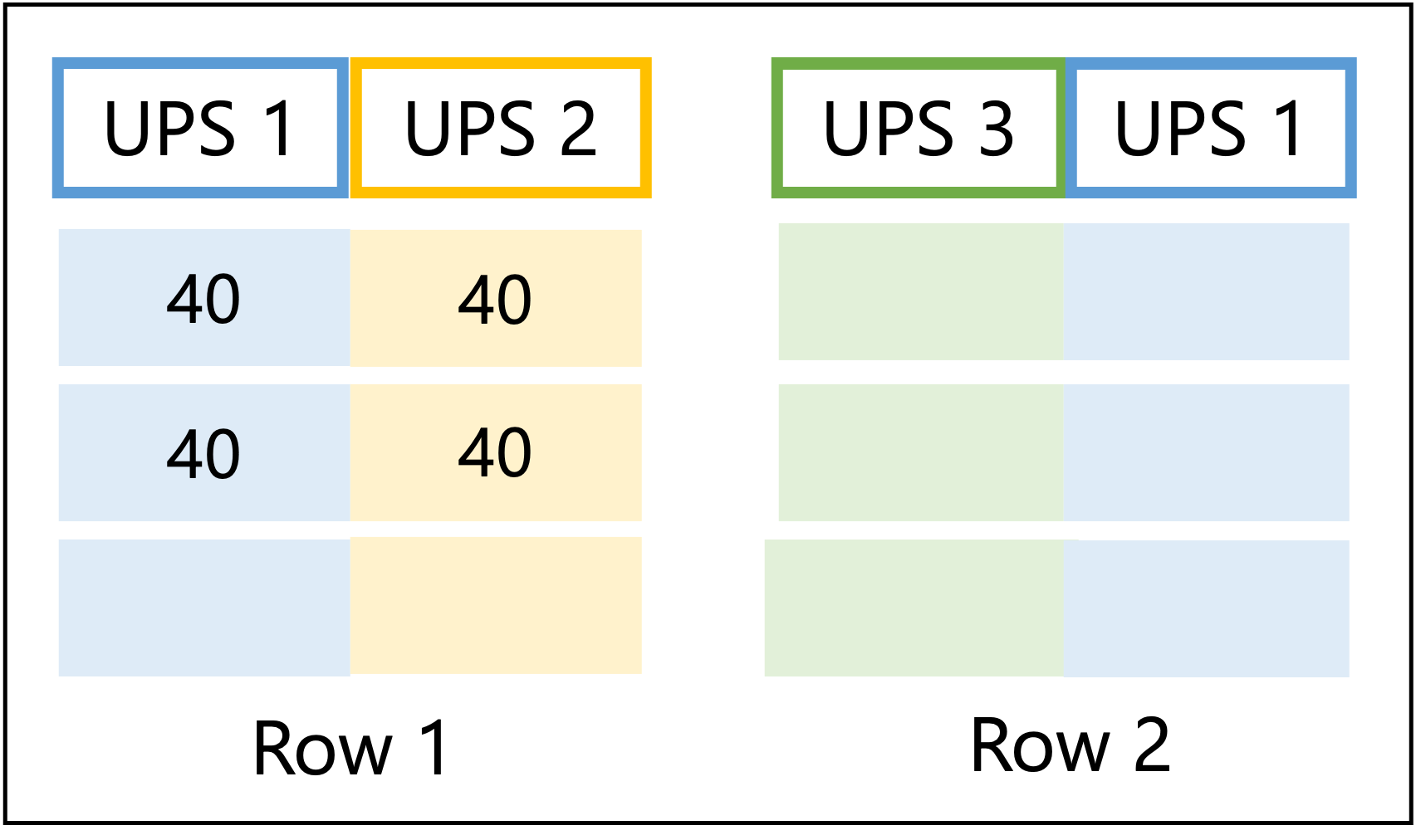}} 
\caption{Illustration of inefficiencies in rack placement operations. The example considers a single-rack request with an 80W power requirement. Each UPS device has a regular capacity of 120W and a failover capacity of 180W. Each row consists of three tiles. The number on each tile represents the amount of power that is obtained from each UPS; tiles without any number denote empty tiles.}\label{fig:stranding}
\vspace{-12pt}
\end{figure}

\subsection{Integer Optimization Formulation for Offline Rack Placement}\label{sec:formulation}

The rack placement problem aims to maximize data center utilization given resource availability and reliability requirements. We define the following decision variables:
\vspace{-6pt}
\begin{align*}
	x^t_{ij} = &
    \ \text{number of racks from request $i\in\calI^t$ from period $t\in\calT$ assigned to tile group $j\in\calJ$.}
    \\
	y^t_{ir} = &
    \begin{cases}
        1 & \text{if request $i\in\calI^t$ from period $t\in\calT$ is assigned to row $r \in \calR$,}
        \\
        0 & \text{otherwise.}
    \end{cases}
\end{align*}

The offline rack placement problem is formulated as follows. \Cref{eq:rackplacement} maximizes the reward from successful placements. \Cref{eq:rackplacement_assignment} ensures that all racks from a request are assigned to the same row, and \Cref{eq:rackplacement_linking} states that a request is placed in a row if all its racks are assigned to corresponding tile groups. \Cref{eq:rackplacement_space} to \Cref{eq:rackplacement_power} apply the space capacity, cooling capacity, and power capacity constraints under regular operations. Each incoming request $i \in \calI^t$ consists of $n_i$ racks, requiring $\gamma_i$ cooling and $\rho_i$ power per rack. The factor $\rho_i/2$ in \Cref{eq:rackplacement_power} reflects that the power requirements of each rack are shared by the two connected PSU devices and by the two sets of connected PDU and UPS devices, per the even power allocation. \Cref{eq:rackplacement_failpower} imposes the power capacity requirements under failover operations: when UPS device $p' \in \calP^{UPS}$ fails, the failover capacities of all non-connected power devices $p\in\calP\setminus\calL_{p'}$ need to accommodate their increased power load, comprising their regular load plus the additional load from all racks connected to $p$ and $p'$.
\begin{align}
    \label{eq:rackplacement}
    \max \quad 
    & \sum_{t\in\calT} \sum_{i\in\calI^t} r_{i} \sum_{r \in \mathcal{R}} y^t_{ir}
    \\
    \st \quad 
    & \sum_{r \in \mathcal{R}} y^t_{ir} \leq 1
    && 
    \ \forall \ i \in \calI^t,
    \ \forall \ t \in \calT
    \label{eq:rackplacement_assignment}
    \\
    & \sum_{\text{row}(j) = r} x^t_{ij} = n_i \cdot y^t_{ir}
    && 
    \ \forall \ i \in\calI^t, 
    \ \forall \ t \in\calT,
    \ \forall \ r \in \mathcal{R}
    \label{eq:rackplacement_linking}
    \\
    & \sum_{t=1}^T \sum_{i\in\calI^t} x^t_{ij} \leq s_j 
    && 
    \ \forall \ j \in \calJ
    \label{eq:rackplacement_space}
    \\
    & \sum_{t=1}^T \sum_{i\in\calI^t} \gamma_i \sum_{\text{cz}(\text{row}(j)) = c} x^t_{ij} \leq C_c
    && 
    \ \forall \ c \in \mathcal{C}
    \label{eq:rackplacement_cooling}
    \\
    & \sum_{t=1}^T \sum_{i\in\calI^t} \frac{\rho_i}{2} \sum_{j\in\calJ_p} x^t_{ij} \leq P_p
    && 
    \ \forall \ p \in \mathcal{P}
    \label{eq:rackplacement_power}
    \\
    & \sum_{t=1}^T \sum_{i\in\calI^t} \frac{\rho_i}{2} \left[
        \sum_{j\in\calJ_p} x^t_{ij} 
        + \sum_{j\in\calJ_p \cap \calJ_{p'}} x^t_{ij}
    \right] \leq F_p
    && 
    \ \forall \ p' \in \mathcal{P}^{\text{UPS}},
    \ \forall \ p \in \mathcal{P} \setminus \mathcal{L}_{p'}
    \label{eq:rackplacement_failpower}
    \\
    & \bx \text{ non-negative integer,} \ \by \text{ binary}
    \label{eq:rackplacement_domain}
\end{align}

We can rewrite this model in general terms to identify its resource allocation structure in \Cref{fig:structure}. The model assigns demand nodes $i\in\calI^t$ (racks, in our context) to supply nodes $j\in\calJ$ (tile groups) that map to a set of $d$ resource nodes $\calK$. In our context, resources include tile groups $j \in \calJ$, cooling zones $c \in \calC$, power devices $p \in \calP$, and coupled pairs of power devices $(p',p) \in \mathcal{P}^{\text{UPS}} \times \left( \mathcal{P} \setminus \mathcal{L}_{p'} \right)$. Let $A_{ijk}$ be the consumption of resource $k\in\calK$ if request $i \in \calI^t$ is assigned to tile group $j \in \calJ$, and $b_k$ be the capacity of resource $k\in\calK$. Specifically: (i) if $k$ indexes tile group $j'\in\calJ$, $A_{ijk} = \ind{j=j'}$ and $b_k=s_{j'}$; (ii) if $k$ indexes cooling zone $c \in \calC$, $A_{ijk} = \gamma_i \cdot \ind{\text{cz}(\text{row}(j)) = c}$ and $b_k = C_c$; (iii) if $k$ indexes power device $p \in \calP$, $A_{ijk} = \frac{\rho_i}{2}\cdot \ind{j \in \calJ_p}$ and $b_k=P_p$; and (iv) if $k$ indexes $(p',p) \in \mathcal{P}^{\text{UPS}}\times\left(\mathcal{P} \setminus \mathcal{L}_{p'}\right)$, $A_{ijk}=\frac{\rho_i}{2}\left(\ind{j \in \calJ_p}+\ind{j \in \calJ_p\cap\calJ_{p'}}\right)$, and $b_k=F_p$. The problem becomes:
\begin{align}
    \label{eq:rackplacement_resourceallocation}
    \max \quad 
    & \sum_{t \in \calT} \sum_{i \in \calI^t} r_{i} \sum_{r \in \calR} y^t_{ir}
    \\
    \st \quad 
    & \text{\Cref{eq:rackplacement_assignment}, \Cref{eq:rackplacement_linking}}
    \\
    & \sum_{t \in \calT} \sum_{i \in \calI^t} \sum_{j \in \calJ} A_{ijk} x^t_{ij} \leq b_k 
    && 
    \ \forall \ k \in \calK
    \label{eq:rackplacement_capacity_resourceallocation}
    \\
    & \bx \text{ non-negative integer,} \ \by \text{ binary}
\end{align}

Thus, the rack placement problem features a three-layer resource allocation structure between racks, tile groups, and resource nodes, along with linking constraints within multi-rack demand requests. Inputs describe the data center architecture by linking tile groups to resource nodes and by specifying the capacities at the resource nodes; and decisions determine the data center configuration by assigning racks to tile groups. More broadly, we study a three-layer resource allocation problem with demand nodes indexed by $i\in\calI$, supply nodes indexed by $j\in\calJ$, and resource nodes indexed by $k\in\calK$; with a slight abuse of notation, we use $\calI$ to refer to demand items (corresponding to individual racks, in our context) as opposed to demand requests in our rack placement formulation (multi-rack reservations). This resource allocation problem captures our rack placement problem upon relaxing the multi-rack linking constraints (\Cref{eq:rackplacement_assignment}, \Cref{eq:rackplacement_linking}).
\begin{figure}[h!]
    \centering
    \begin{tikzpicture}[scale=0.8,transform shape,every text node part/.style={align=center}]
        \node[circle,fill=myred,inner sep=3pt] (D1) at (0,10) {};
        \node[circle,fill=myred,inner sep=3pt] (D2) at (0,9.5) {};
        \node[circle,fill=myred,inner sep=3pt] (D3) at (0,9) {};
        \node[circle,fill=myred,inner sep=3pt] (D4) at (0,8.5) {};
        \node[circle,fill=myred,inner sep=3pt] (D5) at (0,8) {};
        \node[circle,fill=myred,inner sep=3pt] (D6) at (0,7.5) {};
        \node[circle,fill=black!80,inner sep=3pt] (S1) at (5,9.5) {};
        \node[circle,fill=black!80,inner sep=3pt] (S2) at (5,9) {};
        \node[circle,fill=black!80,inner sep=3pt] (S3) at (5,8.5) {};
        \node[circle,fill=black!80,inner sep=3pt] (S4) at (5,8) {};
        \node[circle,fill=DarkPurple,inner sep=3pt] (G1) at (10,9.5) {};
        \node[circle,fill=DarkPurple,inner sep=3pt] (G2) at (10,9) {};
        \node[circle,fill=DarkPurple,inner sep=3pt] (G3) at (10,8.5) {};
        \node[circle,fill=DarkPurple,inner sep=3pt] (G4) at (10,8) {};
        \node[circle,fill=myorange,inner sep=3pt] (C1) at (10,7.5) {};
        \node[circle,fill=myorange,inner sep=3pt] (C2) at (10,7) {};
        \node[circle,fill=mygreen,inner sep=3pt] (P1) at (10,6.5) {};
        \node[circle,fill=mygreen,inner sep=3pt] (P2) at (10,6) {};
        \node[circle,fill=mygreen,inner sep=3pt] (P3) at (10,5.5) {};
        \node[circle,fill=mygreen,inner sep=3pt] (P4) at (10,5) {};
        \node[circle,fill=mygreen,inner sep=3pt] (P5) at (10,4.5) {};
        \node[circle,fill=mygreen,inner sep=3pt] (P6) at (10,4) {};
        \node[circle,fill=mygreen,inner sep=3pt] (P7) at (10,3.5) {};
        \node[circle,fill=mygreen,inner sep=3pt] (P8) at (10,3) {};
        \node[circle,fill=myblue,inner sep=3pt] (PP1) at (10,2.5) {};
        \node[circle,fill=myblue,inner sep=3pt] (PP2) at (10,2) {};
        \node[circle,fill=myblue,inner sep=3pt] (PP3) at (10,1.5) {};
        \node[circle,fill=myblue,inner sep=3pt] (PP4) at (10,1) {};
        \node[circle,fill=myblue,inner sep=3pt] (PP5) at (10,.5) {};
        \node[circle,fill=myblue,inner sep=3pt] (PP6) at (10,0) {};
        \draw[->,color=black!80] (D1) to (S1);
        \draw[->,color=black!80] (D1) to (S2);
        \draw[->,color=black!80] (D1) to (S3);
        \draw[->,color=black!80] (D1) to (S4);
        \draw[->,color=black!80] (D2) to (S1);
        \draw[->,color=black!80] (D2) to (S2);
        \draw[->,color=black!80] (D2) to (S3);
        \draw[->,color=black!80] (D2) to (S4);
        \draw[->,color=black!80] (D3) to (S1);
        \draw[->,color=black!80] (D3) to (S2);
        \draw[->,color=black!80] (D3) to (S3);
        \draw[->,color=black!80] (D3) to (S4);
        \draw[->,color=black!80] (D4) to (S1);
        \draw[->,color=black!80] (D4) to (S2);
        \draw[->,color=black!80] (D4) to (S3);
        \draw[->,color=black!80] (D4) to (S4);
        \draw[->,color=black!80] (D5) to (S1);
        \draw[->,color=black!80] (D5) to (S2);
        \draw[->,color=black!80] (D5) to (S3);
        \draw[->,color=black!80] (D5) to (S4);
        \draw[->,color=black!80] (D6) to (S1);
        \draw[->,color=black!80] (D6) to (S2);
        \draw[->,color=black!80] (D6) to (S3);
        \draw[->,color=black!80] (D6) to (S4);
        \draw[->,color=DarkPurple] (S1) to (G1);
        \draw[->,color=DarkPurple] (S2) to (G2);
        \draw[->,color=DarkPurple] (S3) to (G3);
        \draw[->,color=DarkPurple] (S4) to (G4);
        \draw[->,color=myorange] (S1) to (C1);
        \draw[->,color=myorange] (S2) to (C1);
        \draw[->,color=myorange] (S3) to (C2);
        \draw[->,color=myorange] (S4) to (C2);
        \draw[->,color=mygreen] (S1) to (P1);
        \draw[->,color=mygreen] (S2) to (P3);
        \draw[->,color=mygreen] (S3) to (P4);
        \draw[->,color=mygreen] (S4) to (P6);
        \draw[->,color=mygreen] (S1) to (P7);
        \draw[->,color=mygreen] (S2) to (P7);
        \draw[->,color=mygreen] (S3) to (P8);
        \draw[->,color=mygreen] (S4) to (P8);
        \draw[->,color=myblue] (S1) to (PP1);
        \draw[->,color=myblue] (S2) to (PP3);
        \draw[->,color=myblue] (S3) to (PP4);
        \draw[->,color=myblue] (S4) to (PP6);
        \draw [decorate,decoration={brace,amplitude=8pt,mirror,raise=4ex}](0,0.5) -- (5,0.5) node[midway,yshift=-5em]{Decisions:\\Data center configuration};
        \draw [decorate,decoration={brace,amplitude=8pt,mirror,raise=4ex}](5,0.5) -- (10,0.5) node[midway,yshift=-5em]{Inputs:\\Data center architecture};
        \draw [decorate,decoration={brace,amplitude=8pt,raise=2ex}](10,9.75) -- (10,7.75) node[anchor=west,midway,xshift=2em]{Tile groups};
        \draw [decorate,decoration={brace,amplitude=8pt,raise=2ex}](10,7.75) -- (10,6.75) node[anchor=west,midway,xshift=2em]{Cooling zones};
        \draw [decorate,decoration={brace,amplitude=8pt,raise=2ex}](10,6.75) -- (10,2.75) node[anchor=west,midway,xshift=2em]{Power devices};
        \draw [decorate,decoration={brace,amplitude=8pt,raise=2ex}](10,2.75) -- (10,-.25) node[anchor=west,midway,xshift=2em]{Pairs of power devices};
        \draw[->] (0,11.25) to (0,10.5) node[above,yshift=20pt]{Demand: $i \in \calI$\\(racks)};
        \draw[->] (5,11.25) to (5,10.5) node[above,yshift=20pt]{Supply: $j \in \calJ$\\(tile groups)};
        \draw[->] (10,11.25) to (10,10.5) node[above,yshift=20pt]{Resources: $k\in\calK$\\(equipment)};
    \end{tikzpicture}
    \caption{Three-layer resource allocation structure in the rack placement problem.}
    \label{fig:structure}
    \vspace{-12pt}
\end{figure}

\subsection{The Online Rack Placement Problem}
\label{subsec:onlinerackplacement}

We now model the \emph{online rack placement problem}, where demands are revealed over time and assignment decisions are made sequentially, as a multi-stage stochastic program. Recall that demand requests come in batches $\calI^t$ under uncertainty regarding future demand batches. Let $\bXi^t = \{ (r_i, n_i, \bA_i) \}_{i \in \calI^t}$ be a random variable encapsulating uncertainty in demand requests in stage $t$, including the reward parameter, the number of racks, the cooling requirements, and the power requirements. We denote by $\bxi^t$ a realization of $\bXi^t$, by $\bxi^{1:t}=\{\bxi^1,\cdots,\bxi^t\}$ the past realizations, and by $\bxi^{(t+1):T}=\{\bxi^{t+1},\cdots,\bxi^T\}$ the future realizations. Similarly, we denote the previous decisions by $(\bx^{1:t-1}, \by^{1:t-1})$. At time $t$, placement decisions $(\bx^t, \by^t)$ are constrained by the history of observed realizations and previous decisions. The feasible set, denoted by $\calF_t(\bx^{1:t-1}, \by^{1:t-1}, \bxi^{1:t})$, includes all solutions $(\bx^t, \by^t)$ that satisfy assignment and linking constraints at period $t \in \calT$ (\Cref{eq:mssip_rackplacement_feasibleset_t_assignment} and \Cref{eq:mssip_rackplacement_feasibleset_t_linking}) and capacity constraints across periods $1,\dots,t$ (\Cref{eq:mssip_rackplacement_feasibleset_t_capacity}):
\begin{alignat}{2}
    &\sum_{r \in \calR} y^t_{ir} 
    \leq 1
    &&\quad 
    \forall \ i \in \calI^t
    \label{eq:mssip_rackplacement_feasibleset_t_assignment}
    \\
    &\sum_{\text{row}(j) = r} x^t_{ij} 
    = n_i \cdot y^t_{ir}
    &&\quad 
    \forall \ i \in \calI^t, 
    \ \forall \ r \in \calR
    \label{eq:mssip_rackplacement_feasibleset_t_linking}
    \\
    &\left(\sum_{\tau=1}^{t-1} \sum_{i \in \calI^\tau} \sum_{j \in \calJ} A_{ijk} x^\tau_{ij}\right) + \sum_{i \in \calI^t} \sum_{j \in \calJ} A_{ijk} x^t_{ij}
    \leq b_k 
    &&\quad
    \forall \ k \in \calK
    \label{eq:mssip_rackplacement_feasibleset_t_capacity}\\
    & \bx^t \text{ non-negative integer,} \ \by^t \text{ binary}
    \label{eq:mssip_rackplacement_feasibleset_t_domain}
\end{alignat}

Similarly, per \Cref{eq:rackplacement}, we define a reward function $f^t(\bx^t, \by^t, \bxi^{1:t}) = \sum_{i \in \calI^t} r_i \sum_{r \in \calR} y^{t}_{ir}$ in period $t \in \calT$, as a function of previous realizations $\bxi^{1:t}$ and the current decisions $(\bx^t, \by^t)$. We express the online rack placement problem as the following multi-stage stochastic integer program:
\begin{alignat}{2}
    \label{eq:mssip_rackplacement}
    \mathbb{E}_{\bxi^1}
    \biggl[
        \max_{(\bx^1, \by^1) \in \calF_1(\bxi^1)} \biggl\{ 
            & f^1(\bx^1, \by^1, \bxi^1)
            + 
            \mathbb{E}_{\bxi^2}
            \biggl[
                \max_{(\bx^2, \by^2)  \in \calF_2(\bx^{1}, \by^{1}, \bxi^{1:2})} \biggl\{ 
                    f^2(\bx^2, \by^2, \bxi^{1:2})
                    + \dots
                    \\
                    & + \mathbb{E}_{\bxi^T}{
                        \max_{(\bx^T, \by^T) \in \calF_T(\bx^{1:T-1}, \by^{1:T-1}, \bxi^{1:T})} 
                        \Bigl\{ f^T(\bx^T, \by^T, \bxi^{1:T})\Bigr\}
                    } 
                \dots
                \biggr\}
            \biggr]
        \biggr\} 
    \biggr]
    \nonumber
\end{alignat}

The online rack placement problem remains highly intractable. As a multi-stage stochastic program, it is complicated by the continuous uncertainty of power and cooling requirements, which would require granular discretization in the scenario tree, and by its long time horizon, which would lead to exponential growth in the scenario tree. Moreover, at each node of the scenario tree, the problem involves a discrete optimization structure to assign incoming racks to tile groups. As a dynamic program, the problem involves a large action space and a continuous state space, which also hinders the scalability of available algorithms (see \Cref{app:resourcealloc_comp_exact}). These difficulties motivate our online sampling-based optimization algorithm in the next section to solve it dynamically.
\section{Online Sampling Optimization (OSO)}\label{sec:OSO}

The OSO algorithm provides an easily-implementable, tractable and generalizable sampling-based resolving heuristic in multi-stage stochastic optimization that (i) represents uncertainty with a single sample path (or a few sample paths); (ii) approximates the problem at each decision epoch via deterministic optimization (or two-stage stochastic optimization); and (iii) re-optimizes decisions dynamically in a rolling horizon. We prove that, even with the single-sample approximation, OSO provides strong theoretical guarantees in canonical online optimization problems, and that it can achieve unbounded benefits as compared to myopic and mean-based certainty-equivalent (CE) resolving heuristics. We also report computational results showing that OSO yields high-quality solutions across large-scale online optimization problems, including online rack placement, for which stochastic and dynamic programming methods remain intractable.

We consider a general-purpose framework for multi-stage discrete optimization under uncertainty, with a separable objective function, separable additive constraints, and exogenous uncertainty. We assume that $(\bXi^1, \dots, \bXi^T)$ are independent and identically distributed following distribution $\calD$. Simplifying the notation from \Cref{subsec:onlinerackplacement}, we let $\bx^t$ be the decision variable at period $t\in\calT$, $\bx^{1:t-1}$ the previous decisions, and $\calX^t$ a mixed-integer domain. We define a cost function $f^t(\bx^t,\bxi^{1:t})$ at time $t\in\calT$. The decision $\bx^t$ is constrained by the previous decisions $\bx^{1:t-1}$ and the realizations $\bxi^{1:t}$ based on the following constraints defining a feasible region $\calF_t(\bx^{1:t-1}, \bxi^{1:t})$:
\begin{subequations}
\label{eq:mssip_feasibleset}
\begin{alignat}{2}
    \calF_t(\bx^{1:t-1}, \bxi^{1:t}) 
    & = \Set{ 
        \bx^t \in \calX^t
        |
        \ \sum_{\tau=1}^{t-1} \bA^\tau(\bxi^{1:\tau}) \bx^\tau 
        \ + \bA^t(\bxi^{1:t}) \bx^t 
        \leq \bh^t(\bxi^{1:t})
    }
\end{alignat}
\end{subequations}

The stochastic optimization problem is then formulated as follows:
\begin{alignat}{2}
\label{eq:mssip}
    \mathbb{E}_{\bxi^1}
        \biggl[
        \min_{ \bx^1 \in \calF_1(\bxi^1)} 
        \biggl\{ 
            f^1(\bx^1, \bxi^1) 
            & + \mathbb{E}_{\bxi^2} \biggl[  \min_{ \bx^2 \in \calF_2(\bx^{1}, \bxi^{1:2})} \biggl\{ f^2(\bx^2, \bxi^{1:2})  + \ldots
            \\
            & 
            + \condE{\bxi^T}{
                \min_{\bx^T \in \calF_T(\bx^{1:T-1}, \bxi^{1:T})} 
                \Bigl\{ f^T(\bx^T, \bxi^{1:T}) \Bigr\}
            } \ldots
        \biggr\}
    \biggr]\biggr\}\biggr]
    \nonumber
\end{alignat}

\subsection{The OSO Algorithm}\label{sec:generalOSO}

The algorithm optimizes decisions dynamically using an online implementation of a tractable sampling-based approximation of the problem (\Cref{alg:OSO}). In its simplest version, the algorithm relies on a single-sample deterministic approximation at each iteration; otherwise, it relies on a small-sample two-stage stochastic optimization. In period $t\in\calT$, it generates $S\geq1$ sample paths from period $t+1$ to period $T$, denoted by $\widetilde\bxi^{t+1}_s,\dots,\widetilde\bxi^T_s$ for $s=1,\cdots,S$; it then solves the resulting optimization model; and it implements the immediate decision $\bx^t$. The realization $\bxi^{t+1}$ is then revealed and the algorithm proceeds iteratively by re-optimizing decisions in period $t+1$ onward. We also add a problem-specific regularizer $\Psi(\bx^t)$ which can provide an extra level of flexibility to adjust decisions based on future demand using problem-specific characteristics.

\begin{algorithm}[h!]
    \caption{Online Sampling Optimization (OSO) algorithm.}\label{alg:OSO}
    \begin{algorithmic} 
    \item Input: problem data, number of sample paths at each iteration $S$
    \item Repeat, for $t=1,\dots,T$:
    \begin{itemize}
    \item[] \textbf{Observe:} Observe realization $\bxi^t$.
    \item[] \textbf{Sample:} Collect $S$ sample paths ${ \widetilde{\bxi}_1^{t+1:T}, \dots, \widetilde{\bxi}_S^{t+1:T}}$ from the distribution $\calD$.
    \item[] \textbf{Optimize:} Solve the following problem; store optimal solution $\widetilde{\bx}^t$, $(\widetilde{\bx}^{t+1:T}_1, \dots, \widetilde{\bx}^{t+1:T}_S)$:
        \begin{subequations}
        \label{eq:mssip_oso}
        \begin{alignat}{2}
            \min \quad 
            & f^t(\bx^t, \bxi^{1:t})
            + \frac{1}{S} \sum_{s=1}^S \sum_{\tau = t+1}^T 
                f^{\tau} \bigg( 
                    \bx^\tau_s, 
                    (\bxi^{1:t}, \widetilde{\bxi}^{t+1:\tau}_s)
                \bigg)
                + \Psi(\bx^t)
            \\
            \st \quad 
            & \bx^t \in \calF_t(\overline{\bx}^{1:t-1}, \bxi^{1:t}) \\
            & \bx^\tau_s \in \calF_{\tau} \bigl( 
                (\overline{\bx}^{1:t-1}, \bx^t, \bx^{t+1:\tau-1}_s), 
                (\bxi^{1:t}, \widetilde{\bxi}^{t+1:\tau}_s) \bigr) 
            \ \forall \ \tau \in \{t+1, \dots, T\}, 
            \ \forall \ s \in \{1, \dots, S\}
        \end{alignat}
        \end{subequations}
    \item[] \textbf{Implement:} Implement $\overline{\bx}^t = \widetilde{\bx}^t$, discarding $(\widetilde{\bx}^{t+1:T}_1, \dots, \widetilde{\bx}^{t+1:T}_S)$.
    \end{itemize}
\end{algorithmic}
\end{algorithm}

By design, OSO retains a tractable structure at each decision epoch by relying on a single-sample or a small-sample approximation of uncertainty, illustrated in \Cref{fig:OSO_SAA} in a three-period example. Tractability of the sampling step is guaranteed as long as a sample path can be generated efficiently from distribution $\calD$ (which is the case in all problems considered in this paper). Tractability of the optimization step stems from the deterministic approximation in the single-sample variant, or, in its multi-sample variant, from the two-stage stochastic approximation that relaxes the non-anticipativity constraints in periods $t+1$ onward. In comparison, scenario-tree representations grow exponentially large with the planning horizon. Obviously, the single-sample or small-sample model simplifies the representation of uncertainty at each decision epoch, but the OSO algorithm attempts to mitigate approximation errors via dynamic re-optimization. This relates to certainty-equivalent resolving heuristics, with the difference that OSO leverages a sample path at each iteration rather than expected values. As we shall see theoretically and computationally, the sampling-based approach can outperform mean-based resolving heuristics; moreover, we report in the next section theoretical results showing that the OSO algorithm can provide strong approximations of the perfect-information optimum in canonical online optimization problems.

\begin{figure}[h!]
\centering
\begin{tikzpicture}[scale=0.6,transform shape]
    \node[] at (6,16.6){\large\textbf{Single-sample OSO}};
    \node[rectangle,fill=myred,minimum size=0.5cm] (state1a) at (0,15) {};
    \node[circle,fill=myred,minimum size=0.5cm] (decision1a) at (2,15) {};
    \node[rectangle,fill=black!50,minimum size=0.5cm] (state2a) at (4,15) {};
    \node[circle,fill=black!50,minimum size=0.5cm] (decision2a) at (6,15) {};
    \node[rectangle,fill=black!50,minimum size=0.5cm] (state3a) at (8,15) {};
    \node[circle,fill=black!50,minimum size=0.5cm] (decision3a) at (10,15) {};
    \node[rectangle,fill=black!50,minimum size=0.5cm] (state4a) at (12,15) {};
    \draw[->,ultra thick,color=myred] (state1a) to (decision1a);
    \draw[->,ultra thick,color=black!50] (decision1a) to (state2a);
    \draw[->,ultra thick,color=black!50] (state2a) to (decision2a);
    \draw[->,ultra thick,color=black!50] (decision2a) to (state3a);
    \draw[->,ultra thick,color=black!50] (state3a) to (decision3a);
    \draw[->,ultra thick,color=black!50] (decision3a) to (state4a);
    \node[rectangle,fill=myred,minimum size=0.5cm] (state2b) at (4,14) {};
    \node[circle,fill=myred,minimum size=0.5cm] (decision2b) at (6,14) {};
    \node[rectangle,fill=black!50,minimum size=0.5cm] (state3b) at (8,14) {};
    \node[circle,fill=black!50,minimum size=0.5cm] (decision3b) at (10,14) {};
    \node[rectangle,fill=black!50,minimum size=0.5cm] (state4b) at (12,14) {};
    \draw[->,ultra thick,color=myred] (state2b) to (decision2b);
    \draw[->,ultra thick,color=black!50] (decision2b) to (state3b);
    \draw[->,ultra thick,color=black!50] (state3b) to (decision3b);
    \draw[->,ultra thick,color=black!50] (decision3b) to (state4b);
    \node[rectangle,fill=myred,minimum size=0.5cm] (state3c) at (8,13) {};
    \node[circle,fill=myred,minimum size=0.5cm] (decision3c) at (10,13) {};
    \node[rectangle,fill=black!50,minimum size=0.5cm] (state4c) at (12,13) {};
    \draw[->,ultra thick,color=myred] (state3c) to (decision3c);
    \draw[->,ultra thick,color=black!50] (decision3c) to (state4c);
    \node[rectangle,fill=myred,minimum size=0.5cm] (state4d) at (12,12) {};
    \draw[->,dotted,ultra thick,color=myred] (decision1a) |- (state2b);
    \draw[->,dotted,ultra thick,color=myred] (decision2b) |- (state3c);
    \draw[->,dotted,ultra thick,color=myred] (decision3c) |- (state4d);
    \node[] at (6,10){\large\textbf{Small-sample OSO}};
    \node[rectangle,fill=myred,minimum size=0.5cm] (state1a) at (0,8) {};
    \node[circle,fill=myred,minimum size=0.5cm] (decision1a) at (2,8) {};
    \node[rectangle,fill=black!50,minimum size=0.5cm] (state2as1) at (4,8.6) {};
    \node[circle,fill=black!50,minimum size=0.5cm] (decision2as1) at (6,8.6) {};
    \node[rectangle,fill=black!50,minimum size=0.5cm] (state3as1) at (8,8.6) {};
    \node[circle,fill=black!50,minimum size=0.5cm] (decision3as1) at (10,8.6) {};
    \node[rectangle,fill=black!50,minimum size=0.5cm] (state4as1) at (12,8.6) {};
    \node[rectangle,fill=black!50,minimum size=0.5cm] (state2as2) at (4,8) {};
    \node[circle,fill=black!50,minimum size=0.5cm] (decision2as2) at (6,8) {};
    \node[rectangle,fill=black!50,minimum size=0.5cm] (state3as2) at (8,8) {};
    \node[circle,fill=black!50,minimum size=0.5cm] (decision3as2) at (10,8) {};
    \node[rectangle,fill=black!50,minimum size=0.5cm] (state4as2) at (12,8) {};
    \node[rectangle,fill=black!50,minimum size=0.5cm] (state2as3) at (4,7.4) {};
    \node[circle,fill=black!50,minimum size=0.5cm] (decision2as3) at (6,7.4) {};
    \node[rectangle,fill=black!50,minimum size=0.5cm] (state3as3) at (8,7.4) {};
    \node[circle,fill=black!50,minimum size=0.5cm] (decision3as3) at (10,7.4) {};
    \node[rectangle,fill=black!50,minimum size=0.5cm] (state4as3) at (12,7.4) {};
    \draw[->,ultra thick,color=myred] (state1a) to (decision1a);
    \draw[->,ultra thick,color=black!50] (decision1a) to (state2as1);
    \draw[->,ultra thick,color=black!50] (state2as1) to (decision2as1);
    \draw[->,ultra thick,color=black!50] (decision2as1) to (state3as1);
    \draw[->,ultra thick,color=black!50] (state3as1) to (decision3as1);
    \draw[->,ultra thick,color=black!50] (decision3as1) to (state4as1);
    \draw[->,ultra thick,color=black!50] (decision1a) to (state2as2);
    \draw[->,ultra thick,color=black!50] (state2as2) to (decision2as2);
    \draw[->,ultra thick,color=black!50] (decision2as2) to (state3as2);
    \draw[->,ultra thick,color=black!50] (state3as2) to (decision3as2);
    \draw[->,ultra thick,color=black!50] (decision3as2) to (state4as2);
    \draw[->,ultra thick,color=black!50] (decision1a) to (state2as3);
    \draw[->,ultra thick,color=black!50] (state2as3) to (decision2as3);
    \draw[->,ultra thick,color=black!50] (decision2as3) to (state3as3);
    \draw[->,ultra thick,color=black!50] (state3as3) to (decision3as3);
    \draw[->,ultra thick,color=black!50] (decision3as3) to (state4as3);
    \node[rectangle,fill=myred,minimum size=0.5cm] (state2b) at (4,5.8) {};
    \node[circle,fill=myred,minimum size=0.5cm] (decision2b) at (6,5.8) {};
    \node[rectangle,fill=black!50,minimum size=0.5cm] (state3bs1) at (8,6.4) {};
    \node[circle,fill=black!50,minimum size=0.5cm] (decision3bs1) at (10,6.4) {};
    \node[rectangle,fill=black!50,minimum size=0.5cm] (state4bs1) at (12,6.4) {};
    \node[rectangle,fill=black!50,minimum size=0.5cm] (state3bs2) at (8,5.8) {};
    \node[circle,fill=black!50,minimum size=0.5cm] (decision3bs2) at (10,5.8) {};
    \node[rectangle,fill=black!50,minimum size=0.5cm] (state4bs2) at (12,5.8) {};
    \node[rectangle,fill=black!50,minimum size=0.5cm] (state3bs3) at (8,5.2) {};
    \node[circle,fill=black!50,minimum size=0.5cm] (decision3bs3) at (10,5.2) {};
    \node[rectangle,fill=black!50,minimum size=0.5cm] (state4bs3) at (12,5.2) {};
    \draw[->,ultra thick,color=myred] (state2b) to (decision2b);
    \draw[->,ultra thick,color=black!50] (decision2b) to (state3bs1);
    \draw[->,ultra thick,color=black!50] (state3bs1) to (decision3bs1);
    \draw[->,ultra thick,color=black!50] (decision3bs1) to (state4bs1);
    \draw[->,ultra thick,color=black!50] (decision2b) to (state3bs2);
    \draw[->,ultra thick,color=black!50] (state3bs2) to (decision3bs2);
    \draw[->,ultra thick,color=black!50] (decision3bs2) to (state4bs2);
    \draw[->,ultra thick,color=black!50] (decision2b) to (state3bs3);
    \draw[->,ultra thick,color=black!50] (state3bs3) to (decision3bs3);
    \draw[->,ultra thick,color=black!50] (decision3bs3) to (state4bs3);
    \node[rectangle,fill=myred,minimum size=0.5cm] (state3c) at (8,3.6) {};
    \node[circle,fill=myred,minimum size=0.5cm] (decision3c) at (10,3.6) {};
    \node[rectangle,fill=black!50,minimum size=0.5cm] (state4cs1) at (12,4.2) {};
    \node[rectangle,fill=black!50,minimum size=0.5cm] (state4cs2) at (12,3.6) {};
    \node[rectangle,fill=black!50,minimum size=0.5cm] (state4cs3) at (12,3) {};
    \draw[->,ultra thick,color=myred] (state3c) to (decision3c);
    \draw[->,ultra thick,color=black!50] (decision3c) to (state4cs1);
    \draw[->,ultra thick,color=black!50] (decision3c) to (state4cs2);
    \draw[->,ultra thick,color=black!50] (decision3c) to (state4cs3);
    \node[rectangle,fill=myred,minimum size=0.5cm] (state4d) at (12,2) {};
    \draw[->,dotted,ultra thick,color=myred] (decision1a) |- (state2b);
    \draw[->,dotted,ultra thick,color=myred] (decision2b) |- (state3c);
    \draw[->,dotted,ultra thick,color=myred] (decision3c) |- (state4d);
    \node[] at (22,16.6){\large\textbf{Scenario-tree representation}};
    \node[rectangle,fill=myred,minimum size=0.5cm] (state1) at (14,7.8) {};
    \node[circle,fill=myred,minimum size=0.5cm] (decision1) at (16,7.8) {};
    \node[rectangle,fill=myred,minimum size=0.5cm] (state2s1) at (18,13.2) {};
    \node[rectangle,fill=myred,minimum size=0.5cm] (state2s2) at (18,7.8) {};
    \node[rectangle,fill=myred,minimum size=0.5cm] (state2s3) at (18,2.4) {};
    \node[circle,fill=myred,minimum size=0.5cm] (decision2s1) at (20,13.2) {};
    \node[circle,fill=myred,minimum size=0.5cm] (decision2s2) at (20,7.8) {};
    \node[circle,fill=myred,minimum size=0.5cm] (decision2s3) at (20,2.4) {};
    \node[rectangle,fill=myred,minimum size=0.5cm] (state3s1s1) at (22,15.0) {};
    \node[rectangle,fill=myred,minimum size=0.5cm] (state3s1s2) at (22,13.2) {};
    \node[rectangle,fill=myred,minimum size=0.5cm] (state3s1s3) at (22,11.4) {};
    \node[rectangle,fill=myred,minimum size=0.5cm] (state3s2s1) at (22,9.6) {};
    \node[rectangle,fill=myred,minimum size=0.5cm] (state3s2s2) at (22,7.8) {};
    \node[rectangle,fill=myred,minimum size=0.5cm] (state3s2s3) at (22,6.0) {};
    \node[rectangle,fill=myred,minimum size=0.5cm] (state3s3s1) at (22,4.2) {};
    \node[rectangle,fill=myred,minimum size=0.5cm] (state3s3s2) at (22,2.4) {};
    \node[rectangle,fill=myred,minimum size=0.5cm] (state3s3s3) at (22,0.6) {};
    \node[circle,fill=myred,minimum size=0.5cm] (decision3s1s1) at (24,15.0) {};
    \node[circle,fill=myred,minimum size=0.5cm] (decision3s1s2) at (24,13.2) {};
    \node[circle,fill=myred,minimum size=0.5cm] (decision3s1s3) at (24,11.4) {};
    \node[circle,fill=myred,minimum size=0.5cm] (decision3s2s1) at (24,9.6) {};
    \node[circle,fill=myred,minimum size=0.5cm] (decision3s2s2) at (24,7.8) {};
    \node[circle,fill=myred,minimum size=0.5cm] (decision3s2s3) at (24,6.0) {};
    \node[circle,fill=myred,minimum size=0.5cm] (decision3s3s1) at (24,4.2) {};
    \node[circle,fill=myred,minimum size=0.5cm] (decision3s3s2) at (24,2.4) {};
    \node[circle,fill=myred,minimum size=0.5cm] (decision3s3s3) at (24,0.6) {};
    \node[rectangle,fill=myred,minimum size=0.5cm] (state4s1s1s1) at (26,15.6) {};
    \node[rectangle,fill=myred,minimum size=0.5cm] (state4s1s1s2) at (26,15.0) {};
    \node[rectangle,fill=myred,minimum size=0.5cm] (state4s1s1s3) at (26,14.4) {};
    \node[rectangle,fill=myred,minimum size=0.5cm] (state4s1s2s1) at (26,13.8) {};
    \node[rectangle,fill=myred,minimum size=0.5cm] (state4s1s2s2) at (26,13.2) {};
    \node[rectangle,fill=myred,minimum size=0.5cm] (state4s1s2s3) at (26,12.6) {};
    \node[rectangle,fill=myred,minimum size=0.5cm] (state4s1s3s1) at (26,12.0) {};
    \node[rectangle,fill=myred,minimum size=0.5cm] (state4s1s3s2) at (26,11.4) {};
    \node[rectangle,fill=myred,minimum size=0.5cm] (state4s1s3s3) at (26,10.8) {};
    \node[rectangle,fill=myred,minimum size=0.5cm] (state4s2s1s1) at (26,10.2) {};
    \node[rectangle,fill=myred,minimum size=0.5cm] (state4s2s1s2) at (26,9.6) {};
    \node[rectangle,fill=myred,minimum size=0.5cm] (state4s2s1s3) at (26,9.0) {};
    \node[rectangle,fill=myred,minimum size=0.5cm] (state4s2s2s1) at (26,8.4) {};
    \node[rectangle,fill=myred,minimum size=0.5cm] (state4s2s2s2) at (26,7.8) {};
    \node[rectangle,fill=myred,minimum size=0.5cm] (state4s2s2s3) at (26,7.2) {};
    \node[rectangle,fill=myred,minimum size=0.5cm] (state4s2s3s1) at (26,6.6) {};
    \node[rectangle,fill=myred,minimum size=0.5cm] (state4s2s3s2) at (26,6.0) {};
    \node[rectangle,fill=myred,minimum size=0.5cm] (state4s2s3s3) at (26,5.4) {};
    \node[rectangle,fill=myred,minimum size=0.5cm] (state4s3s1s1) at (26,4.8) {};
    \node[rectangle,fill=myred,minimum size=0.5cm] (state4s3s1s2) at (26,4.2) {};
    \node[rectangle,fill=myred,minimum size=0.5cm] (state4s3s1s3) at (26,3.6) {};
    \node[rectangle,fill=myred,minimum size=0.5cm] (state4s3s2s1) at (26,3.0) {};
    \node[rectangle,fill=myred,minimum size=0.5cm] (state4s3s2s2) at (26,2.4) {};
    \node[rectangle,fill=myred,minimum size=0.5cm] (state4s3s2s3) at (26,1.8) {};
    \node[rectangle,fill=myred,minimum size=0.5cm] (state4s3s3s1) at (26,1.2) {};
    \node[rectangle,fill=myred,minimum size=0.5cm] (state4s3s3s2) at (26,0.6) {};
    \node[rectangle,fill=myred,minimum size=0.5cm] (state4s3s3s3) at (26,0.0) {};
    \draw[->,ultra thick,color=myred] (state1) to (decision1);
    \draw[->,ultra thick,color=myred] (decision1) to (state2s1);
    \draw[->,ultra thick,color=myred] (decision1) to (state2s2);
    \draw[->,ultra thick,color=myred] (decision1) to (state2s3);
    \draw[->,ultra thick,color=myred] (state2s1) to (decision2s1);
    \draw[->,ultra thick,color=myred] (state2s2) to (decision2s2);
    \draw[->,ultra thick,color=myred] (state2s3) to (decision2s3);
    \draw[->,ultra thick,color=myred] (decision2s1) to (state3s1s1);
    \draw[->,ultra thick,color=myred] (decision2s1) to (state3s1s2);
    \draw[->,ultra thick,color=myred] (decision2s1) to (state3s1s3);
    \draw[->,ultra thick,color=myred] (decision2s2) to (state3s2s1);
    \draw[->,ultra thick,color=myred] (decision2s2) to (state3s2s2);
    \draw[->,ultra thick,color=myred] (decision2s2) to (state3s2s3);
    \draw[->,ultra thick,color=myred] (decision2s3) to (state3s3s1);
    \draw[->,ultra thick,color=myred] (decision2s3) to (state3s3s2);
    \draw[->,ultra thick,color=myred] (decision2s3) to (state3s3s3);
    \draw[->,ultra thick,color=myred] (state3s1s1) to (decision3s1s1);
    \draw[->,ultra thick,color=myred] (state3s1s2) to (decision3s1s2);
    \draw[->,ultra thick,color=myred] (state3s1s3) to (decision3s1s3);
    \draw[->,ultra thick,color=myred] (state3s2s1) to (decision3s2s1);
    \draw[->,ultra thick,color=myred] (state3s2s2) to (decision3s2s2);
    \draw[->,ultra thick,color=myred] (state3s2s3) to (decision3s2s3);
    \draw[->,ultra thick,color=myred] (state3s3s1) to (decision3s3s1);
    \draw[->,ultra thick,color=myred] (state3s3s2) to (decision3s3s2);
    \draw[->,ultra thick,color=myred] (state3s3s3) to (decision3s3s3);
    \draw[->,ultra thick,color=myred] (decision3s1s1) to (state4s1s1s1);
    \draw[->,ultra thick,color=myred] (decision3s1s1) to (state4s1s1s2);
    \draw[->,ultra thick,color=myred] (decision3s1s1) to (state4s1s1s3);
    \draw[->,ultra thick,color=myred] (decision3s1s2) to (state4s1s2s1);
    \draw[->,ultra thick,color=myred] (decision3s1s2) to (state4s1s2s2);
    \draw[->,ultra thick,color=myred] (decision3s1s2) to (state4s1s2s3);
    \draw[->,ultra thick,color=myred] (decision3s1s3) to (state4s1s3s1);
    \draw[->,ultra thick,color=myred] (decision3s1s3) to (state4s1s3s2);
    \draw[->,ultra thick,color=myred] (decision3s1s3) to (state4s1s3s3);
    \draw[->,ultra thick,color=myred] (decision3s2s1) to (state4s2s1s1);
    \draw[->,ultra thick,color=myred] (decision3s2s1) to (state4s2s1s2);
    \draw[->,ultra thick,color=myred] (decision3s2s1) to (state4s2s1s3);
    \draw[->,ultra thick,color=myred] (decision3s2s2) to (state4s2s2s1);
    \draw[->,ultra thick,color=myred] (decision3s2s2) to (state4s2s2s2);
    \draw[->,ultra thick,color=myred] (decision3s2s2) to (state4s2s2s3);
    \draw[->,ultra thick,color=myred] (decision3s2s3) to (state4s2s3s1);
    \draw[->,ultra thick,color=myred] (decision3s2s3) to (state4s2s3s2);
    \draw[->,ultra thick,color=myred] (decision3s2s3) to (state4s2s3s3);
    \draw[->,ultra thick,color=myred] (decision3s3s1) to (state4s3s1s1);
    \draw[->,ultra thick,color=myred] (decision3s3s1) to (state4s3s1s2);
    \draw[->,ultra thick,color=myred] (decision3s3s1) to (state4s3s1s3);
    \draw[->,ultra thick,color=myred] (decision3s3s2) to (state4s3s2s1);
    \draw[->,ultra thick,color=myred] (decision3s3s2) to (state4s3s2s2);
    \draw[->,ultra thick,color=myred] (decision3s3s2) to (state4s3s2s3);
    \draw[->,ultra thick,color=myred] (decision3s3s3) to (state4s3s3s1);
    \draw[->,ultra thick,color=myred] (decision3s3s3) to (state4s3s3s2);
    \draw[->,ultra thick,color=myred] (decision3s3s3) to (state4s3s3s3);
\end{tikzpicture}
\caption{Schematic representation of single-sample and small-sample OSO versus scenario-tree representations, in a three-period example. Squares represent states, and circles represent decisions. In OSO, grey elements indicate sampled paths and the brown path represents realizations and implemented decisions.}
\label{fig:OSO_SAA}
\end{figure}

We define in~\ref{app:MSSIP_algs} the perfect-information benchmark along with two baseline algorithms. The first one is a myopic resolving heuristic, which optimizes decisions at each period without anticipating future uncertainty. The second one is a mean-based certainty-equivalent (CE) resolving heuristic, which replaces future uncertain parameters by their averages toward a deterministic approximation of the multi-stage stochastic optimization problem (\Cref{eq:mssip}) at each iteration.

\subsection{Theoretical Results: OSO Approximation Guarantees} \label{sec:theoryBacking}

We consider two canonical problems: online resource allocation and online batched bin packing. Each of these problems capture some of the core dynamics of online rack placement. Both admit feasible solutions with all algorithms. We provide worst-case guarantees of the solution of the single-sample OSO algorithm against a perfect-information benchmark. We refer to the perfect-information optimum as $\OPT$ in a given instance, and by its expected value as $\E{\OPT}$.

\subsubsection*{Online resource allocation.}
This problem assigns demand items to supply items, given multi-dimensional resource capacity constraints. It follows the three-layer resource allocation structure shown in \Cref{fig:structure} with demand nodes $i\in\calI$, supply nodes $j\in\calJ$, and resource nodes $k\in\calK$. As mentioned in Section~\ref{sec:formulation}, this formulation captures the online rack placement relaxation upon relaxing the multi-rack linking constraints. We assume that items arrive one at a time, so we treat the indices $i \in \calI$ and $t \in \calT$ interchangeably. The offline problem is formulated as follows, where $x^t_j$ indicates whether the demand item at time $t$ is assigned to supply $j$; for completeness, we formulate the multi-stage stochastic program in \ref{app:mssip_resourcealloc}.
\begin{equation}\label{eq:resourceallocation}
    \max \quad \left\{
    \sum_{t \in \calT}\sum_{j \in \calJ} r^t_{j} x^t_{j}
    \quad \bigg|\quad
    \sum_{j \in \calJ} x^t_{j} \leq 1,
    \forall \ t \in \calT
    ;\quad
    \sum_{t \in \calT}\sum_{j \in \calJ} A^t_{jk} x^t_{j} \leq b_k,
    \forall \ k \in \calK
    ;\quad
    \bx \text{ binary}
    \right\}
\end{equation}

This resource allocation structure includes canonical optimization problems as special cases:
\begin{itemize}
    \item[--] Multi-dimensional knapsack, when $|\calJ|=1$. Then, demand items can be accepted or rejected into a single supply bin, based on multiple capacity constraints.
    \item[--] Generalized assignment, when $\calK$ can be partitioned into $\Set{\calK_j | j \in \calJ }$, and each supply node is connected to its own resource nodes, i.e., $A^t_{jk} = 0$ if $k \notin \calK_j$. In rack placement, this structure arises from space restrictions, which apply a capacity for each tile group separately (in contrast, power and cooling restrictions give rise to a broader class of resource allocation problems due to a many-to-many mapping between supply and resource nodes). The generalized assignment problem can also model job-machine assignments subject to multiple capacity constraints per machine.
\end{itemize}

Our main result shows that single-sample OSO yields a $(1-\e_{d,T,B})$-approximation of the expected offline optimum of the online resource allocation problem (\Cref{thm:resourceallocation}). Specifically, $\varepsilon_{d,T,B}$ scales approximately as $\calO\left(\sqrt{\frac{\log(dT)}{B}}\right)$, where $d = |\calK|$ is the number of resources, $T$ is the time horizon, and $B$ is the tightest resource capacity normalized to resource requirements. This result shows a weak dependency in the dimensionality and the time horizon. Moreover, the approximation improves as resource capacities become larger relative to resource requirements; the larger the capacity, the less constraining initial assignments are for future items. In particular, OSO becomes asymptotically optimal if capacities scale as $B = \Omega(\log^{1+\sigma} T)$ for any constant $\sigma>0$ (Corollary~\ref{cor:aympopt}).

\begin{theorem} \label{thm:resourceallocation}
    Let $\varepsilon\in (0, 0.001]$ such that $b_k \ge 1024 \cdot \log (\frac{2 d T}{\e}) \cdot \frac{\log^3 (1/\e)}{\e^2} A^t_{jk}$ for all $j \in \calJ$, $k\in\calK$ and $t \in \calT$ and for every $\bA$ in the support of $\calD$. The single-sample OSO algorithm yields an expected value of at least $(1-\e) \cdot \E{\OPT}$ in the online resource allocation problem (\Cref{eq:resourceallocation}).
\end{theorem}
\begin{corollary}\label{cor:aympopt}
    Define $B=\min_{k\in\calK} \left\{ \frac{b_k}{\max_{t\in\calT,j\in\calJ} A^t_{jk}} \right\}$. If $B = \Omega(\log^{1+\sigma} T)$ for some constant $\sigma>0$, then single-sample OSO is asymptotically optimal in the online resource allocation problem.
\end{corollary}

The asymptotic regime is relevant in our rack placement problem. In practice, demand requests are handled in batches every few days, with up to a dozen requests per batch; at the same time, each data center can host tens of thousands of racks. Thus, the full rack placement problem from the start of the data center's operations to the point where it is full, involves a long planning horizon---hundreds to thousands of iterations. In addition, the size of demand batches is relatively consistent across data centers, so the planning horizon correlates with data center capacity---larger data centers take longer to fill. These observations motivate the asymptotic regime where $T\to\infty$ and where $B$ increases with the planning horizon (the condition $B = \Omega(\log^{1+\sigma} T)$, in fact, captures a weak dependency between data center capacity and the planning horizon).

The proof (in \ref{app:resourcealloc_proofs}) proceeds by (i) deriving the scaling of the expected offline optimum with the time horizon and resource capacities (\Cref{lemma:sampleOpt}), (ii) showing that the algorithm uses approximately a fraction $t/T$ of the budget after $t$ periods in expectation and with high probability (\Cref{lemma:S} and \Cref{lemma:concS}) and (iii) showing that each period contributes a reward of approximately $1/T\cdot \E{\OPT}$ (\Cref{lemma:val}). A key difficulty lies in the dependence between random variables in the problem; for instance, the incremental resource utilization at period $t$ depends on the utilization in periods $1,\dots,t-1$. This prevents the use of traditional concentration inequalities, so we prove a new concentration inequality for affine stochastic processes that may be of independent interest (\Cref{thm:centeringConc} in \ref{app:concentration}).

\subsubsection*{Online batched bin packing.} 

$n = Tq$ items in batches of $q$ items arrive over $T$ time periods. Each item $i \in \calI^t$ has size $V^t_i \in \{1, \dots, B\}$. The objective is to pack items in as few bins of capacity $B$ as possible. We use the flow-based formulation from \cite{valerio1999exact}, in \ref{app:binpacking}.

We study a slight variant of single-sample OSO in which all uncertain quantities are sampled at the beginning of the horizon---as opposed to being re-sampled in each period (see \Cref{alg:OSO_binpacking} in \ref{app:binpacking}). This change simplifies the proofs without impacting the overall methodology. 

\Cref{thm:binpacking} shows that the single-sample OSO algorithm yields a $\calO\left(\frac{n \log^{3/4} q}{\sqrt{q}}\right)$ regret for the batched bin packing problem, as long as batches are large enough. Notably, if $q = \Omega(n^\delta)$ for $\delta > 0$, OSO achieves sublinear regret. Compared to the online resource allocation setting, this result does not depend on the size of the jobs but depends on the number of jobs in each batch. The proof (in \ref{app:binpacking_proof}) leverages the monotone matching theorem from \cite{rhee1993lineii} to bound the cost difference between the number of bins opened when the decision at time $t=1,\dots,T$ is based on the true job sizes $V^t$ versus the sampled job sizes $\widetilde{V}^t$.

\begin{theorem}\label{thm:binpacking}	
If $\sqrt{q}\, (\log^{3/4} q)\, e^{c \cdot \log^{3/2} q} \ge n$ for a sufficiently small constant $c > 0$, single-sample OSO yields an expected cost of $\E{\OPT} +  \calO\left(\frac{n \log^{3/4} q}{\sqrt{q}}\right)$ in online batched bin packing.
\end{theorem}

\subsubsection*{Discussion.}

These results provide theoretical guarantees on the performance of the OSO algorithm. In online resource allocation, \Cref{thm:resourceallocation} yields a $(1-\varepsilon_{d,T,B})$-approximation, where $\varepsilon_{d,T,B}$ approximately scales with the number of resources $d$ as $\calO(\sqrt{\log d})$, the time horizon $T$ as $\calO(\sqrt{\log T})$ and resource capacities $B$ as $\calO(1/\sqrt{B})$; and it is asymptotically optimal when resource capacities scale with the planning horizon. This result applies to discrete and continuous probability distributions.  In online batched bin packing, \Cref{thm:binpacking} yields additive sublinear regret with large enough batches. This result is somewhat weaker than previous bounds in online (non-batched) bin packing, as \cite{gr2020} proved an additive $\calO(\sqrt{n})$ regret bound without distributional knowledge, and \cite{banerjee2024good} derived a constant additive regret under suitable distributional assumptions. Still, the theoretical guarantees demonstrate the performance of the simple and generalizable OSO algorithm across a broad class of multi-stage stochastic integer optimization problems, and highlight the role of batching in the performance of the OSO algorithm.

To shed further light on these insights, \Cref{prop:myopic} shows that the single sample path at the core of the algorithm plays a critical role in managing resources for future demand in online decision-making. In comparison, myopic decision rules can lead to arbitrarily poor performance.

\begin{proposition}\label{prop:myopic}
	Single-sample OSO can yield unbounded benefits vs. the myopic policy.
\end{proposition}

More surprisingly, the sampling-based approach in OSO can also provide unbounded benefits versus mean-based certainty-equivalent resolving heuristics, as shown in \Cref{prop:CE}. The proof constructs multidimensional knapsack instances with a discrete distribution of resource requirements, in which each item has value 1 and consumes one unit of at least one resource, while the average item consumes strictly less than one unit of all resources. The CE benchmark rejects all incoming items in favor of future average items (because accepting an item means rejecting more than one ``average item'' in the future), until the end of the horizon when it is forced to accept all remaining items---including ``sub-optimal'' items with high resource requirements. In contrast, by sampling future unit-sized items, OSO can favor the incoming item. We show that the difference can become arbitrarily large as the number of resources grows infinitely large. It is important to note that this result compares the OSO algorithm to the mean-based CE equivalent based on the online resource allocation formulation in \ref{app:mssip_resourcealloc}, whereas other CE resolving heuristics have been designed in the literature for specific classes of problems \citep[see, e.g.,][]{gallego1994optimal,vera2021bayesian,balseiro2024survey}.

\begin{proposition}\label{prop:CE}
	Single-sample OSO can yield unbounded benefits vs. mean-based certainty-equivalent resolving heuristics.
\end{proposition}

In summary, OSO provides a simple, easily-implementable and generalizable approach to multi-stage stochastic optimization. We proved performance guarantees of single-sample OSO in generic online resource allocation and online batched bin packing settings, and showed that it can also outperform myopic decision-making and mean-based resolving heuristics. These results underscore the potential of even a single sample path when combined with online re-optimization.

\subsection{Computational Assessment}\label{ssec:computation}

\subsubsection*{Online resource allocation.}
Results in~\ref{app:resourcealloc_comp_exact} first establish that multi-stage stochastic programming and dynamic programming methods do not scale to even moderately-sized instances of the problem. Viewed as a stochastic program, the problem involves a scenario tree that grows exponentially with the time horizon and the number of resources, and features discrete decision variables at each node. Stochastic programming models with scenario-tree representations become intractable with as few as 12 time periods, 1 supply node, 2 resources, and binary uncertainty. Dynamic programming algorithms also remain several orders of magnitude slower than OSO and time out in moderate instances (e.g., 20 time periods, 1 supply node, and 6 resources), due to the exponential growth in the state space as $\calO(T (2B)^d)$. In comparison, we tackle much larger instances in this paper, with up to 100 time periods, 10 supply nodes, 20 resources, and continuous uncertainty. Performance could be improved via approximate dynamic programming and reinforcement learning \citep{sutton2018reinforcement,powell2022reinforcement}; yet, these results are indicative of very high-dimensional multi-stage stochastic optimization problems for which exact methods feature limited scalability---thus motivating the need for resolving heuristics.

We then assess the OSO algorithm against the myopic and certainty-equivalent (CE) benchmarks. We implement both the single-sample and small-sample variants of the algorithm ($S=1$ and $S=5$), with no regularizer ($\Psi(\cdot)=0$). We consider problems with unit rewards, scaled resource capacities by a parameter $B$, and unknown resource consumption. Specifically, we define (i) a multi-dimensional knapsack problem with $T = 50$ items, $|\calK| = 10$ and $b_k = T B$; (ii) an online generalized assignment problem with $T = 50$ items, $|\calJ| = 10$, $|\calK| = 50$, and $b_k = T B / |\calJ|$; and (iii) an online resource allocation with $T = 100$ items, $|\calJ| = 10$, and $3|\calJ|/2$ resources overall, such that $|\calJ|$ resources are consumed by a single supply node and have capacity $b_k = T B / |\calJ|$, $|\calJ| / 2$ resources are consumed by two supply nodes and have capacity $b_k=2 T B / |\calJ|$, and 5 resources are consumed by half of the supply nodes and have capacity $b_k= T B / 2$. When supply node $j$ consumes resource $k$, the parameter $A^t_{jk}$ is modeled via a bimodal distribution; specifically, $A^t_{jk}$ follows a triangular distribution with width $0.5$ centered in $0.5 - \psi$ with probability 0.5, and centered in $0.5 + \psi$ with probability 0.5 (\Cref{fig:bimodal_pdf}). Thus, the problem is governed by the capacity parameter $B$ and the extent of bimodality $\psi$. For each problem, we set parameter to ensure overall demand-supply balance; for each combination of parameters, we generate five instances and, for each one, we run OSO five times. Full computational details and results are in \ref{app:resourcealloc_comp}.

\Cref{fig:resource_alloc_groupedbar_obj_time} reports the proportion of accepted items and computational times. We first observe that the myopic policy can induce a loss of up to 50\% as compared to the perfect-information solution, reflecting the cost of uncertainty regarding future arrivals. By accounting for future demand, the CE benchmark improves upon the myopic solution, but the benefits remain rather limited (3--22\% improvements, leading to solutions within 60 to 85\% of the hindsight optimum). In comparison, OSO generates significant performance improvements, and can yield high-quality solutions against the---unattainable---perfect-information solution. Quantitatively, even the single sample OSO algorithm further improves upon the CE solution by 3--38\% and falls within 84--94\% of the perfect-information benchmark (\Cref{tab:resourcealloc_app} in \ref{app:resourcealloc_comp}). With $S=5$, OSO can further improve the solution within 88--95\% of the perfect-information benchmark. The OSO algorithm involves longer computational times but remains a tractable approximation approach. Notably, these results confirm that even the single-sample approximation of uncertainty combined with online re-optimization consistently yields high-quality solutions within the time limit.

\begin{figure}[h!]
\centering
\subfloat[\small Objective]{\includegraphics[width=0.49\textwidth]{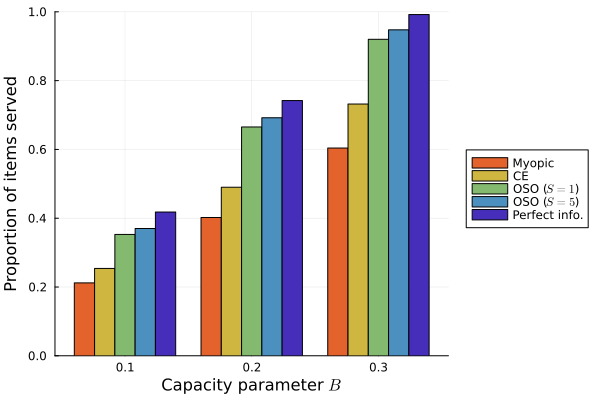}
\label{fig:resource_alloc_groupedbar_obj}
}%
\subfloat[\small Computation time]{\includegraphics[width=0.49\textwidth]{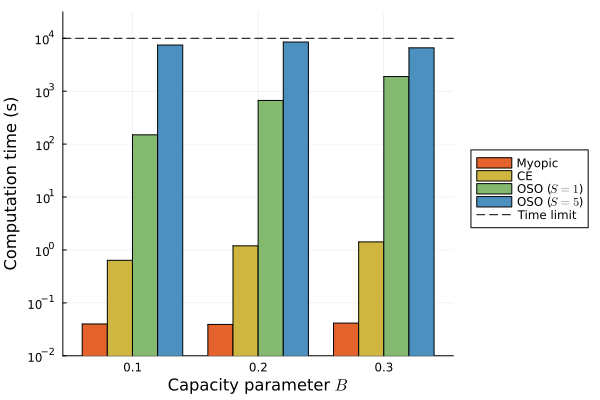}
\label{fig:resource_alloc_groupedbar_time}
}%
\caption{Normalized objectives and computation times for the online resource allocation problem ($\psi=1/4$).} 
\label{fig:resource_alloc_groupedbar_obj_time}
\vspace{-12pt}
\end{figure}

To shed further light into this comparison, \Cref{fig:OSO_CE_heatmap_objective} shows the percent-wise improvement of the single-sample OSO solution from the CE solution. The heatmaps reveal that OSO outperforms CE across virtually all instances. The relative differences can be significant, with benefits of up to 23\% in multi-dimensional knapsack, 19\% in generalized assignment, and 39\% in the general resource allocation problem. Moreover, the heatmaps indicate that OSO tends to provide stronger benefits as the probability distribution of unknown parameters becomes more bimodal---that is, when the mean is less representative of the distribution---and when the capacity parameter is not too large---that is, when poor-quality decisions have a stronger impact down the road.

\begin{figure}[h!]
\centering
\subfloat[\small Multi-dimensional knapsack]{\includegraphics[width=0.365\textwidth]{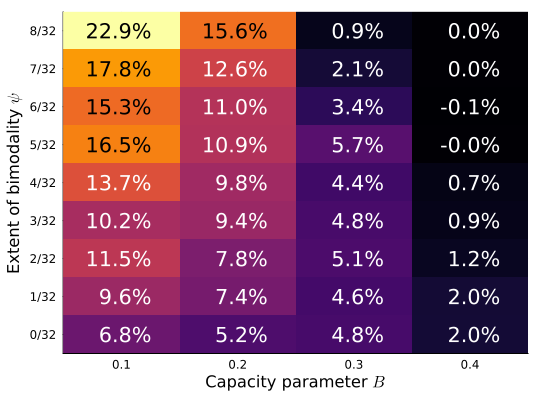}
\label{fig:assignment_knapsack_heatmap_OSO_CE_objective}
}%
\subfloat[\small Generalized assignment]{\includegraphics[width=0.365\textwidth]{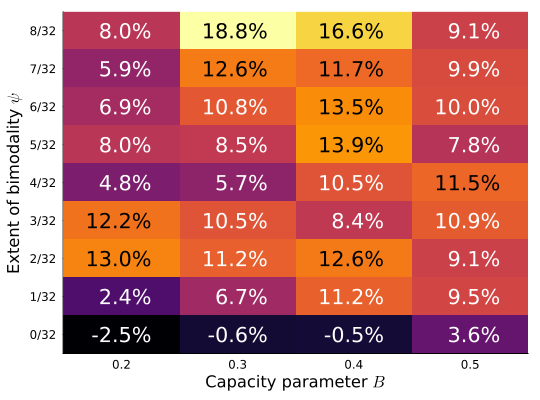}
\label{fig:assignment_fixedC_heatmap_OSO_CE_objective}
}%
\subfloat[\small Resource allocation]{\includegraphics[width=0.27\textwidth]{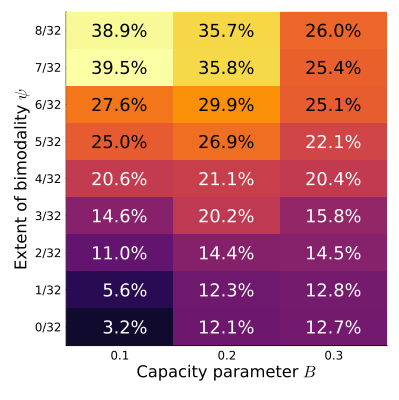}
\label{fig:resource_alloc_heatmap_OSO_CE_objective}
}%
\caption{Relative average improvement of single-sample OSO vs. CE in online resource allocation (positive numbers indicate instances where single-sample OSO outperforms CE on average).} \label{fig:OSO_CE_heatmap_objective}
\vspace{-12pt}
\end{figure}

\subsubsection*{Online batched bin packing.}
We report in \ref{app:binpacking_comp} similar results for the online batched bin packing problem. These results confirm that OSO generates higher-quality solutions than the myopic and CE benchmarks, even with a few samples (1 to 5), in longer but manageable computational times. Moreover, these results show the impact of the regularizer---in that case, the regularizer promotes packing incoming items on fuller bins to leave space for future batches.

\subsubsection*{Online rack placement.}

We define two variants of the problem: the core problem with an online resource allocation structure and discrete linking constraints (Equations~\eqref{eq:rackplacement}--\eqref{eq:rackplacement_domain}); and a variant with ``precedence'' constraints stating that requests cannot be rejected to leave space for future racks. The former is closer to online resource allocation, whereas the latter is closer to the one deployed in production (in practice, data centers must accept requests if possible). Without precedence constraints, we measure performance as the number of accepted requests until the end of the horizon; with precedence constraints, we measure performance as the number of accepted requests until the data center rejects an incoming item.

We consider a data center with two rooms, each with 36 rows, 4 top-level UPS devices, 6 PDU devices per UPS device and 3 PSU devices per PDU device. We use real-world data to simulate incoming demand in data centers, based on the historical distributions of request sizes, power requirements, and cooling requirements (\Cref{F:historicalDemand}). We provide details on the setup in \ref{app:rackplacement_comp_setup}.

Results are reported in \Cref{fig:rack_placement_OSO}. Both the CE and OSO algorithms provide strong performance improvements as compared to myopic decision-making. Whereas the myopic solution ranges from 75\% to 90\% of the hindsight-optimal solution, both resolving methods achieve 90\% to 99\% of the benchmark. Then, the OSO algorithm yields additional gains as compared to the CE solution, although by a smaller margin than in the general online resource allocation setting. Quantitatively, the benefits are estimated at up to 1\% with precedence constraints and up to 4\% without precedence constraints. The difference is strongest with small batches and no precedence constraints, because smaller batches exacerbate the differences between mean-based and sampling-based approximations, and because the variant without precedence constraints provides most flexibility to the model. More broadly, the results confirm our earlier insights, both regarding the large improvements of OSO over the myopic benchmark and the added benefits of OSO over the CE heuristic.

\begin{figure}[h!]
\centering
\subfloat[No precedence constraints]{\includegraphics[width=0.49\textwidth]{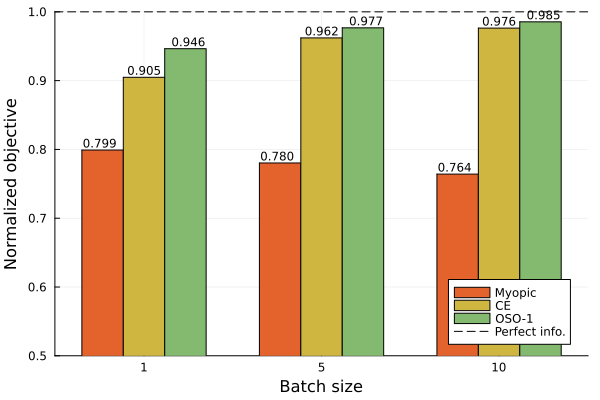}
\label{fig:rack_placement_obj_oldDC_oldDist_noCool}
}%
\subfloat[Precedence constraints]{\includegraphics[width=0.49\textwidth]{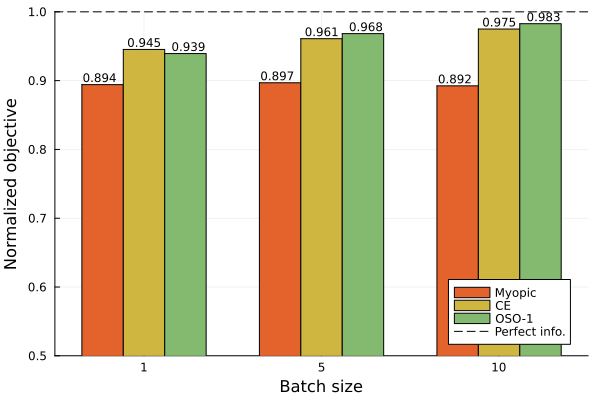}
\label{fig:rack_placement_obj_prec_noafter_oldDC_oldDist_noCool}
}%
\caption{Performance of the OSO algorithm and the myopic and CE benchmarks on the online rack placement problem, with and without precedence constraints.} \label{fig:rack_placement_OSO}
\vspace{-12pt}
\end{figure}
\section{Real-world Deployment in Microsoft Data Centers}
\label{sec:real}

In practice, rack placement decisions used to rely on the expertise of data center managers, assisted by spreadsheet tools and feasibility-oriented software. Cloud computing growth has rendered these decisions increasingly complex, creating interdependent considerations and conflicting objectives. To alleviate mental loads and operational inefficiencies, we have deployed our rack placement solution across Microsoft's fleet of data centers. The goal was to combine the strength of optimization and human expertise, by building a decision-support tool but leaving decision-making authority to data center managers. We have extensively collaborated with stakeholder groups to gradually deploy the solution across the organization, and modified the model to capture practical considerations. The software's recommendations have been increasingly adopted by data center managers. The success of this deployment underscores the impact of human-machine interactions in cloud supply chains to turn an optimization prototype into a full-scale software solution in production.

\subsection{Solution Deployment in Microsoft's Fleet of Data Centers} \label{sec:deployment-milestones}

\subsubsection*{Software development.}

We packaged our algorithm into a software tool that could be embedded into the production ecosystem. We built data pipelines to get access to real-time information on incoming demand and data center configurations. Each demand batch triggers our optimization algorithm to generate placement suggestions. To ensure acceptable wait times, we imposed a four-minute time limit for each optimization run---strengthening the need for our single-sample OSO algorithm as compared to more complex multi-stage stochastic optimization algorithms. We developed a user interface enabling data center managers to visualize placement suggestions in the data center (\Cref{fig:interface}). For each request, data center managers can either accept the placement suggestion or reject it. The suggested and final placements are both recorded.

\begin{figure}[h!]
	\centering
	\subfloat[Suggested placement of an incoming request (yellow).]{\label{fig:interface}\includegraphics[width=.65\textwidth]{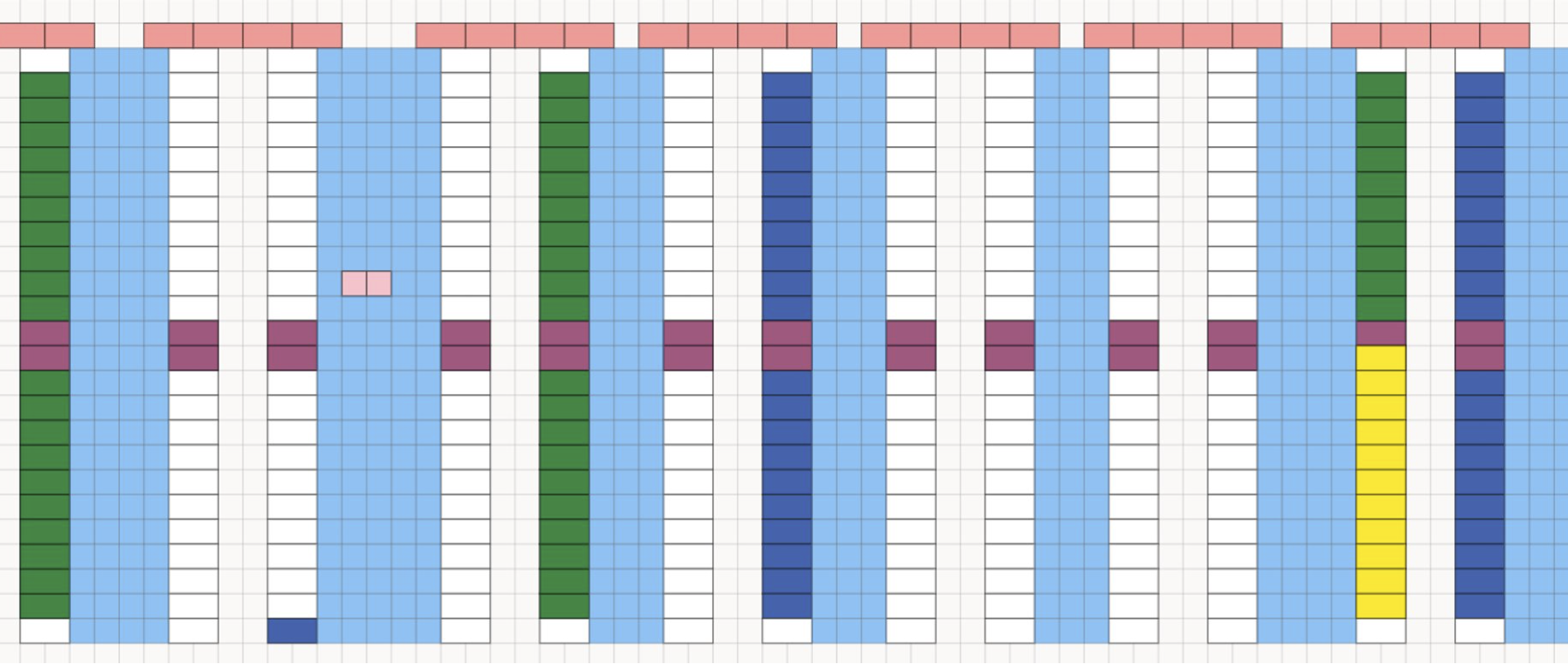}}
    \hspace{12pt} 
	\subfloat[Feedback form.]{\label{fig:feedback-ui}\includegraphics[width=.22\textwidth]{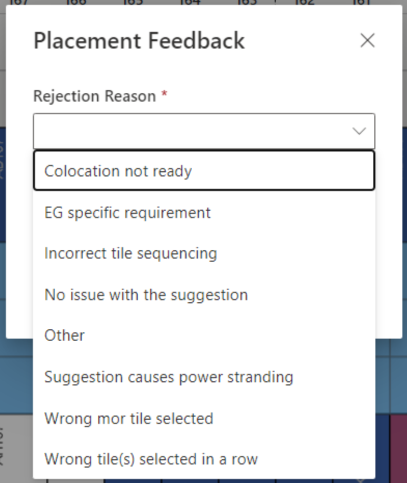}}
	\caption{User interface for data center managers at the core of our solution deployment.}\label{fig:deployment}
    \vspace{-12pt}
\end{figure}

\subsubsection*{Pilot.}

After extensive simulations and testing, we initiated a small-scale pilot in 13 data centers. This phase started at the end of 2021 and lasted three months. We fostered direct interactions with data center managers to assess the new solution. Initial feedback helped us identify issues in the data pipelines. For instance, early deployments failed to record that some rooms were unavailable, that some rows were already reserved, and that some requests came with placement restrictions.

To continue gathering feedback, we augmented the user interface for data center managers to specify the reason for each placement rejection (\Cref{fig:feedback-ui}). This new deployment phase provided valuable insights on real-time adoption. We devised a principled approach to analyze feedback by grouping the main rejection reasons into (i) engineering group requirements; (ii) power balancing considerations; (iii) conflicts with other requests (``already reserved''); (iv) availability of better placements by throttling lower-priority demands (``multi-availability''); and (v) opportunities to utilize small pockets of space (``better space packing''). These options were supplemented with a free-text ``Other'' category for ad hoc requests, hardware compatibility issues, and software bugs.

\subsubsection*{Modeling modifications.}

Any optimization model of supply chain operations builds a necessarily simplified representation of reality. The feedback gathered through the pilot deployment identified the most critical limitations of the initial model. We have performed iterative modeling adjustments by adding secondary objectives for tie-breaking purposes, including: (i) room minimization and row minimization objectives to mitigate operational overhead for data center managers; (ii) a tile group minimization objective to place racks from the same request closer together for better customer service; and (iii) power surplus and power balance objectives to mitigate the risk of overload and device failure. Details on these adjustments are provided in~\ref{app:rackplacement_production_objectives}.

\subsubsection*{Full-scale deployment.}

We organized information sessions to familiarize data center managers with the new system and demonstrate its capabilities. These sessions led to strong engagement on the details of the model. Within a month, our solution was launched globally across Microsoft's global fleet of data centers. Throughout, we performed modeling adjustments to improve the quality of rack placement recommendations, using data on the rejection reasons (\Cref{fig:feedback-ui}).

\subsection{Impact and Adoption} \label{sec:adoption-impact}

\Cref{fig:rejections} reports the main rejection reasons in the last quarter of 2022 and the second quarter of 2023. Once our modeling adjustments got implemented, tested and deployed in production in early 2023, the incidence of rejections due to engineering group requirements decreased from over 40\% to less than 20\%, while at the same time total rejections went down as well. The remaining engineering group requirements are mostly driven by one rack type, which currently lies out of scope of the model. Thus, the iterative modeling adjustments enabled to address the main issues and increase overall adoption of the rack placement solution.

\begin{figure}[h!]
	\centering
	\subfloat[Q4 2022.]{\label{fig:rejections-before}\includegraphics[width=.495\textwidth]{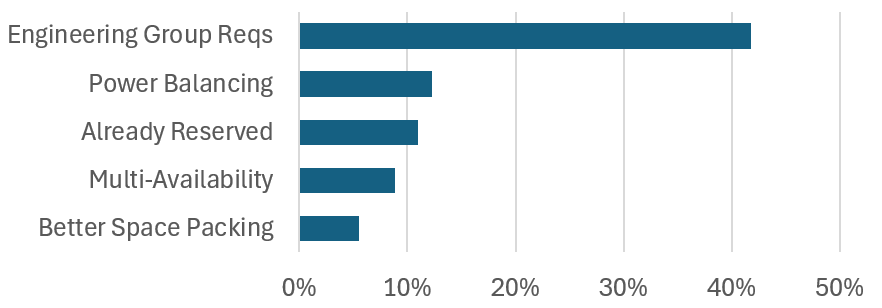}}
	\subfloat[Q2 2023.]{\label{fig:rejections-after}\includegraphics[width=.495\textwidth]{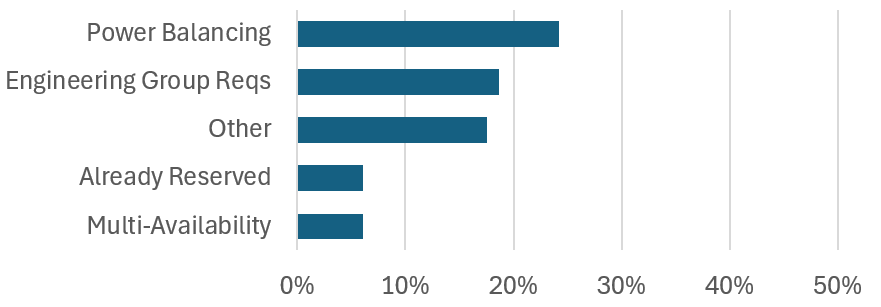}}
	\caption{Rejection reasons before and after incorporating engineering group requirements in the optimization.}
    \label{fig:rejections}
    \vspace{-12pt}
\end{figure}

\Cref{fig:acceptance} reports the proportion of recommendations accepted by data center managers across all Microsoft data centers in April--July 2023. During this period, we made two significant improvements. In May 2023, we incorporated modeling adjustments to capture preferences from engineering groups. In June 2023, we augmented our solution to support a particular data center architecture (Flex) that can throttle low-priority demands in case of failover \citep{zhang2021flex}.

\begin{figure}[h!]
	\centering
	\subfloat[Accepted and rejected placements.]{\label{fig:histograms-acceptance}\includegraphics[width=.4\textwidth]{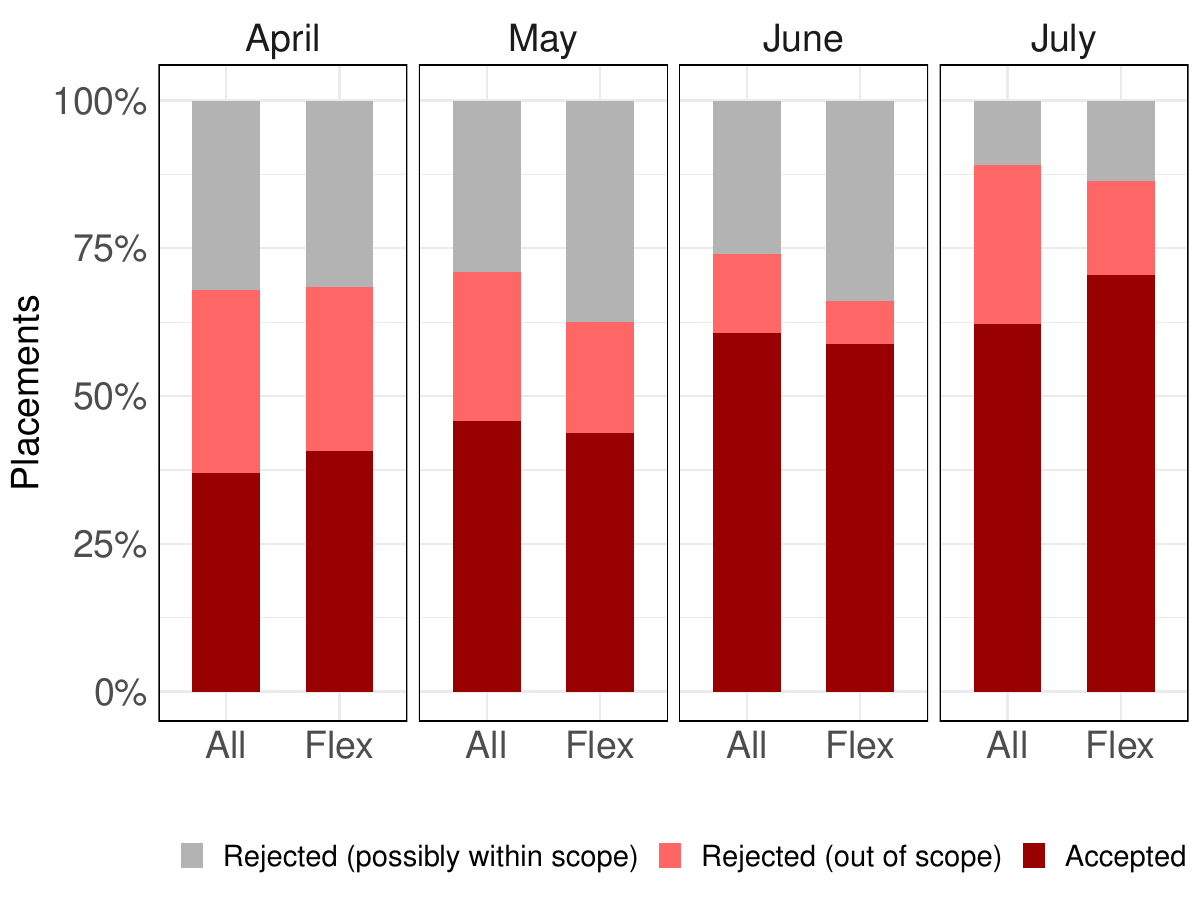}}
 \hspace{0.1cm}
	\subfloat[Accepted location/row/room.]{\label{fig:correct-matches}\includegraphics[width=.4\textwidth]{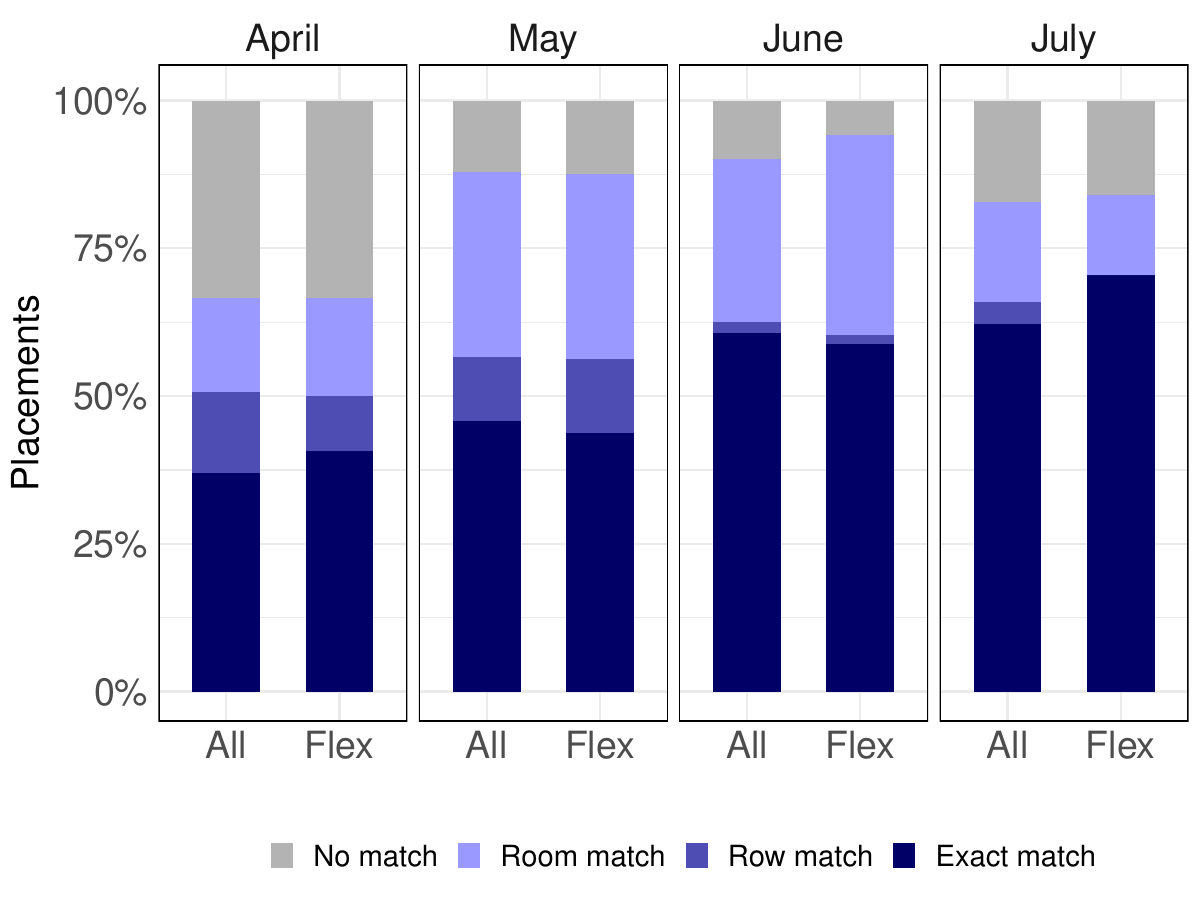}}
	\caption{Acceptance of our recommendations by the data center managers.}
    \label{fig:acceptance}
    \vspace{-12pt}
\end{figure}

\Cref{fig:histograms-acceptance} shows a strong increase in accepted requests between April 2023 and July 2023. The deployment of an optimization solution does not necessarily lead to immediate large-scale impact. Rather, adoption increases over time as users get progressively more familiar with it, and as it gets improved to capture practical requirements. In our case, the acceptance ratio increased from 35\% in April 2023 to over 60\% in July 2023 across all Microsoft data centers---and to over 70\% among the Flex data centers in particular. In fact, \Cref{fig:correct-matches} shows that, even when the data center managers rejected the specific recommendation, many placements remained in the same room and the same row. In particular, the room was accepted for 80--90\% of placements.

\Cref{fig:histograms-acceptance} also breaks down rejected requests into those out of the scope of the optimization model and those possibly within scope. Out-of-scope rejections are primarily due to data issues and bugs (indicated, for instance, in the ``already reserved'' and ``other'' categories of the feedback form). Such rejections could be addressed over time as data pipelines mature. The remaining rejections include requests for which the solution provided an appropriate recommendation but the data center managers decided for other placements. Such rejections would require deeper changes to the optimization model. This breakdown suggests that the vast majority of rejections at the end of the deployment period fell out of scope, suggesting that our iterative improvements were successful at addressing the main in-scope issues. Ultimately, when disregarding out-of-scope rejections, the potential of our optimization solution reaches 80-90\% of requests across all data centers.

\subsubsection*{Takeaways.} 
Deploying an optimization algorithm at the scale of Microsoft's cloud supply chains involves a number of technical and practical challenges. The first one is \emph{solving the right problem}. In practice, there is no ``clean'' problem description outlining objectives and constraints that can be easily translated into an optimization framework. We devoted significant time and effort to understand the rack placement process and adjust the model to meet practical requirements and preferences, in close collaboration with stakeholders (e.g., data center managers, program managers, engineering groups). The second one is \textit{data challenges}. To overcome inconsistencies between databases, we had to build dedicated pipelines into our optimization model. We also developed user interfaces to make the optimization solutions available to decision-makers in real time and to collect data on adoption (\Cref{fig:deployment}). A third challenge is \emph{human factors}: to replace an existing (mostly manual) system with a sophisticated optimization approach, it was critical to involve data center managers early on in the process and gain their trust. In fact, this collaboration was a two-way street. On the one hand, it allowed us to leverage their expertise and feedback to strengthen the optimization solution. At the same time, this enabled them to better understand the logic behind the new rack placement system, thereby alleviating the pitfalls of ``black-box'' optimization. The working sessions were particularly useful to make the model-based recommendation more interpretable and transparent. Ultimately, our full-scale optimization deployment highlights the importance of keeping the user's perspective in mind when designing real-world optimization solutions through cross-organizational collaborations at the human-machine interface.
\section{Empirical Assessment and Impact}
\label{sec:empirical}

We leverage post-deployment data to identify the impact of our solution on data center performance. To this end, we construct econometric specifications measuring the effect of adoption on power stranding (defined formally below), while controlling for possible confounding factors. The results from this section show that our solution can contribute to higher power utilization within data centers, resulting in joint financial and environmental benefits.

\subsubsection*{Unit of analysis.}
Our empirical analysis is complicated by the need to compare performance metrics at the aggregate level---i.e., at the level of a data center over the entire deployment period---as opposed to a disaggregated level---e.g., at the individual rack level. This is driven by the combinatorial complexities of the rack placement problem. Indeed, our algorithm is designed to generate a strong data center configuration over time, that is, over multiple rack placements. We therefore cannot measure its impact one rack at a time; rather, we need to wait for an extended period of time to measure the impact of high- and low-quality decisions on the data center configuration.

For example, consider a data center with low initial utilization. Assume that the next racks are placed based on ``poor-quality decisions’’. These decisions might not result in an \emph{immediate} deterioration in data center performance. However, they might leave limited resources for future rack placements, increasing the risk of fragmentation or resource unavailability (\Cref{fig:stranding}). When other racks get placed down the road, performance might deteriorate regardless of whether these decisions are ``good’’ or ``bad’’. Thus, performance deterioration might be unjustly attributed to the later decisions whereas they would come, in fact, from poor-quality earlier decisions. This example underscores interdependencies across rack placements, which require an empirical analysis at the aggregated data center level rather than a disaggregated rack level.

This aggregation restricts the number of observations. This departs from other empirical contexts, in which treatment is applied on a small cross-section but still impacts a large number of disaggregated observations. For example, \cite{bray2016multitasking} tested the impact of task juggling on six judges but observed hundreds of thousands of (independent) hearings; \cite{stamatopoulos2021effects} tested the impact of electronic shelf labels on two stores but observed hundreds of store-date observations; \cite{cohen2023managing} tested the impact of a mark-up strategy in airline pricing on 11 origin-destination markets but observed hundreds of market-week observations. In contrast, our unit of analysis remains the data center at the aggregate level. Still, we exploit Microsoft's large fleet to identify the effect of adoption on data center performance with statistical significance.

We also stress that this aggregated level of analysis makes it challenging to run a field experiment---a common challenge in system-wide interventions. A field experiment would entail delaying deployment in some data centers by one year or so, which was deemed impractical. Instead, we leverage post-deployment observational data to identify the impact of adoption on performance.

\subsubsection*{Data.}

We have access to a monthly report for each of Microsoft's data centers between October 2022 and September 2023. We aggregate all metrics per data center over the 11-month deployment period. The key performance metric is \textit{power stranding}, defined as the ratio of unusable power to the top-level power capacity of the data center (i.e., the relative slack in \Cref{eq:rackplacement_power} when the data center fills up, summed over all top-level UPS devices, which are the main bottlenecks). In practice, power stranding is ultimately observed at the end of the data center's lifecycle---when the data center is full. However, we use power stranding measurements from Microsoft's engineering teams available in the monthly report, which capture the power that has already been stranded even if the data center can still accommodate demands. For instance, if a power device is no longer connected to any available tile, any residual power is classified as stranded. Note that as more racks get placed within a data center, utilization generally increases---except for decommissioning and other minor events---and power stranding may increase as well.

It is difficult to isolate the impact of our solution, as power stranding depends on the data center configuration (determined by our solution) but also on data centers' broader operations. We use the acceptance ratio as an estimate of the prevalence of our algorithm vs. human decision-making in the data center, and control for data center characteristics. Our main hypothesis is that higher reliance on our algorithm’s recommendations (measured via higher acceptance) leads to stronger performance (measured via a smaller increase in power stranding over the deployment period).

\subsubsection*{Variables.}

\paragraph{Treatment variable:} adoption rate, a continuous variable between 0 and 1 defined as the ratio of the number of demands placed per the algorithm’s recommendations over the number of demands placed during the deployment period. A higher adoption rate indicates a higher reliance on our solution by data center managers.

\paragraph{Outcome variable:} increase in power stranding over the deployment period. New rack placements represent the main changes in data center configurations, leading to an increase in power stranding. Still, occasional rack decommissioning and minor events outside the scope of our problem can lead to episodic decreases in power stranding (Section~\ref{sec:model}). Since our algorithmic solution only impacts new rack placements, we define an aggregate metric of increase in power stranding by isolating month-over-month gains. With $\widetilde{y}_{it}$ denoting the power stranding in data center $i$ at the end of month $t$, obtained from our data, we define the outcome variable in data center $i$ as $y_i=\sum_{t=1}^T(\widetilde{y}_{it}-\widetilde{y}_{i,t-1})^+$.

\paragraph{Control variables:} We define seven control variables to capture characteristics of data centers:
\begin{enumerate}
	\item IT capacity: available power capacity within data center $i$, denoted by $x^C_i$ and measured in kW. We use a scaled version of this variable for confidentiality purposes, i.e., $x^C_i/\max_\ell x^C_\ell$.
	\item Demand: demand for new racks during the deployment period, in percentage of power capacity. As for power stranding, we aggregate utilization data to isolate the month-over-month increases, and disregard episodic reductions caused by out-of-scope events. Let $\widetilde{x}^U_{it}$ denote rack utilization, in kW, in data center $i$ at the end of month $t$; we define the demand variable $x^D_i$ as $x^D_i=\frac{1}{x^C_i}\sum_{t=1}^T(\widetilde{x}^U_{it}-\widetilde{x}^U_{i,t-1})^+$.
	\item Initial utilization: relative occupancy of data center $i$ at the start of the deployment period, denoted by $x^0_i$ and measured as the percentage of power utilization. It is given by $x^0_i=\widetilde{x}^U_{i0}/x^C_i$. This variable differentiates younger and more empty data centers from older and fuller ones.
	\item Initial power stranding: power stranding at the start of the deployment period in data center $i$, denoted by $x^S_i$. This variable captures the historical performance of the data center.
	\item Rooms: number of rooms within data center $i$, denoted by $x^R_i$.
	\item ``Flex'':  binary variable $x^F_i$ indicating whether data center $i$ has the \emph{Flex} architecture.
	\item Location-US: binary variable $x^{US}_i$ indicating whether data center $i$ is in the United States.
\end{enumerate}

\subsubsection*{Raw statistics and model-free evidence.}

We first report statistics on the control and outcome variables, disaggregated between high- and low-adoption data centers. We classify data centers into the high-adoption category if their adoption rate exceeds a threshold of 60\%, corresponding to the 75\textsuperscript{th} percentile of the distribution. For robustness, we replicate the analyses with a threshold of 45\%, corresponding to the 50\textsuperscript{th} percentile of the distribution.

\Cref{fig:controls} plots the distribution of the four continuous control variables. The figure suggests that the distributions are rather balanced between the high- and low-adoption groups, with exceptions of the long tails with high demand and low initial utilization among low-adoption data centers. This visualization suggests that adoption is not driven by underlying control variables but occurs independently from demand, utilization, power stranding, and capacity. In contrast, \Cref{fig:treatment} shows that the distribution of the outcome variable clearly shifts to the left among high-adoption data centers, reflecting a lower increase in power stranding during the deployment period. On average, high-adoption data centers face a smaller increase in power stranding than low-adoption ones (+1.03\% vs +3.02\% with a 60\% threshold, and +1.65\% vs. +3.38\% with a 45\% threshold). This corresponds to a reduction in power stranding by 1.73--1.99 percentage points in absolute terms, or by 105--193\% in relative terms. 

\begin{figure}[h!]
\centering
\subfloat[Demand.]{\includegraphics[width=0.4\textwidth]{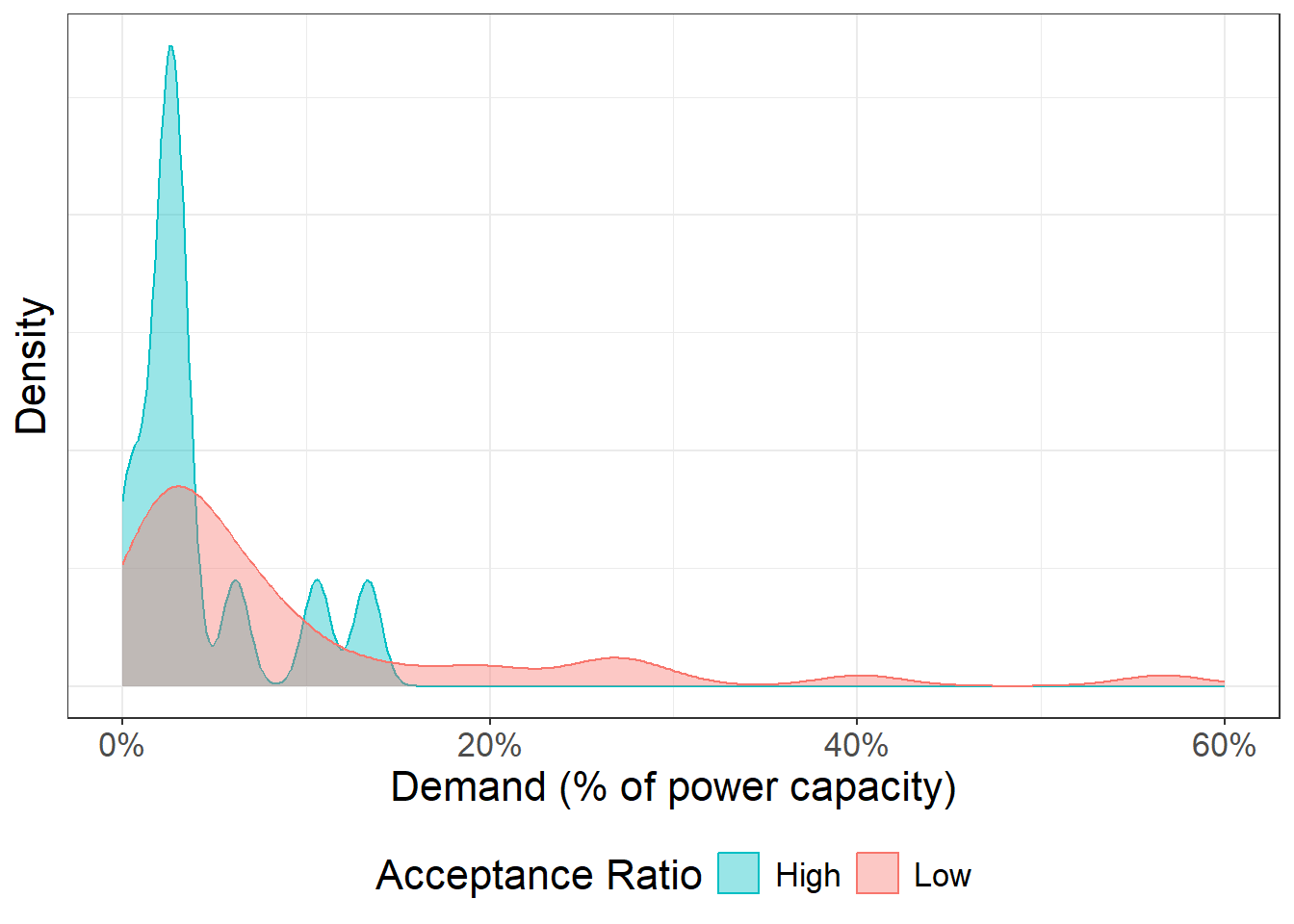}} 
\hspace{0.1 cm}
\subfloat[Initial utilization.]{\includegraphics[width=0.4\textwidth]{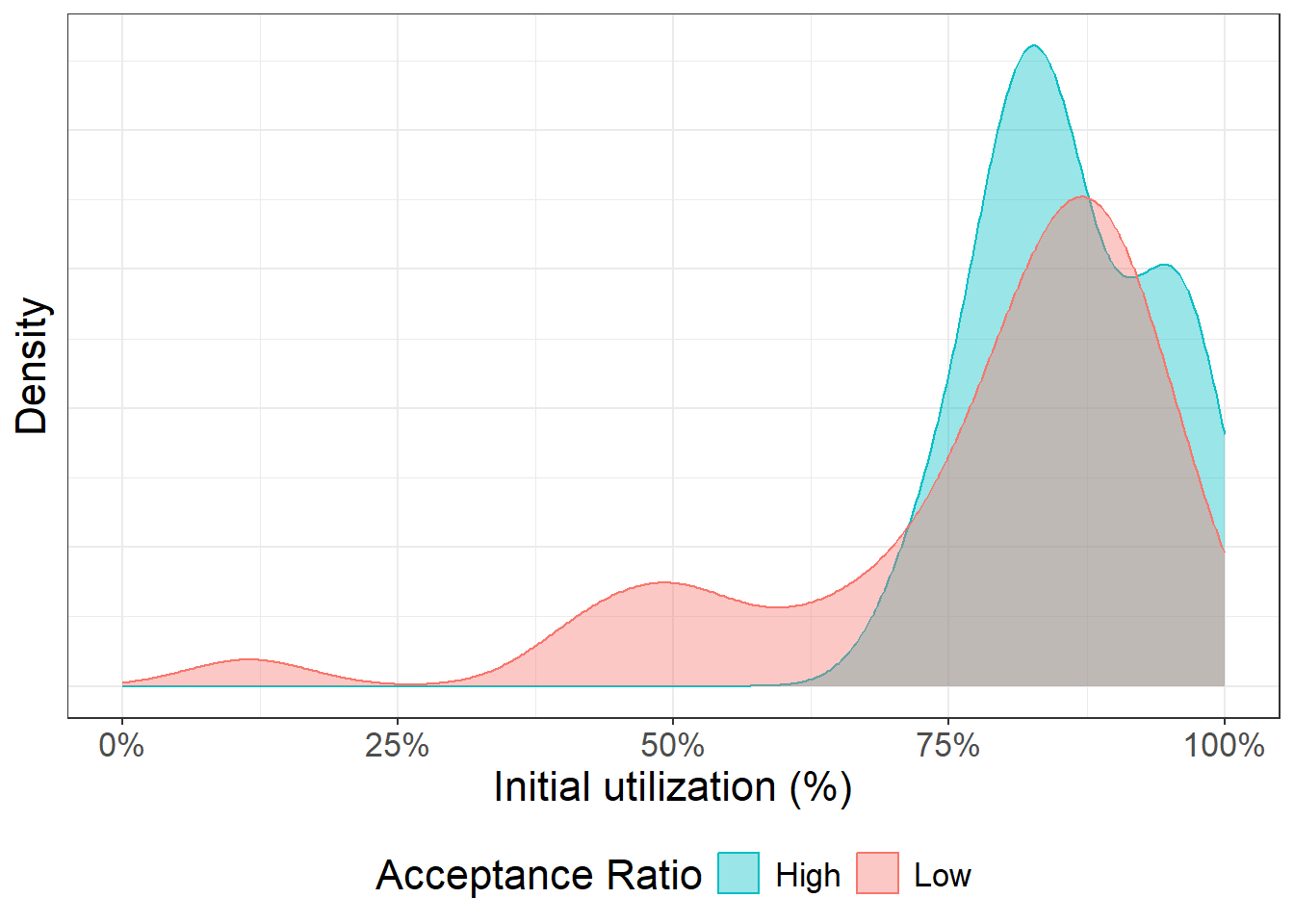}} 
\hspace{0.1 cm}
\subfloat[Initial power stranding.]{\includegraphics[width=0.4\textwidth]{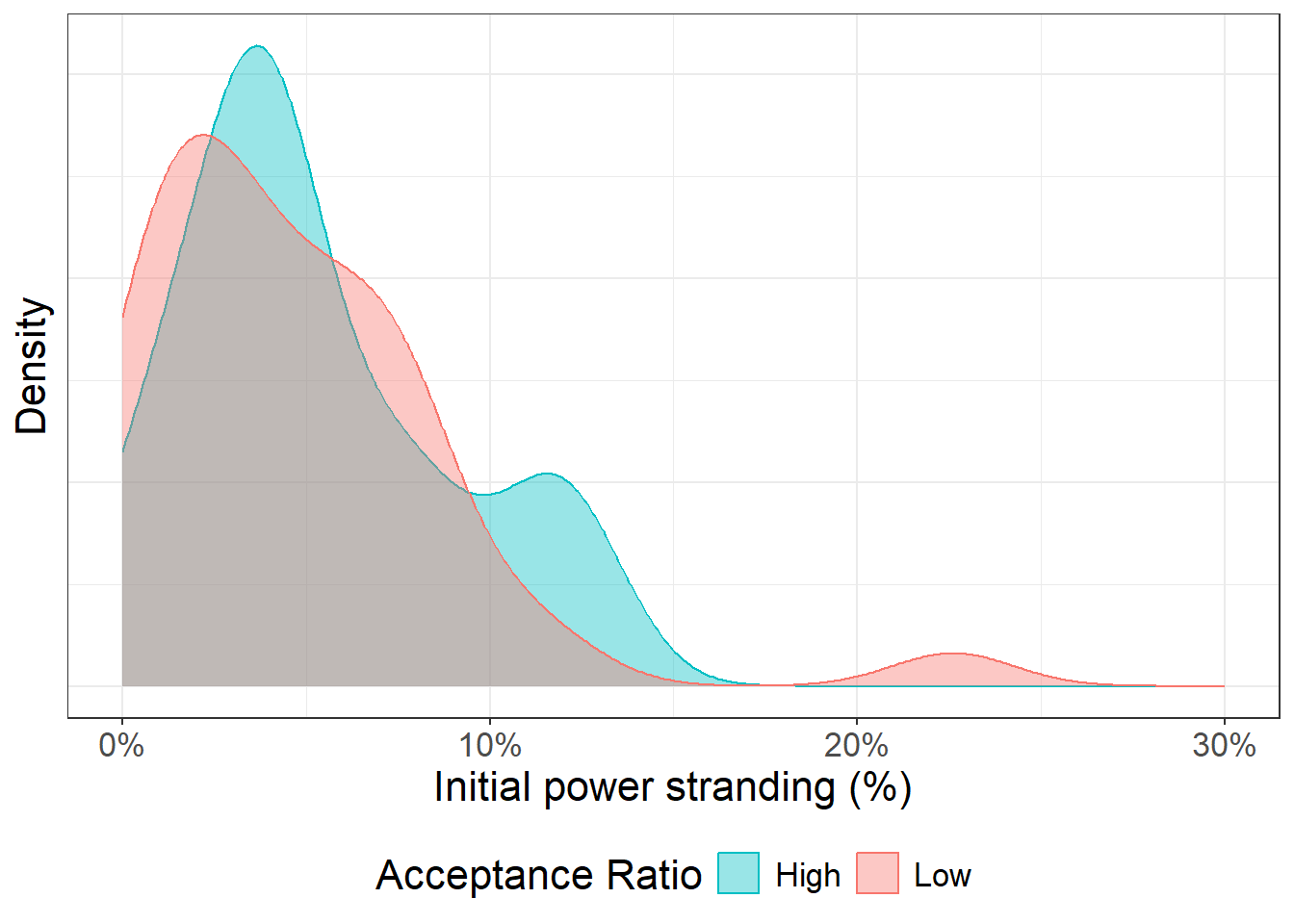}} 
\hspace{0.1 cm}
\subfloat[Power capacity.]{\includegraphics[width=0.4\textwidth]{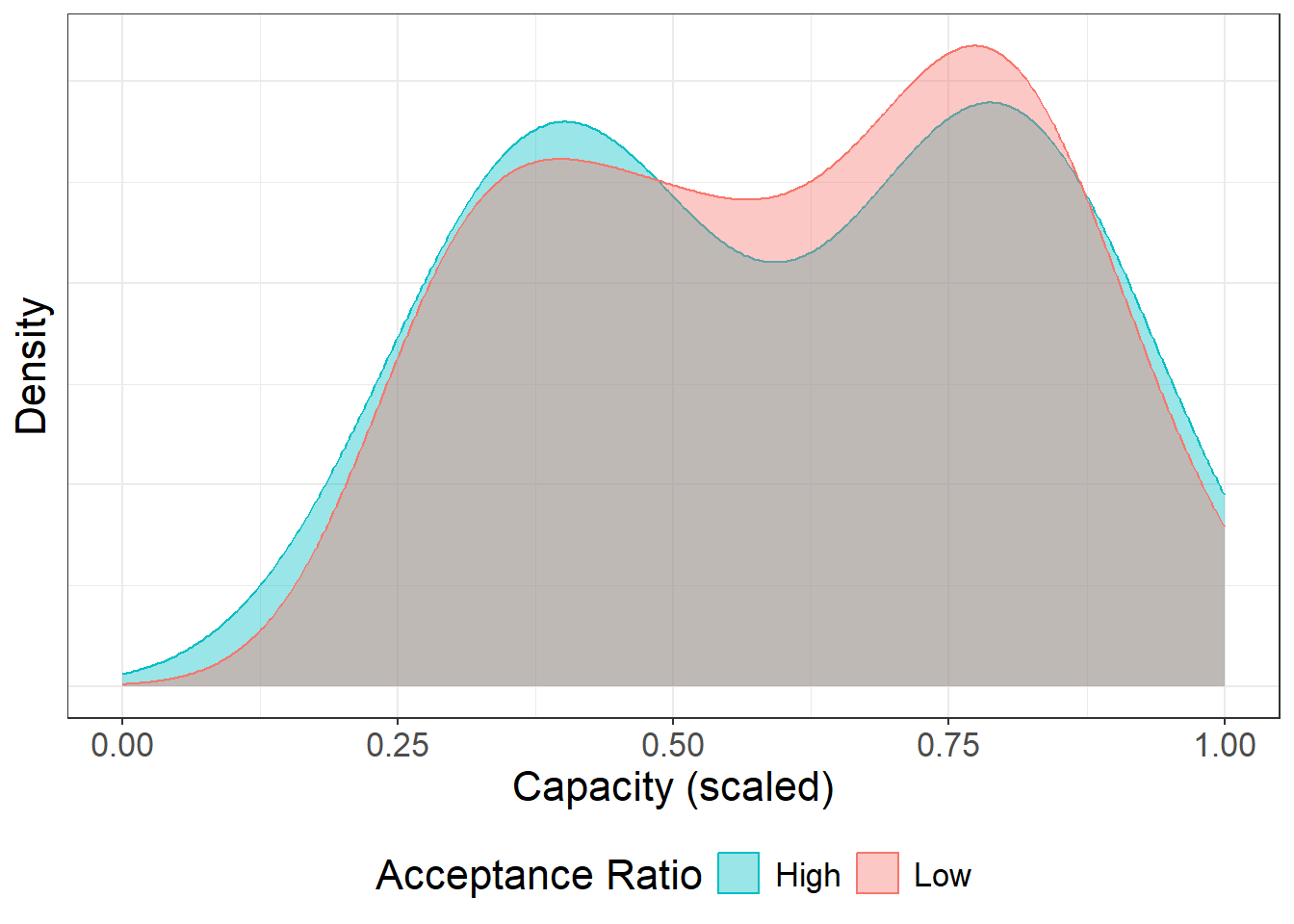}} 
\caption{Distribution of continuous control variables across data centers (using a 60\% threshold), smoothed via Kernel Density Estimation from the \texttt{geom\_density} function in the \texttt{ggplot2} package in R.} 
\label{fig:controls}
\vspace{-12pt}
\end{figure}

\begin{figure}[h!]
\centering
\small
\subfloat[Threshold: 60\%.]{\includegraphics[width=0.4\textwidth]{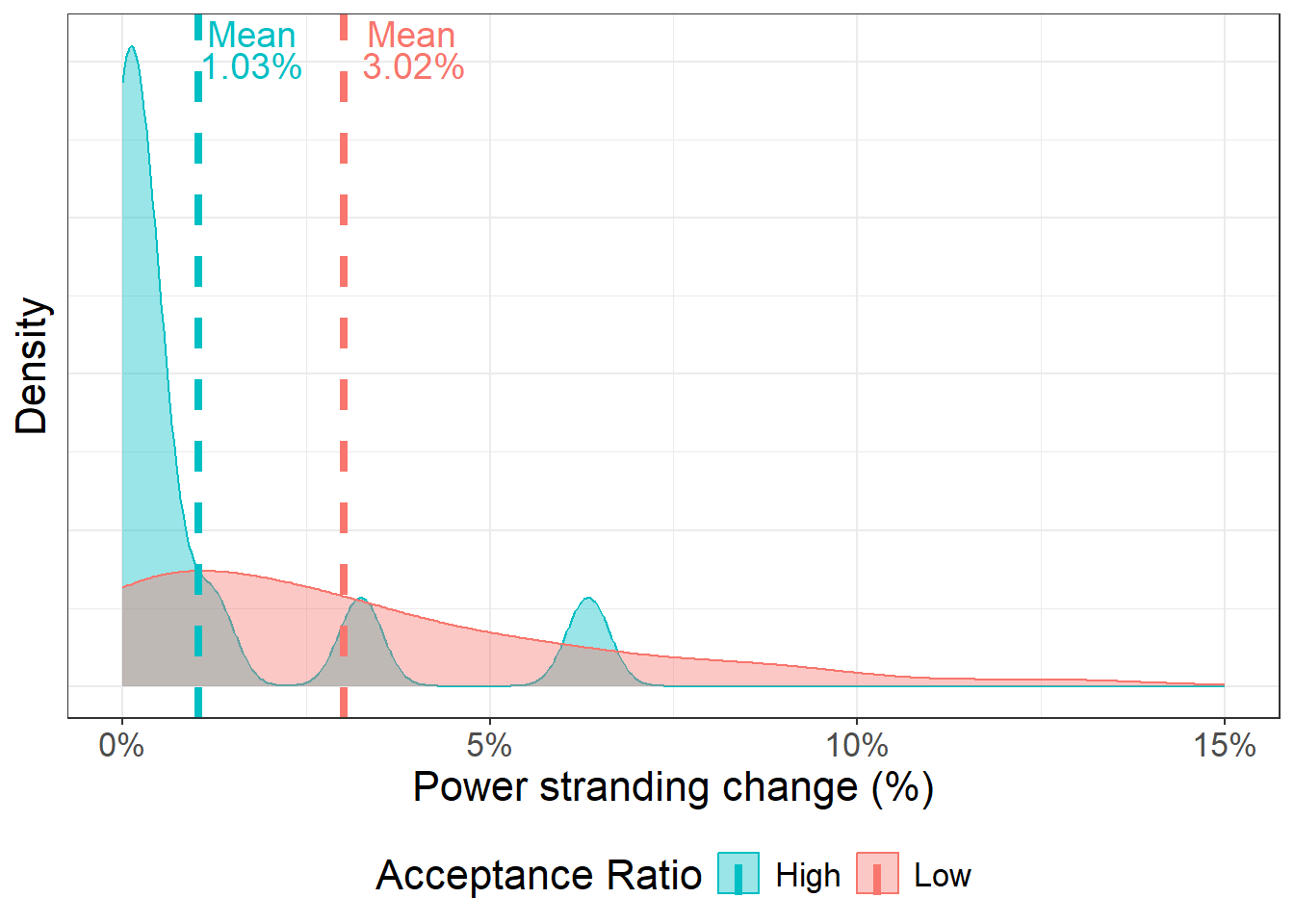}} 
\hspace{0.1 cm}
\subfloat[Threshold: 45\%.]{\includegraphics[width=0.4\textwidth]{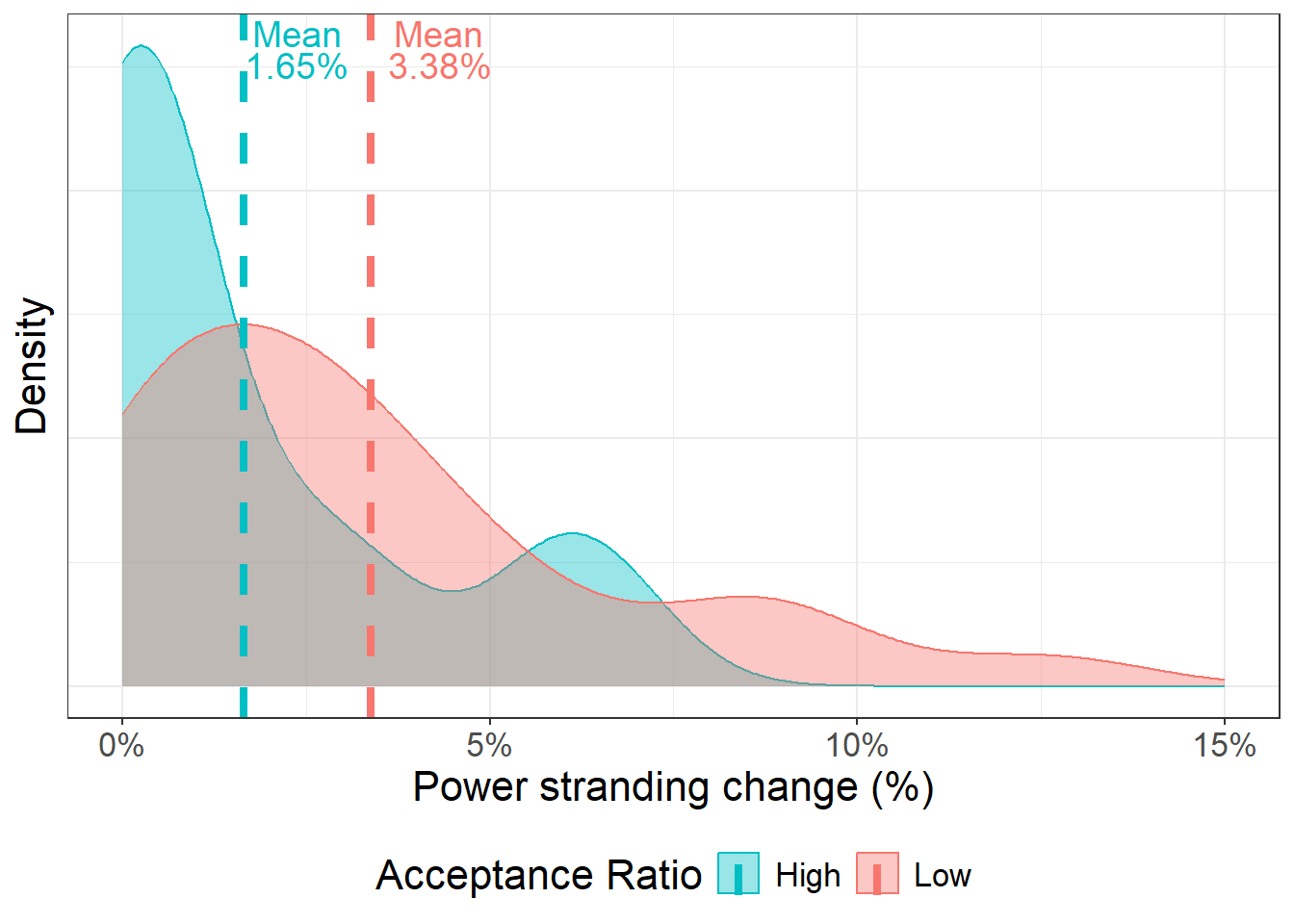}} 
\caption{Distribution of the outcome variable across data centers, smoothed via Kernel Density Estimation from the \texttt{geom\_density} function in the \texttt{ggplot2} package in R.} 
\label{fig:treatment}
\vspace{-12pt}
\end{figure}

In summary, raw deployment data suggest that high- and low-adoption data centers are statistically indistinguishable across most control variables, but then feature a smaller increase in power stranding. Next, we corroborate these observations via an econometric analysis.

\subsubsection*{OLS regression.}

We propose the following specification, where the outcome variable $y_i$ denotes the increase in power stranding in data center $i$, $\texttt{Controls}_i$ refers to the vector of the seven control variables, the treatment variable $\tau_i$ measures adoption, and $\varepsilon_i$ denotes idiosyncratic noise. The coefficient of interest is $\delta$, which measures the impact of adoption on the increase in power stranding. Our hypothesis is that $\delta<0$, reflecting that, all else equal, data centers with a higher adoption of our solution face a milder increase in power stranding during the deployment period.
$$y_i=\beta_0+\bbeta^\top\texttt{Controls}_i+\delta\tau_i+\varepsilon_i$$

\Cref{tab:regression} reports the regression results in an increasingly controlled environment. Results show that adoption has a negative impact on power stranding, and that the effect is statistically significant at the 5\% level. This finding is robust and consistent across all six model specifications. These results confirm the aggregated population-based averages, indicating that the algorithm does, in fact, have a moderating impact on power stranding. In terms of magnitude, the average treatment effect is estimated between $-2.99$ to $-3.58$, meaning that full adoption of our algorithmic solution---resulting in a shift from 0 to 1 in the adoption rate---would mitigate power stranding during the 11-month deployment period by 3 percentage points in the average data center.

\begin{table}[h!]
    \footnotesize
    \renewcommand{\arraystretch}{1}
    \caption{OLS regression estimates, with six controlled specifications.}
    \label{tab:regression}
    \begin{tabular}{
        S[table-column-width=5cm,table-text-alignment=left]
        *{6}{S[table-format = -1.5,digit-group-size=5]}
    }
        \toprule
    	& {(1)} & {(2)} & {(3)} & {(4)} & {(5)} & {(6)} 
        \\
        \midrule
        \texttt{adoption} 
        & -0.0304\sym{**}    & -0.0299\sym{**}    & -0.0324\sym{**}    & -0.0320\sym{**}    & -0.0358\sym{**}    & -0.0358\sym{**}    \\
        &  (0.0135)         &  (0.0134)         &  (0.0140)         &  (0.0144)         &  (0.0140)         &  (0.0143)         \\
    	{Demand (\% of capacity)}
        &                   & 0.0434            & 0.0871            & 0.1365            & 0.0184\sym{**}     & 0.0184\sym{**}     \\
        &                   &  (0.0371)         &  (0.0849)         &  (0.0884)         &  (0.0888)         &  (0.0905)         \\
    	{Initial power stranding (\% of capacity)}
        &                   &                   & 0.0984            & 0.0719            & 0.0996            & 0.0996            \\
        &                   &                   &  (0.1195)         &  (0.1185)         &  (0.1155)         &  (0.1171)         \\
    	{Initial utilization (\% of capacity)}
        &                   &                   & 0.0374            & 0.0482            & 0.0925            & 0.0925            \\
        &                   &                   &  (0.0592)         &  (0.0596)         &  (0.0618)         &  (0.0626)         \\
    	{IT capacity}
        &                   &                   &                   & -0.0263           & -0.0091           & -0.0091           \\
        &                   &                   &                   &  (0.0229)         &  (0.0238)         &  (0.0245)         \\
    	{Rooms}
        &                   &                   &                   & -0.0024           & -0.0029           & -0.0029           \\
        &                   &                   &                   &  (0.0026)         &  (0.0025)         &  (0.0026)         \\
    	{Flex architecture}
        &                   &                   &                   &                   & -0.0177\sym{*}     & -0.0177\sym{*}     \\
        &                   &                   &                   &                   &  (0.0090)         &  (0.0091)         \\
    	{Location-US}
        &                   &                   &                   &                   &                   & -0.00003          \\
        &                   &                   &                   &                   &                   &  (0.0096)         \\
        \midrule
        {Observations}
        &  {49}      &  {49}      &  {49}      &  {49}      &  {49}      &  {49}      \\
        {Adjusted $R^2$}
        &  0.079            &  0.086            &  0.060            &  0.091            &  0.149            &  0.128            \\
    \bottomrule
    \end{tabular}
    \begin{tablenotes}
	\item \sym{*}, \sym{**}, and \sym{***} indicate significance levels of 10\%, 5\%, and 1\%. Standard deviations in parentheses.
    \end{tablenotes}
    \vspace{-12pt}
\end{table}

\subsubsection*{Robustness and discussion.}

We use propensity score matching (PSM) to corroborate our OLS regression estimates by matching high-adoption data centers to low-adoption ones with similar characteristics in terms of the control variables (see~\ref{app:PSM}). Together, econometric results establish that power stranding increases more slowly within data centers with high adoption than low adoption, while controlling for potential confounders. Over an 11-month period, this can translate into a difference of 3 percentage points between zero-adoption and full-adoption data centers (OLS estimates) and of 1--2 percentage points between low- and high-adoption data centers (PSM estimates). At the scale of Microsoft's cloud computing operations, a percentage-point increase in power utilization represents savings on the order of hundreds of millions of dollars. Moreover, better power utilization can yield environmental benefits by delaying building and utilizing new data centers. Given the number and size of Microsoft's data centers \citep{numberDCs,sizeDC,sizeDC2} and the embodied carbon per square foot of a data center \citep{embodiedCarbon}, a percentage-point increase in power utilization can translate into an order of tens to hundreds of thousands of tons of CO$_2$ equivalents.
\section{Conclusion}

This paper addresses the rack placement problem in data center operations. We formulated an integer optimization problem to maximize utilization under space, cooling, power and redundancy constraints. To solve it, we proposed an online sampling optimization (OSO) algorithm as an easily-implementable and generalizable approach in multi-stage stochastic optimization. The algorithm relies on a single-sample or a small-sample approximation of uncertainty along with online re-optimization; thus, it solves a deterministic approximation or a two-stage stochastic optimization at each iteration. Theoretical results established performance guarantees of single-sample OSO. In particular, in canonical online resource allocation, OSO achieves a multiplicative loss that scales with the number of resources $d$ in $\calO(\sqrt{\log d})$, with the time horizon $T$ in $\calO(\sqrt{\log T})$ and with resource capacities $B$ in $\calO(1/\sqrt{B})$. We also showed that single-sample OSO can yield unbounded improvements as compared to mean-based resolving heuristics. We corroborated these insights with computational results, suggesting that OSO can return high-quality solutions in manageable computational times for a range of online optimization problems, outperforming benchmarks.

We packaged the optimization model and algorithm into a dedicated decision-support software tool to deploy it across Microsoft's data centers. Thanks to iterative model improvements performed in close collaboration with data center managers, our solution was increasingly adopted in practice. Using post-deployment data, we conducted econometric analyses to identify the impact of our solution in practice. Results suggest that adoption of our solution has a positive and statistically significant impact on data center performance, resulting in a decrease in power stranding by 1--3 percentage points. These energy efficiency improvements can translate into very large financial and environmental benefits at the scale of Microsoft's cloud computing operations.

These positive results also motivate future research in online resource allocation and cloud supply chains. Methodologically, the OSO algorithm could be augmented with probabilistic allocations and thresholding rules, which have been successful in certainty-equivalent resolving heuristics, and be compared to distribution-free dual mirror descent approaches in online resource allocation. Other opportunities involve characterizing performance guarantees of the small-sample OSO algorithms, and augmenting it with other stochastic programming techniques such as progressive hedging \citep{rockafellar1991scenarios} or two-stage decision rules \citep{bodur2022two}. Practically, the rack placement model could be integrated into the optimization of upstream data center design and downstream virtual machine management. At a time when cloud computing is growing into a major component of modern supply chains, this paper contributes methodologies, theoretical guarantees, and empirical evidence toward the optimization of data center operations.

\ACKNOWLEDGMENT{This work was partially supported by the MIT Center for Transportation and Logistics UPS PhD Fellowship.}

\bibliographystyle{informs2014} 
\bibliography{ref}

@misc{embodiedCarbon,
    title="Technical Information Paper: Embodied Carbon in Enterprise Data Centre IT Equipment",
    author="{{ADW} Developments}",
    year="2025",
    note={Accessed at \url{https://adwdevelopments.com/sustainability/technical-information-paper-embodied-carbon-in-enterprise-data-centre-it-equipment/} on 26 Dec 2025}
}

@misc{numberDCs,
    title="Azure global infrastructure",
    author="{Microsoft}",
    year="2025",
    note={Accessed at \url{https://azure.microsoft.com/en-us/explore/global-infrastructure} on 26 Dec 2025}
}

@misc{sizeDC,
    title="Inside the world’s most powerful AI datacenter",
    author="{Microsoft}",
    year="2025",
    note={Accessed at \url{https://blogs.microsoft.com/blog/2025/09/18/inside-the-worlds-most-powerful-ai-datacenter/} on 26 Dec 2025}
}

@misc{sizeDC2,
    title="Data centers from Microsoft
",
    author="{Aterio}",
    year="2026",
    note={Accessed at \url{https://www.aterio.io/us-data-centers/providers/0820543f-b9ae-43de-8887-751f944d38a3} on 12 Jan 2026}
}

@misc{DecomLifespan,
    title="Data Center Hardware Refresh Cutback by Microsoft",
    author="{Data Center Knowledge}",
    year="2022",
    note={Accessed at \url{https://www.datacenterknowledge.com/hyperscalers/data-center-hardware-refresh-cutback-by-microsoft-what-s-next-} on 22 Dec 2025}
}

@misc{DC2024,
    title="Charted: How Many Data Centers do Major Big Tech Companies Have?",
    author="{Pallavi Rao}",
    year="2024",
    note={Accessed at \url{https://www.visualcapitalist.com/charted-how-many-data-centers-do-major-big-tech-companies-have/} on 22 Dec 2025}
}

@misc{DC2025,
    title="The Backbone of Microsoft Cloud: Our Datacenters",
    author="{Microsoft Datacenters}",
    year="2025",
    note={Accessed at \url{https://datacenters.microsoft.com/} on 22 Dec 2025}
}

@article{cmsz2023,
author = {Chen, Shi and Moinzadeh, Kamran and Song, Jing-Sheng and Zhong, Yuan},
title = {Cloud Computing Value Chains: Research from the Operations Management Perspective},
journal = {Manufacturing \& Service Operations Management},
volume = {25},
number = {4},
pages = {1338-1356},
year = {2023}
}

@article{gr2020,
    journal={Operations Research},
    volume={68},
    number={5},
    year={2020},
    pages={1474-1492},
    title={Interior-Point-Based Online Stochastic Bin Packing},
    author={Varun Gupta and Ana Radovanovic}
}

@article{ckmz2019,
    author={Maxime C. Cohen and Philipp W. Keller and Vahab Mirrokni and Morteza Zadimoghaddam},
    journal={Management Science},
    title={Overcommitment in Cloud Services: Bin Packing with Chance Constraints},
    volume={65},
    number={7},
    year={2019}
}

@INPROCEEDINGS{xl2013,
  author={Xu, Hong and Li, Baochun},
  booktitle={2013 Proceedings IEEE INFOCOM}, 
  title={Joint request mapping and response routing for geo-distributed cloud services}, 
  year={2013},
  volume={},
  number={},
  pages={854-862}
}

@InProceedings{mmz2023,
  author =	{Mellou, Konstantina and Molinaro, Marco and Zhou, Rudy},
  title =	{{Online Demand Scheduling with Failovers}},
  booktitle =	{50th International Colloquium on Automata, Languages, and Programming (ICALP 2023)},
  pages =	{92:1--92:20},
  year =	{2023},
  volume =	{261}
}

@article{bfmmn2022,
author = {Buchbinder, Niv and Fairstein, Yaron and Mellou, Konstantina and Menache, Ishai and Naor, Joseph (Seffi)},
title = {Online Virtual Machine Allocation with Lifetime and Load Predictions},
year = {2022},
publisher = {Association for Computing Machinery},
volume = {49},
number = {1},
journal = {SIGMETRICS Perform. Eval. Rev.},
pages = {9–10},
numpages = {2}
}

@inproceedings{lyu2023hyrax,
author = {Lyu, Jialun and You, Marisa and Irvene, Celine and Jung, Mark and Narmore, Tyler and Shapiro, Jacob and Marshall, Luke and Samal, Savyasachi and Manousakis, Ioannis and Hsu, Lisa and Subbarayalu, Preetha and Raniwala, Ashish and Warrier, Brijesh and Bianchini, Ricardo and Shroeder, Bianca and Berger, Daniel S.},
title = {Hyrax: Fail-in-Place Server Operation in Cloud Platforms},
organization = {USENIX},
booktitle = {Proceedings of the 17th Symposium on Operating Systems Design and Implementation (OSDI)},
year = {2023}
}

@inproceedings{zhang2021flex,
  title={Flex: High-availability datacenters with zero reserved power},
  author={Zhang, Chaojie and Kumbhare, Alok Gautam and Manousakis, Ioannis and Zhang, Deli and Misra, Pulkit A and Assis, Rod and Woolcock, Kyle and Mahalingam, Nithish and Warrier, Brijesh and Gauthier, David and others},
  booktitle={2021 ACM/IEEE 48th Annual International Symposium on Computer Architecture (ISCA)},
  pages={319--332},
  year={2021},
  organization={IEEE}
}

@article{carbondc,
  title={Carbon-aware computing for datacenters},
  author={Radovanovi{\'c}, Ana and Koningstein, Ross and Schneider, Ian and Chen, Bokan and Duarte, Alexandre and Roy, Binz and Xiao, Diyue and Haridasan, Maya and Hung, Patrick and Care, Nick and others},
  journal={IEEE Transactions on Power Systems},
  volume={38},
  number={2},
  pages={1270--1280},
  year={2022},
  publisher={IEEE}
}

@techreport{uptime,
author={{Uptime Institute}},
title={Data Center Site Infrastructure Tier Standard: Operational Sustainability},
year={2014},
note={{http://uptimeinstitute.com/publications}
}}

@article{rhee1993lineii,
  title={On-line bin packing of items of random sizes, {II}},
  author={Rhee, Wansoo T and Talagrand, Michel},
  journal={SIAM Journal on Computing},
  volume={22},
  number={6},
  pages={1251--1256},
  year={1993},
  publisher={SIAM}
}

@article{GM-MOR16,
  author    = {Anupam Gupta and
               Marco Molinaro},
  title     = {How the Experts Algorithm Can Help Solve {LPs} Online},
  journal   = {Mathematics of Operations Research},
  volume    = {41},
  number    = {4},
  pages     = {1404--1431},
  year      = {2016}
}

@techreport{nrdc,
    title={Data Center Efficiency Assessment},
    author={{National Resources Defense Council}},
    year={2014},
    howpublished={https://www.nrdc.org/sites/default/files/data-center-efficiency-assessment-IB.pdf}
}

@article{wu2016dynamo,
  title={Dynamo: Facebook's data center-wide power management system},
  author={Wu, Qiang and Deng, Qingyuan and Ganesh, Lakshmi and Hsu, Chang-Hong and Jin, Yun and Kumar, Sanjeev and Li, Bin and Meza, Justin and Song, Yee Jiun},
  journal={ACM SIGARCH Computer Architecture News},
  volume={44},
  number={3},
  pages={469--480},
  year={2016},
  publisher={ACM New York, NY, USA}
}

@article{pst2022,
  title={Adaptive bin packing with overflow},
  author={Perez-Salazar, Sebastian and Singh, Mohit and Toriello, Alejandro},
  journal={Mathematics of Operations Research},
  volume={47},
  number={4},
  pages={3317--3356},
  year={2022},
  publisher={INFORMS}
}

@book{van1997weak,
  title={Weak convergence and empirical processes: with applications to statistics},
  author={{Van Der Vaart}, Aad W and Wellner, Jon A},
  year={1996},
  publisher={Springer New York}
}

@techreport{iea,
author={{International Energy Agency}},
title={Infrastructure Deep Dive: Data Centres and Data Transmission Networks},
year={2022},
howpublished={https://www.iea.org/reports/data-centres-and-data-transmission-networks}
}

@book{powell2022reinforcement,
  title={Reinforcement Learning and Stochastic Optimization: A Unified Framework for Sequential Decisions},
  author={Powell, Warren B},
  year={2022},
  publisher={John Wiley \& Sons}
}

@article{kleywegt2002sample,
  title={The sample average approximation method for stochastic discrete optimization},
  author={Kleywegt, Anton J and Shapiro, Alexander and Homem-de-Mello, Tito},
  journal={SIAM Journal on optimization},
  volume={12},
  number={2},
  pages={479--502},
  year={2002},
  publisher={SIAM}
}

@article{rhee1993line,
  title={On line bin packing with items of random size},
  author={Rhee, Wansoo T and Talagrand, Michel},
  journal={Mathematics of Operations Research},
  volume={18},
  number={2},
  pages={438--445},
  year={1993},
  publisher={INFORMS}
}

@article{schroeder2006closed,
  title={Closed versus open system models: a cautionary tale},
  author={Schroeder, Bianca and Wierman, Adam and Harchol-Balter, Mor},
  journal={Network System Design and Implementation},
  year={2006}
}

@article{gardner2017redundancy,
  title={Redundancy-d: The power of d choices for redundancy},
  author={Gardner, Kristen and Harchol-Balter, Mor and Scheller-Wolf, Alan and Velednitsky, Mark and Zbarsky, Samuel},
  journal={Operations Research},
  volume={65},
  number={4},
  pages={1078--1094},
  year={2017},
  publisher={INFORMS}
}

@inproceedings{romisch2009scenario,
  title={Scenario reduction techniques in stochastic programming},
  author={R{\"o}misch, Werner},
  booktitle={International Symposium on Stochastic Algorithms},
  pages={1--14},
  year={2009},
  organization={Springer}
}

@article{bertsimas2022optimization,
  title={Optimization-based scenario reduction for data-driven two-stage stochastic optimization},
  author={Bertsimas, Dimitris and Mundru, Nishanth},
  journal={Operations Research},
  year={2022},
  publisher={INFORMS}
}

@article{zhang2023optimized,
  title={Optimized Scenario Reduction: Solving Large-Scale Stochastic Programs with Quality Guarantees},
  author={Zhang, Wei and Wang, Kai and Jacquillat, Alexandre and Wang, Shuaian},
  journal={INFORMS Journal on Computing},
  year={2023},
  publisher={INFORMS}
}

@book{sutton2018reinforcement,
  title={Reinforcement learning: An introduction},
  author={Sutton, Richard S and Barto, Andrew G},
  year={2018},
  publisher={MIT press}
}

@article{maglaras2006dynamic,
  title={Dynamic pricing strategies for multiproduct revenue management problems},
  author={Maglaras, Constantinos and Meissner, Joern},
  journal={Manufacturing \& Service Operations Management},
  volume={8},
  number={2},
  pages={136--148},
  year={2006},
  publisher={INFORMS}
}

@article{secomandi2008analysis,
  title={An analysis of the control-algorithm re-solving issue in inventory and revenue management},
  author={Secomandi, Nicola},
  journal={Manufacturing \& Service Operations Management},
  volume={10},
  number={3},
  pages={468--483},
  year={2008},
  publisher={INFORMS}
}

@book{bertsekas2012dynamic,
  title={Dynamic programming and optimal control: Volume I},
  author={Bertsekas, Dimitri},
  volume={4},
  year={2012},
  publisher={Athena scientific}
}

@article{jasin2012re,
  title={A re-solving heuristic with bounded revenue loss for network revenue management with customer choice},
  author={Jasin, Stefanus and Kumar, Sunil},
  journal={Mathematics of Operations Research},
  volume={37},
  number={2},
  pages={313--345},
  year={2012},
  publisher={INFORMS}
}

@article{bumpensanti2020re,
  title={A re-solving heuristic with uniformly bounded loss for network revenue management},
  author={Bumpensanti, Pornpawee and Wang, He},
  journal={Management Science},
  volume={66},
  number={7},
  pages={2993--3009},
  year={2020},
  publisher={INFORMS}
}

@article{arlotto2019uniformly,
  title={Uniformly bounded regret in the multisecretary problem},
  author={Arlotto, Alessandro and Gurvich, Itai},
  journal={Stochastic Systems},
  volume={9},
  number={3},
  pages={231--260},
  year={2019},
  publisher={INFORMS}
}

@article{vera2021bayesian,
  title={The bayesian prophet: A low-regret framework for online decision making},
  author={Vera, Alberto and Banerjee, Siddhartha},
  journal={Management Science},
  volume={67},
  number={3},
  pages={1368--1391},
  year={2021},
  publisher={INFORMS}
}

@article{banerjee2024good,
  title={Good prophets know when the end is near},
  author={Banerjee, Siddhartha and Freund, Daniel},
  journal={Management Science},
  year={2024},
  publisher={INFORMS}
}

@inproceedings{freund2021overbooking,
  title={Overbooking with bounded loss},
  author={Freund, Daniel and Zhao, Jiayu},
  booktitle={Proceedings of the 22nd ACM Conference on Economics and Computation},
  pages={477--478},
  year={2021}
}

@inproceedings{prophetLimitedInfo,
  title={Prophet inequalities with limited information},
  author={Azar, Pablo D and Kleinberg, Robert and Weinberg, S Matthew},
  booktitle={Proceedings of the twenty-fifth annual ACM-SIAM symposium on Discrete algorithms},
  pages={1358--1377},
  year={2014},
  organization={SIAM}
}

@article{rubinsteinSingleSample,
  title={Optimal single-choice prophet inequalities from samples},
  author={Rubinstein, Aviad and Wang, Jack Z and Weinberg, S Matthew},
  journal={arXiv preprint arXiv:1911.07945},
  year={2019}
}

@inproceedings{singleSampleProphet,
  title={Single-sample prophet inequalities via greedy-ordered selection},
  author={Caramanis, Constantine and D{\"u}tting, Paul and Faw, Matthew and Fusco, Federico and Lazos, Philip and Leonardi, Stefano and Papadigenopoulos, Orestis and Pountourakis, Emmanouil and Reiffenh{\"a}user, Rebecca},
  booktitle={Proceedings of the 2022 Annual ACM-SIAM Symposium on Discrete Algorithms (SODA)},
  pages={1298--1325},
  year={2022},
  organization={SIAM}
}

@article{arbabian2021capacity,
  title={Capacity expansions with bundled supplies of attributes: An application to server procurement in cloud computing},
  author={Arbabian, Mohammad Ebrahim and Chen, Shi and Moinzadeh, Kamran},
  journal={Manufacturing \& Service Operations Management},
  volume={23},
  number={1},
  pages={191--209},
  year={2021},
  publisher={INFORMS}
}

@article{muir2024temporal,
  title={Temporal Bin Packing with Half-Capacity Jobs},
  author={Muir, Christopher and Marshall, Luke and Toriello, Alejandro},
  journal={INFORMS Journal on Optimization},
  volume={6},
  number={1},
  pages={46--62},
  year={2024},
  publisher={INFORMS}
}

@article{liu2023efficient,
  title={Efficient Cloud Server Deployment Under Demand Uncertainty},
  author={Liu, Rui Peng and Mellou, Konstantina and Gong, Xiao-Yue and Li, Beibin and Coffee, Thomas and Pathuri, Jeevan and Simchi-Levi, David and Menache, Ishai},
  journal={Available at SSRN 4501810},
  year={2023}
}

@article{grosof2022optimal,
  title={Optimal scheduling in the multiserver-job model under heavy traffic},
  author={Grosof, Isaac and Scully, Ziv and Harchol-Balter, Mor and Scheller-Wolf, Alan},
  journal={Proceedings of the ACM on Measurement and Analysis of Computing Systems},
  volume={6},
  number={3},
  pages={1--32},
  year={2022},
  publisher={ACM New York, NY, USA}
}

@article{cohen2023managing,
  title={Managing airfares under competition: Insights from a field experiment},
  author={Cohen, Maxime C and Jacquillat, Alexandre and Serpa, Juan Camilo and Benborhoum, Michael},
  journal={Management Science},
  volume={69},
  number={10},
  pages={6076--6108},
  year={2023},
  publisher={INFORMS}
}

@article{rubin2001using,
  title={Using propensity scores to help design observational studies: application to the tobacco litigation},
  author={Rubin, Donald B},
  journal={Health Services and Outcomes Research Methodology},
  volume={2},
  pages={169--188},
  year={2001},
  publisher={Springer}
}

@article{bray2016multitasking,
  title={Multitasking, multiarmed bandits, and the Italian judiciary},
  author={Bray, Robert L and Coviello, Decio and Ichino, Andrea and Persico, Nicola},
  journal={Manufacturing \& Service Operations Management},
  volume={18},
  number={4},
  pages={545--558},
  year={2016},
  publisher={Informs}
}

@article{stamatopoulos2021effects,
  title={The effects of menu costs on retail performance: Evidence from adoption of the electronic shelf label technology},
  author={Stamatopoulos, Ioannis and Bassamboo, Achal and Moreno, Antonio},
  journal={Management Science},
  volume={67},
  number={1},
  pages={242--256},
  year={2021},
  publisher={INFORMS}
}

@article{rosenbaum1983central,
  title={The central role of the propensity score in observational studies for causal effects},
  author={Rosenbaum, Paul R and Rubin, Donald B},
  journal={Biometrika},
  volume={70},
  number={1},
  pages={41--55},
  year={1983},
  publisher={Oxford University Press}
}

@article{stuart2013prognostic,
  title={Prognostic score--based balance measures can be a useful diagnostic for propensity score methods in comparative effectiveness research},
  author={Stuart, Elizabeth A and Lee, Brian K and Leacy, Finbarr P},
  journal={Journal of clinical epidemiology},
  volume={66},
  number={8},
  pages={S84--S90},
  year={2013},
  publisher={Elsevier}
}

@article{
    chen2023feature,
    title={Feature based dynamic matching},
    author={Chen, Yilun and Kanoria, Yash and Kumar, Akshit and Zhang, Wenxin},
    journal={Available at SSRN 4451799},
    year={2023}
}

@article{
    ciocan2012model,
    author = {Ciocan, Dragos Florin and Farias, Vivek},
    title = {Model Predictive Control for Dynamic Resource Allocation},
    journal = {Mathematics of Operations Research},
    volume = {37},
    number = {3},
    pages = {501-525},
    year = {2012},
    doi = {10.1287/moor.1120.0548},
    URL = {https://doi.org/10.1287/moor.1120.0548},
    eprint = {https://doi.org/10.1287/moor.1120.0548}
}

@article{defond2017client,
  title={Do client characteristics really drive the Big N audit quality effect? New evidence from propensity score matching},
  author={DeFond, Mark and Erkens, David H and Zhang, Jieying},
  journal={Management Science},
  volume={63},
  number={11},
  pages={3628--3649},
  year={2017},
  publisher={INFORMS}
}

@article{rockafellar1991scenarios,
    title={Scenarios and policy aggregation in optimization under uncertainty},
    author={Rockafellar, R Tyrrell and Wets, Roger J-B},
    journal={Mathematics of operations research},
    volume={16},
    number={1},
    pages={119--147},
    year={1991},
    publisher={INFORMS}
}

@article{
    valerio1999exact,
    title={Exact solution of bin-packing problems using column generation and branch-and-bound},
    author={Val{\'e}rio de Carvalho, JM},
    journal={Annals of Operations Research},
    volume={86},
    number={0},
    pages={629--659},
    year={1999},
    publisher={Springer}
}

@article{chen2013simple,
  title={Simple policies for dynamic pricing with imperfect forecasts},
  author={Chen, Yiwei and Farias, Vivek F},
  journal={Operations Research},
  volume={61},
  number={3},
  pages={612--624},
  year={2013},
  publisher={INFORMS}
}

@article{gallego1994optimal,
  title={Optimal dynamic pricing of inventories with stochastic demand over finite horizons},
  author={Gallego, Guillermo and Van Ryzin, Garrett},
  journal={Management science},
  volume={40},
  number={8},
  pages={999--1020},
  year={1994},
  publisher={INFORMS}
}

@article{gallego1997multiproduct,
  title={A multiproduct dynamic pricing problem and its applications to network yield management},
  author={Gallego, Guillermo and Van Ryzin, Garrett},
  journal={Operations research},
  volume={45},
  number={1},
  pages={24--41},
  year={1997},
  publisher={INFORMS}
}

@article{cooper2002asymptotic,
  title={Asymptotic behavior of an allocation policy for revenue management},
  author={Cooper, William L},
  journal={Operations Research},
  volume={50},
  number={4},
  pages={720--727},
  year={2002},
  publisher={INFORMS}
}

@article{jasin2013analysis,
  title={Analysis of deterministic LP-based booking limit and bid price controls for revenue management},
  author={Jasin, Stefanus and Kumar, Sunil},
  journal={Operations Research},
  volume={61},
  number={6},
  pages={1312--1320},
  year={2013},
  publisher={INFORMS}
}

@article{reiman2008asymptotically,
  title={An asymptotically optimal policy for a quantity-based network revenue management problem},
  author={Reiman, Martin I and Wang, Qiong},
  journal={Mathematics of Operations Research},
  volume={33},
  number={2},
  pages={257--282},
  year={2008},
  publisher={INFORMS}
}

@article{bray2024logarithmic,
  title={Logarithmic regret in multisecretary and online linear programs with continuous valuations},
  author={Bray, Robert L},
  journal={Operations Research},
  year={2024},
  publisher={INFORMS}
}

@article{jiang2025degeneracy,
  title={Degeneracy is ok: Logarithmic regret for network revenue management with indiscrete distributions},
  author={Jiang, Jiashuo and Ma, Will and Zhang, Jiawei},
  journal={Operations Research},
  year={2025},
  publisher={INFORMS}
}

@article{arlotto2020logarithmic,
  title={Logarithmic regret in the dynamic and stochastic knapsack problem with equal rewards},
  author={Arlotto, Alessandro and Xie, Xinchang},
  journal={Stochastic Systems},
  volume={10},
  number={2},
  pages={170--191},
  year={2020},
  publisher={INFORMS}
}

@inproceedings{besbes2022multi,
  title={The multi-secretary problem with many types},
  author={Besbes, Omar and Kanoria, Yash and Kumar, Akshit},
  booktitle={Proceedings of the 23rd ACM Conference on Economics and Computation},
  pages={1146--1147},
  year={2022}
}

@article{lueker1998average,
  title={Average-case analysis of off-line and on-line knapsack problems},
  author={Lueker, George S},
  journal={Journal of Algorithms},
  volume={29},
  number={2},
  pages={277--305},
  year={1998},
  publisher={Elsevier}
}

@article{li2022online,
  title={Online linear programming: Dual convergence, new algorithms, and regret bounds},
  author={Li, Xiaocheng and Ye, Yinyu},
  journal={Operations Research},
  volume={70},
  number={5},
  pages={2948--2966},
  year={2022},
  publisher={INFORMS}
}

@article{balseiro2024survey,
  title={Survey of dynamic resource-constrained reward collection problems: Unified model and analysis},
  author={Balseiro, Santiago R and Besbes, Omar and Pizarro, Dana},
  journal={Operations Research},
  volume={72},
  number={5},
  pages={2168--2189},
  year={2024},
  publisher={INFORMS}
}

@article{liu2021online,
  title={Online bin packing with known T},
  author={Liu, Shang and Li, Xiaocheng},
  journal={arXiv preprint arXiv:2112.03200},
  year={2021}
}

@article{chen2025beyond,
  title={Beyond Non-Degeneracy: Revisiting Certainty Equivalent Heuristic for Online Linear Programming},
  author={Chen, Yilun and Wang, Wenjia},
  journal={arXiv preprint arXiv:2501.01716},
  year={2025}
}

@inproceedings{kleinberg,
author = {Kleinberg, Robert},
title = {A multiple-choice secretary algorithm with applications to online auctions},
booktitle = {SODA},
year = {2005},
isbn = {0-89871-585-7},
location = {Vancouver, British Columbia},
numpages = {2},
}

@inproceedings{agrawal2014fast,
  title={Fast algorithms for online stochastic convex programming},
  author={Agrawal, Shipra and Devanur, Nikhil R},
  booktitle={Proceedings of the twenty-sixth annual ACM-SIAM symposium on Discrete algorithms},
  pages={1405--1424},
  year={2014},
  organization={SIAM}
}

@inproceedings{kesselheim,
author = {Kesselheim, Thomas and T\"{o}nnis, Andreas and Radke, Klaus and V\"{o}cking, Berthold},
title = {Primal Beats Dual on Online Packing LPs in the Random-order Model},
booktitle = {Proceedings of the 46th Annual ACM Symposium on Theory of Computing},
series = {STOC '14},
year = {2014},
isbn = {978-1-4503-2710-7},
location = {New York, New York},
pages = {303--312},
numpages = {10},
url = {http://doi.acm.org/10.1145/2591796.2591810},
doi = {10.1145/2591796.2591810},
acmid = {2591810},
publisher = {ACM},
address = {New York, NY, USA},
}

@article{kesselheim2020knapsack,
  title={Knapsack secretary with bursty adversary},
  author={Kesselheim, Thomas and Molinaro, Marco},
  journal={arXiv preprint arXiv:2006.11607},
  year={2020}
}

@article{guptaMolinaroMOR,
author = {Gupta, Anupam and Molinaro, Marco},
title = {How the Experts Algorithm Can Help Solve LPs Online},
journal = {Mathematics of Operations Research},
volume = {41},
number = {4},
pages = {1404-1431},
year = {2016},
doi = {10.1287/moor.2016.0782},
URL = { https://doi.org/10.1287/moor.2016.0782 },
eprint = { https://doi.org/10.1287/moor.2016.0782 }
}

@inproceedings{10.5555/1888935.1888957,
author = {Feldman, Jon and Henzinger, Monika and Korula, Nitish and Mirrokni, Vahab S. and Stein, Cliff},
title = {Online stochastic packing applied to display ad allocation},
year = {2010},
isbn = {3642157742},
publisher = {Springer-Verlag},
address = {Berlin, Heidelberg},
booktitle = {Proceedings of the 18th Annual European Conference on Algorithms: Part I},
pages = {182–194},
numpages = {13},
location = {Liverpool, UK},
series = {ESA'10}
}

@inproceedings{DevanurHayes09,
author = {Nikhil R. Devenur and
Thomas P. Hayes},
title = {The adwords problem: online keyword matching with budgeted
bidders under random permutations},
booktitle = {EC},
year = {2009},
}

@article{doi:10.1287/moor.2013.0612,
author = {Molinaro, Marco and Ravi, R.},
title = {The Geometry of Online Packing Linear Programs},
journal = {Mathematics of Operations Research},
volume = {39},
number = {1},
pages = {46-59},
year = {2014},
doi = {10.1287/moor.2013.0612},
URL = { https://doi.org/10.1287/moor.2013.0612},
eprint = { https://doi.org/10.1287/moor.2013.0612}
}

@article{lulli2004branch,
  title={A branch-and-price algorithm for multistage stochastic integer programming with application to stochastic batch-sizing problems},
  author={Lulli, Guglielmo and Sen, Suvrajeet},
  journal={Management Science},
  volume={50},
  number={6},
  pages={786--796},
  year={2004},
  publisher={Informs}
}

@article{gade2016obtaining,
  title={Obtaining lower bounds from the progressive hedging algorithm for stochastic mixed-integer programs},
  author={Gade, Dinakar and Hackebeil, Gabriel and Ryan, Sarah M and Watson, Jean-Paul and Wets, Roger J-B and Woodruff, David L},
  journal={Mathematical Programming},
  volume={157},
  pages={47--67},
  year={2016},
  publisher={Springer}
}

@article{guan2009cutting,
  title={Cutting planes for multistage stochastic integer programs},
  author={Guan, Yongpei and Ahmed, Shabbir and Nemhauser, George L},
  journal={Operations research},
  volume={57},
  number={2},
  pages={287--298},
  year={2009},
  publisher={INFORMS}
}

@article{kuhn2011primal,
  title={Primal and dual linear decision rules in stochastic and robust optimization},
  author={Kuhn, Daniel and Wiesemann, Wolfram and Georghiou, Angelos},
  journal={Mathematical Programming},
  volume={130},
  pages={177--209},
  year={2011},
  publisher={Springer}
}

@article{bodur2022two,
  title={Two-stage linear decision rules for multi-stage stochastic programming},
  author={Bodur, Merve and Luedtke, James R},
  journal={Mathematical Programming},
  volume={191},
  number={1},
  pages={347--380},
  year={2022},
  publisher={Springer}
}

@article{daryalal2024lagrangian,
  title={Lagrangian dual decision rules for multistage stochastic mixed-integer programming},
  author={Daryalal, Maryam and Bodur, Merve and Luedtke, James R},
  journal={Operations Research},
  volume={72},
  number={2},
  pages={717--737},
  year={2024},
  publisher={INFORMS}
}

@article{lohndorf2013optimizing,
  title={Optimizing trading decisions for hydro storage systems using approximate dual dynamic programming},
  author={L{\"o}hndorf, Nils and Wozabal, David and Minner, Stefan},
  journal={Operations Research},
  volume={61},
  number={4},
  pages={810--823},
  year={2013},
  publisher={INFORMS}
}

@article{philpott2020midas,
  title={Midas: A mixed integer dynamic approximation scheme},
  author={Philpott, Andrew B and Wahid, Faisal and Bonnans, J Fr{\'e}d{\'e}ric},
  journal={Mathematical Programming},
  volume={181},
  number={1},
  pages={19--50},
  year={2020},
  publisher={Springer}
}

@article{zou2019stochastic,
  title={Stochastic dual dynamic integer programming},
  author={Zou, Jikai and Ahmed, Shabbir and Sun, Xu Andy},
  journal={Mathematical Programming},
  volume={175},
  pages={461--502},
  year={2019},
  publisher={Springer}
}

@article{pereira1991multi,
  title={Multi-stage stochastic optimization applied to energy planning},
  author={Pereira, Mario VF and Pinto, Leontina MVG},
  journal={Mathematical programming},
  volume={52},
  pages={359--375},
  year={1991},
  publisher={Springer}
}

@article{gupta_interior-point-based_2020,
    title = {Interior-{Point}-{Based} {Online} {Stochastic} {Bin} {Packing}},
    volume = {68},
    issn = {0030-364X},
    url = {https://pubsonline.informs.org/doi/10.1287/opre.2019.1914},
    doi = {10.1287/opre.2019.1914},
    number = {5},
    urldate = {2025-05-03},
    journal = {Operations Research},
    author = {Gupta, Varun and Radovanović, Ana},
    month = sep,
    year = {2020},
    note = {Publisher: INFORMS},
    pages = {1474--1492},
}

@article{balseiro_best_2023,
    title = {The {Best} of {Many} {Worlds}: {Dual} {Mirror} {Descent} for {Online} {Allocation} {Problems}},
    volume = {71},
    issn = {0030-364X},
    shorttitle = {The {Best} of {Many} {Worlds}},
    url = {https://pubsonline.informs.org/doi/10.1287/opre.2021.2242},
    doi = {10.1287/opre.2021.2242},
    number = {1},
    urldate = {2025-11-07},
    journal = {Operations Research},
    author = {Balseiro, Santiago R. and Lu, Haihao and Mirrokni, Vahab},
    month = jan,
    year = {2023},
    note = {Publisher: INFORMS},
    pages = {101--119},
}

\newpage
\ECSwitch

\ECHead{Online Rack Placement in Large-Scale Data Centers\\Electronic Companion}

\section{Multi-stage Stochastic Integer Programming (MSSIP)}
\label{app:MSSIP}

This paper considers multi-stage stochastic (mixed-)integer programs with a separable objective function as the sum of per-period functions $f^t\left(\bx^t, \bxi^{1:t}\right)$; additive constraints over the decision variables (\Cref{eq:mssip_feasibleset}), and uncertain parameters following an exogenous, time-independent and history-independent distribution $\calD$.

\subsection{Benchmark algorithms}
\label{app:MSSIP_algs}

We compare the OSO algorithm to the following benchmarks to solve \Cref{eq:mssip}:
\begin{itemize}
	\item[--] \emph{Certainty-Equivalent (CE) resolving heuristic}. The CE algorithm proceeds as single-sample OSO, except that it replaces uncertain parameters by their mean. Since our uncertainty is i.i.d. over time, this results in a single sample path equal to the mean $\overline{\bxi}^{t+1:T} := (\overline{\bxi}, \dots, \overline{\bxi})$. Again, the solution is re-optimized dynamically (Algorithm~\ref{alg:CE}).

\begin{algorithm}[h!]
    \caption{Certainty-Equivalent (CE) algorithm.}\label{alg:CE}
    \begin{algorithmic} 
    \item Repeat, for $t=1,\dots,T$:
    \begin{itemize}
    \item[] \textbf{Observe:} Observe realization $\bxi^t$.
    \item[] \textbf{Compute mean:} Compute the mean ${ \overline{\bxi}_{t+1:T}} = (\overline{\bxi}, \dots, \overline{\bxi})$ of the distribution $\calD$.
    \item[] \textbf{Optimize:} Solve the following problem; store optimal solution $(\widetilde{\bx}^t, \widetilde{\bx}^{t+1}, \dots, \widetilde{\bx}^T)$:
        \begin{subequations}
        \label{eq:mssip_ce}
        \begin{alignat}{2}
            \min \quad 
            & f^t\left(\bx^t, \bxi^{1:t}\right)
            + \sum_{\tau = t+1}^T 
                f^{\tau} \left( 
                    \bx^\tau, 
                    \left(\bxi^{1:t}, \overline{\bxi}^{t+1:\tau}\right)
                \right)
                + \Psi(\bx^t)
            \\
            \st \quad 
            & \bx^t \in \calF_t\left(\overline{\bx}^{1:t-1}, \bxi^{1:t}\right) \\
            & \bx^\tau \in \calF_{\tau} \left( 
                \left(\overline{\bx}^{1:t-1}, \bx^{t:\tau-1}\right), 
                \left(\bxi^{1:t}, \overline{\bxi}^{t+1:\tau}\right) \right) 
            &\quad& \forall \ \tau \in \{t+1, \dots, T\} 
        \end{alignat}
        \end{subequations}
    \item[] \textbf{Implement:} Implement $\overline{\bx}^t = \widetilde{\bx}^t$, discarding $(\widetilde{\bx}^{t+1}, \dots, \widetilde{\bx}^T)$.
    \end{itemize}
\end{algorithmic}
\end{algorithm}

	\item[--] \emph{Myopic benchmark (Myo)}. This benchmark optimizes the immediate decision only without anticipating future uncertainty realizations, future decisions and future costs.

\begin{algorithm}[h!]
    \caption{Myopic (Myo) benchmark.}\label{alg:Myopic}
    \begin{algorithmic} 
    \item Repeat, for $t=1,\dots,T$:
    \begin{itemize}
    \item[] \textbf{Observe:} Observe realization $\bxi^t$.
    \item[] \textbf{Optimize:} Solve the following problem; store optimal solution $\widetilde\bx^t$:
        \begin{subequations}
        \label{eq:mssip_myo}
        \begin{alignat}{2}
            \min \quad 
            & f^t\left(\bx^t, \bxi^{1:t}\right)
            + \Psi(\bx^t)
            \\
            \st \quad 
            & \bx^t \in \calF_t\left(\overline{\bx}^{1:t-1}, \bxi^{1:t}\right)
        \end{alignat}
        \end{subequations}
    \item[] \textbf{Implement:} Implement $\overline{\bx}^t = \widetilde\bx^t$.
    \end{itemize}
\end{algorithmic}
\end{algorithm}

	\item[--] \emph{Hindsight-optimal benchmark (HO)}. This benchmark solves a deterministic optimization problem assuming full knowledge of uncertain parameters.

\begin{algorithm}[h!]
    \caption{Hindsight-optimal (HO) benchmark.}\label{alg:hindsight}
    \begin{algorithmic} 
    \item \textbf{Observe:} Observe the true realizations ${\bxi^1, \dots, \bxi^T}$.
    \item[] \textbf{Optimize:} Solve the following problem; store optimal solution $(\widetilde\bx^1, \dots, \widetilde\bx^T)$:
    \begin{subequations}
    \label{eq:mssip_ho}
    \begin{alignat}{2}
        \min \quad 
        & \sum_{t \in \calT} f^t\left(\bx^t, \bxi^{1:t}\right)
        \\
        \st \quad
        & \bx^t \in \calF_t\left(\bx^{1:t-1}, \bxi^{1:t}\right)
        &\quad& \forall \ t \in \{1, \dots, T\}
    \end{alignat}
    \end{subequations}
    \item Repeat, for $t=1,\dots,T$:
    \begin{itemize}
    \item[] \textbf{Implement:} Implement $\overline{\bx}^t = \widetilde\bx^t$.
    \end{itemize}
\end{algorithmic}
\end{algorithm}

\end{itemize}
\section{Online Resource Allocation}
\label{app:resourcealloc}

\subsection{Problem Statement and MSSIP Formulation}
\label{app:mssip_resourcealloc}

We first recall the definition of our online resource allocation problem:

\begin{definition}[Online resource allocation]
    Items arrive one at a time, indexed by $t = \{1, \dots, T\}$. There are $m$ supply nodes and $d$ resources, denoted by $\calJ$ and $\calK$ respectively. Each resource $k$ has capacity $b_k$. Each item is assigned to at most one supply node $j \in \calJ$; the assignment of item $i$ to supply node $j$ yields a reward $r^t_j$ and consumes $A^t_{jk}$ units of resource $k$. For item $t$, we can define the vector of actions $\bx^t = \{x^t_j\}_{j \in \calJ}$, the vector of rewards $\br^t = \{ r^t_j \}_{j \in \calJ}$ and the resource consumption matrix $\bA^t = \{ A^t_{jk} \}_{j \in \calJ, k \in \calK}$. For each item $t$, the unknown parameters $ \bXi^t = (\br^t, \bA^t)$ are i.i.d. according to distribution $\calD$ (with realizations $\bxi^t$). Assignments are irrevocable, and the decision-maker wishes to maximize total assignment reward subject to resource constraints.
\end{definition}

We denote by $\Delta_I = \Set{\bx \in \{0, 1\}^m | \sum_{j \in \calJ} x_j \leq 1}$ the action space at time $t$, which is a discrete simplex. Denoting the vector of resources by $\bb \in \mathbb{R}^d$, the offline problem can be written as:
\begin{subequations}
\begin{alignat}{2}
    \OPT \ = \ 
    \max \quad 
    & \sum_{t=1}^T \br^t \cdot \bx^t
    \\
    \text{s.t.} \quad 
    & \sum_{t=1}^T (\bA^t)^\top \bx^t \leq \bb
    \\
    & \bx^t \in \Delta_I 
    &\quad& \forall \ t \in \{1, \dots, T\}
\end{alignat}
\end{subequations}

Correspondingly, we can write the multistage stochastic program as follows. The per-period objective functions and feasible sets are:
\begin{alignat}{2}
    \label{eq:mssip_resourcealloc_objective}
    f^t(\bx^t, \bxi^t)
    & = \br^t \cdot \bx^t
    \\
    \label{eq:mssip_resourcealloc_feasibleset}
    \calF_t(\bx^{1:t-1}, \bxi^{1:t})
    & = \Set{
        \bx^t \in \Delta_I
        | 
        \sum_{\tau=1}^{t-1} (\bA^\tau)^\top \bx^\tau 
        + (\bA^t)^\top \bx^t 
        \leq \bb
    }
\end{alignat}
The multi-stage stochastic program can be expressed as follows:
\begin{alignat}{2}
\label{eq:mssip_resource_allocation}
    \mathbb{E}_{\bxi^1}
    \biggl[
        \max_{\bx^1 \in \calF_1(\bxi^{1})} \biggl\{ 
            f^1(\bx^1, \bxi^1)
            & + 
            \mathbb{E}_{\bxi^2}
            \biggl[
                \max_{\bx^2 \in \calF_2(\bx^{1}, \bxi^{1:2})} \biggl\{ 
                    f^2(\bx^2, \bxi^2)
                    + \dots
                    \\
                    & + \condE{\bxi^{T}}{
                        \max_{\bx^T \in \calF_T(\bx^{1:T-1}, \bxi^{1:T})} 
                        \Bigl\{ f^T(\bx^T, \bxi^T)\Bigr\}
                    } 
                \dots
                \biggr\}
            \biggr]
        \biggr\} 
    \biggr]
    \nonumber
\end{alignat}

The single-sample OSO algorithm (``OSO'' henceforth) for the online resource allocation problem is given in \Cref{alg:OSO1_resourcealloc}. The algorithm relies on one sample path at each iteration $t$, denoted by $\widetilde{\bxi}^{t+1:T}_t = (\widetilde{\bxi}^{t+1}_t, \dots, \widetilde{\bxi}^{T}_t)$. By construction, it returns a feasible solution. We then proceed to prove optimality guarantees in \Cref{thm:resourceallocation}.

\begin{algorithm}[h!]
    \caption{Single-sample OSO algorithm for the online resource allocation. problem}\label{alg:OSO1_resourcealloc}
    
    \begin{algorithmic} 
    \item Repeat, for $t = 1, \dots, T$:
    \begin{itemize}
    \item[] \textbf{Observe:} Observe realization $\bxi^t = (\br^t, \bA^t)$.
    \item[] \textbf{Sample:} Choose $\textit{one}$ sample path $\widetilde{\bxi}^{t+1:T}_t = (\widetilde{\bxi}^{t+1}_t, \dots, \widetilde{\bxi}^{T}_t)$ from distribution $\calD$.
    \item[] \textbf{Optimize:} Solve the following problem; store optimal solution $(\widetilde{\bx}^t, \widetilde{\bx}^{t+1}, \dots, \widetilde{\bx}^T)$:
        \begin{subequations}
        \label{eq:OSO1_resourcealloc}
        \begin{alignat}{2}
            \max \quad 
            & \br^t \cdot \bx^t
            + \sum_{\tau = t+1}^{T} \widetilde{\br}^\tau_t \cdot \bx^\tau
            \\
            \text{s.t.} \quad 
            & \sum_{\tau = 1}^{t-1} (\bA^\tau)^\top \overline{\bx}^\tau
            + (\bA^t)^\top \bx^t
            + \sum_{\tau = t+1}^T (\widetilde{\bA}^\tau_t)^\top \bx^\tau
            \leq \bb
            \\
            & \bx^\tau \in \Delta_I
            &\quad& \forall \ t \in \{t, \dots, T\}
        \end{alignat}
        \end{subequations}
    \item[] \textbf{Implement:} Implement $\overline{\bx}^t = \widetilde{\bx}^t$, discarding $(\widetilde{\bx}^{t+1}, \dots, \widetilde{\bx}^T)$.
    \end{itemize}
\end{algorithmic}
\end{algorithm}

\subsection{Proof of \Cref{thm:resourceallocation}}
\label{app:resourcealloc_proofs}

\subsubsection{Preliminaries}\hfill

To formally state our guarantees, we need the definition of an \emph{equivariant} solver:
\begin{definition}
    An integer program solver is \emph{equivariant} if, when we permute the items, the solution is permuted the same way: if it returns $(\bx^1, \dots, \bx^T)$ as an optimal solution to 
    $$
    \max_{\bx^1, \dots, \bx^T \in \Delta_I} \Set{ 
        \sum_{t=1}^T \br^t \cdot \bx^t
        | 
        \sum_{t=1}^T (\bA^t)^\top \bx^t \leq \bb
    },
    $$
    any permutation $\pi$ of $\{1, \dots, T\}$ gives $( \bx^{\pi(1)}, \dots, \bx^{\pi(T)} )$ as an optimal solution to:
    $$
    \max_{\bx^1, \dots, \bx^T \in \Delta_I} \Set{ 
        \sum_{t=1}^T \br^{\pi(t)} \cdot \bx^{\pi(t)}
        | 
        \sum_{t=1}^T (\bA^{\pi(t)})^\top \bx^{\pi(t)} \leq \bb
    }.
    $$
\end{definition}
Any solver can be made equivariant by sorting items according to a pre-specified order (e.g., such that the tuplets $(\br^t, \bA^t)$ are in lexicographic order) and applying the inverse permutation.

We proceed to prove \Cref{thm:resourceallocation} as long as the algorithm uses an equivariant solver. We consider a discrete distribution $\calD$; the general case follows from standard approximation arguments.

Without loss of generality, we assume that $A^t_{jk} \in [0, 1]$ for all $t, j, k$ and $b_k = B$ for all $k$. Otherwise, this can be obtained by rescaling the rows. Indeed, let $B$ denotes the smallest resource capacity normalized to resource requirements, that is:
\begin{equation*}
    B
    =
    \min_{k\in\calK} \left \{ \frac{b_k}{\max_{t\in\calT,j\in\calJ} A^t_{jk}} \right\}
\end{equation*}

Then, we can rewrite the constraint
$$\sum_{t \in \calT}\sum_{j \in \calJ} A^t_{jk} x^t_{j} \leq b_k$$
as:
$$\sum_{t \in \calT}\sum_{j \in \calJ} \widetilde{A}^t_{jk} x^t_{j} \leq B,$$
with
$$
    \widetilde{A}^t_{jk}
    =
    \frac{B}{b_k}A^t_{jk}
    \leq
    \frac{A^t_{jk}}{\max_{t\in\calT,j\in\calJ} A^t_{jk}}
    \leq 1,
    \quad 
    \ \forall \ t \in \calT, 
    \ j \in \calJ,
    \ k \in \calK.
$$

By assumption, $B \geq 1024 \cdot \log\left(\frac{2 d T}{\e}\right) \cdot \frac{\log^3 (1/\e)}{\e^2}.$

Let us perform the change of variables $\overline{\e} = \frac{\e}{8 \log^{1.5} (1/\e)}$. We note three properties:
\begin{itemize}
    \item[--] Since $8 \e \log^{1.5}(1/\e) \leq 1 < 2dT$ for all $\e \leq 0.001$:
    \begin{alignat}{3}
        \nonumber
        &\quad& 
        2dT & \geq 8 \e \log^{1.5}(1/\e)
        = \frac{\e^2}{\overline{\e}}
        &\quad& \text{(by definition of $\overline{\e}$)}
        \\
        \nonumber
        \iff 
        &\quad& 
        \left(\frac{2dT}{\e}\right)^2
        & \geq \frac{2dT}{\overline{\e}}
        \\
        \label{eq:eps_bar1}
        \iff 
        &\quad& 
        2 \cdot \log \left(\frac{2dT}{\e}\right)
        & \geq \log \left(\frac{2dT}{\overline{\e}}\right)
    \end{alignat}
    \item[--] For all $\e \leq 0.0001$:
    \begin{alignat}{3}
        \nonumber
        &\quad& 
        \log \left( 
            \frac{16 \log^{1.5}(1/\e)}{\e}
        \right)
        & \leq \log^{1.5}(1/\e)
        \\
        \nonumber
        \iff 
        &\quad& 
        \log (2 / \overline{\e})
        & \leq \log^{1.5}( 1 / \e)
        &\quad& \text{(by definition of $\overline{\e}$)}
        \\
        \label{eq:eps_bar2}
        \iff 
        &\quad& 
        8 \overline{\e} \log (2 / \overline{\e})
        & \leq 8 \overline{\e} \log^{1.5}( 1 / \e) 
        = \e 
        &\quad& \text{(by definition of $\overline{\e}$)}
    \end{alignat}
    \item[--] By definition of $\overline{\e}$, $B\geq 16 \cdot \log\left(\frac{2 d T}{\e}\right) \cdot \frac{1}{\overline{\e}^2}$, hence, per \Cref{eq:eps_bar1}, $B \geq 8 \cdot \log\left(\frac{2 d T}{\overline{\e}}\right) \cdot \frac{1}{\overline{\e}^2}$.
\end{itemize}

\subsubsection{\Cref{lemma:sampleOpt}: offline optimum scales with time and budget.}\hfill

\begin{definition}
    Let $\OPT(s, \bb')$ be the optimum of the subproblem which considers the first $s$ items and a resource capacity $\bb' \in \mathbb{R}^d$:
    \begin{align}
        \label{eq:samplePack}
        \OPT(s, \bb') \ = \ 
        \max_{\bx^1, \dots, \bx^s \in \Delta_I} \Set{ 
            \sum_{t=1}^s \br^t \cdot \bx^t
            | 
            \sum_{t=1}^s (\bA^t)^\top \bx^t \leq \bb'
        },
    \end{align}
\end{definition}

Note that $\OPT(s, \bb')$ is a random variable and that $\OPT(T,B \bone)$ is equal to the original problem (where $\bone$ denotes the $d$-dimensional vector of ones). \cite{GM-MOR16} prove that $\OPT(s, \bb')$ scales with the fraction of timesteps $\frac{s}{T}$ and with the smallest fraction of budget available $\min_{k} \frac{b'_k}{B}$ in the more restrictive case where $\bA^t$ is a $1 \times d$ vector (i.e. there is only one supply node). We extend this result in \Cref{lemma:sampleOpt} in the case where $\bA^t$ is a $m \times d$ matrix, namely $\E{\OPT(s, \bb')} \approx \min\left\{\frac{s}{T}, \min_k \frac{b'_k}{B}\right\} \cdot \E{\OPT}$.

\begin{lemma} 
    \label{lemma:sampleOpt}
    Let $\zeta = \min \left\{ \frac{s}{T}, \min_k \frac{b'_k}{B} \right\}$. If $s \geq \overline{\e}^2 T$ and $\min_k b'_k \geq \overline{\e}^2 B$, then:
    \begin{equation}
        \E{\OPT(s, \bb')} 
        \geq \left[
            \zeta \left(
                1 - \overline{\e} \sqrt{\frac{1}{\zeta}} 
            \right) 
            - \frac{1 + \overline{\e}}{T}
        \right]
        \cdot \E{\OPT}
    \end{equation}
\end{lemma} 

\begin{proof}{Proof of \Cref{lemma:sampleOpt}.}
Without loss of generality, all coordinates of $\bb'$ are identical, so $\bb' = B' \bone$ and $B' = \min_k b'_k$, and $\frac{B'}{B} = \min_k \frac{b'_k}{B} \leq \frac{s}{T}$. This can be obtained by rescaling the rows.

Let $\bx^{\star} = (\bx^{\star,1}, \dots, \bx^{\star,T})$ be an optimal solution to the offline problem $\OPT(T, B \bone)$ by an equivariant solver. Also, let $\mset$ be the \emph{set} of the realizations $\{\bxi^1, \dots, \bxi^T\}$ of the problem. The definition of $\mset$ ignores the order, so conditioning on $\mset$ leaves the order of $\{\bxi^1, \dots, \bxi^T\}$ uniformly random; in particular, conditioned on $\mset$, the sequence of random variables $\br^1, \dots, \br^T$ is exchangeable, so the distribution of $(\br^{\pi(1)}, \dots, \br^{\pi(T)})$ is the same for every permutation $\pi$ of $\{1, \dots, T\}$. Due to the equivariance of the solution $\bx^{\star}$, the distribution of $ \br^t \cdot \bx^{\star,t}$ is the same for all $t \in \{1, \dots, T\}$, conditioned on $\mset$; in particular, at time $t$, the contribution to the optimal solution is
\begin{align}
    \label{eq:valueEx}
    \Emid{\br^t \cdot \bx^{\star,t}}{\mset} 
    \ = \
    \frac{1}{T} \cdot 
    \Emid{\sum_{\tau = 1}^T \br^\tau \cdot \bx^{\star,\tau}}{\mset} 
    \ = \ 
    \frac{\Emid{\OPT}{\mset}}{T}.  
\end{align}

Informally, the proof of the lemma relies on the observation that, for $\widetilde{s}$ slightly smaller than $\zeta T$, the solution truncated to its first $\widetilde{s}$ elements is a feasible solution with high probability to the problem given in \Cref{eq:samplePack}, and yields an expected value $\widetilde{s}\frac{\E{\OPT}}{T} \approx \zeta \E{\OPT}$. This will give that $\OPT(s,\bb')$ is larger than but close to $\zeta \E{\OPT}$. Let us formalize these arguments.

Let $\widetilde{B} = B' \left(1 - \overline{\e} \sqrt{B/B'}\right)$, and fix $\widetilde{s}$ to the integer in the interval $\left(\frac{T \widetilde{B}}{B}-1, \frac{T \widetilde{B}}{B}\right]$. In particular, since $\widetilde{B} \leq B' \leq \frac{s B}{T}$, we have that $\widetilde{s} \leq s$. We first show that the solution $\widehat{\bx} = (\bx^{\star,1}, \dots, \bx^{\star,\widetilde{s}}, \bo, \dots, \bo) \in \Delta_I^{m}$ is feasible with high probability for \Cref{eq:samplePack}. Again, conditioned on $\mset$, the sequence of random variables $\bA^1, \dots, \bA^T$ is exchangeable.

Due to the equivariance of $\bx^\star$, it follows that for each resource $k$, the sequence of random variables $
    \big\{ 
    \, (\bA^1)^\top \bx^{\star,1}, 
    \ \dots, 
    \ (\bA^T)^\top \bx^{\star,T}
\big\}$ is also exchangeable. From the feasibility of $\bx^\star$ for $\OPT(T, B \bone)$, $\sum_{t=1}^T (\bA^t)^\top \bx^{\star,t} \leq B \bone$ in every scenario. From the concentration inequality for exchangeable sequences (\Cref{cor:exBernstein} in \ref{app:concentration}), we obtain the following inequality for each resource $k$ (using $\tau = \overline{\e} \sqrt{B B'}$ and $M = B$):
\begin{align}
    \nonumber
	\Pmid{
        \sum_{t=1}^T \sum_{j \in \calJ} 
        A^t_{jk} \widehat{x}^t_j 
        \geq B'
    }{\mset}
    & = \Pmid{
        \sum_{t=1}^T \sum_{j \in \calJ} 
        A^t_{jk} \widehat{x}^t_j 
        \geq \widetilde{B} + \overline{\e} \sqrt{BB'}
    }{\mset}
    \\
    \nonumber
    & \leq \Pmid{
        \sum_{t=1}^T \sum_{j \in \calJ} 
        A^t_{jk} \widehat{x}^t_j 
        \geq 
        \frac{\widetilde{s} B}{T} + \overline{\e} \sqrt{B B'}
    }{\mset}
    \\
    \label{eq:sampleOpt}
	& \leq
    2 \exp{
        - \min\left\{ 
            \frac{\overline{\e}^2 B' T}{8 \widetilde{s}}, 
            \ \frac{\overline{\e} \sqrt{B B'}}{2} 
        \right\} 
    }.
\end{align}

To upper bound the right-hand side, we use $B \geq \frac{8}{\overline{\e}^2} \log \left(\frac{2dT}{\overline{\e}}\right)$, $\widetilde B\leq B'$ and $\widetilde{s}\leq \frac{T \widetilde{B}}{B}$ to obtain:
\begin{equation}
    \frac{\overline{\e}^2 B' T}{8 \widetilde{s}} 
    \geq \frac{\overline{\e}^2 B' B}{8 \widetilde{B}} 
    \geq \frac{\overline{\e}^2 B}{8} 
    \geq \log \left(\frac{2dT}{\overline{\e}}\right).
\end{equation}

Moreover, the assumption that $B' \geq \overline{\e}^2 B$ implies:
\begin{equation}
    \frac{\overline{\e} \sqrt{B B'}}{2} 
    \geq \frac{\overline{\e}^2 B}{2} 
    \geq 4 \log \left(\frac{2dT}{\overline{\e}}\right).
\end{equation}

Thus, the solution violates the resource $k$ constraint of $(\OPT(s, b'))$ with probability at most $\frac{\overline{\e}}{dT}$:
\begin{equation}
	\forall \ k \in \calK: 
    \quad 
    \Pmid{
        \sum_{t=1}^T \sum_{j \in \calJ} A^t_{jk} \widehat{x}^t_j \geq B'
    }{\mset} 
    \, \leq \, 
    \frac{\overline{\e}}{dT}.
\end{equation}

Taking a union bound over all $d$ constraints, the solution is feasible with high probability:
\begin{equation}
	\Pmid{
        \sum_{t=1}^T (\bA^t)^\top \widehat{\bx}^t 
        \leq \bb'
    }{\mset} 
    \, \geq \, 
    1 - \frac{\overline{\e}}{T}.
\end{equation}

Let $G$ be the good event that this feasibility condition holds (and $G^c$ be its complement). Under this event, $\OPT(s, \bb')$ is at least equal to $\sum_{t=1}^{\widetilde{s}} \br^t \cdot \widehat{\bx}^t$. We obtain:
\begin{align}
    \nonumber
    \Emid{\OPT(s, \bb')}{\mset} 
    &
    \, \geq \, 
    \Emid{
        \sum_{t=1}^{\widetilde{s}} 
        \br^t \cdot \widehat{\bx}^t
    }{G \textrm{ and } \mset } \cdot \Pmid{G}{\mset} 
    \\
    \nonumber
    &
    \, = \, 
    \Emid{
        \sum_{t=1}^{\widetilde{s}} 
        \br^t \cdot \widehat{\bx}^t
    }{\mset} 
    - \Emid{ 
        \sum_{t=1}^{\widetilde{s}} 
        \br^t \cdot \widehat{\bx}^t
    }{G^c \textrm{ and } \mset} 
    \cdot \Pmid{G^c}{\mset} 
    \\
    \label{eq:condSet}
    &
    \, \geq \, 
    \Emid{
        \sum_{t=1}^{\widetilde{s}} 
        \br^t \cdot \widehat{\bx}^t
    }{\mset} 
    - \Emid{\OPT}{G^c \textrm{ and } \mset} \cdot \Pmid{G^c}{\mset}, 
\end{align}
where the last inequality stems from the feasibility of $\widehat{\bx}$ in the full problem.
            
To bound the first term, recall that $\br^t \cdot \widehat{\bx}^t = \br^t \cdot \bx^{\star,t}$ for $t \leq \tilde{s}$, so 
\Cref{eq:valueEx} implies $\Emid{\br^t \cdot \widehat{\bx}^t}{\mset} = \displaystyle \frac{\Emid{\OPT}{\mset}}{T}$. 
For the second term, notice that conditioning on $\mset$ fixes the items in the instance, hence the optimum $\OPT$, so further conditioning on $G^c$ has no effect. Thus:
\begin{align}
    \Emid{\OPT(s, \bb')}{\mset} 
    &
    \, \geq \, 
    \frac{\widetilde{s}}{T} \, \Emid{\OPT}{\mset} \,-\, \frac{\overline{\e}}{T}\,\Emid{\OPT}{\mset} \\
    &
    \, \geq \, 
    \left[\frac{\widetilde{B}}{B} - \frac{1}{T} - \frac{\overline{\e}}{T}\right] \cdot \Emid{\OPT}{\mset}\\
    &
    \, = \, 
    \left[\frac{B'}{B} \left(1 - \overline{\e} \sqrt{\frac{B}{B'}} \right) - \frac{1+\overline{\e}}{T}\right] \cdot \Emid{\OPT}{\mset} \\
    &
    \, \geq \, 
    \left[\zeta \left(1 - \overline{\e} \sqrt{\frac{1}{\zeta}} \right) - \frac{1+\overline{\e}}{T}\right] \cdot \Emid{\OPT}{\mset},
\end{align} 
where the last inequality uses that $\zeta \geq \overline{\e}^2$, and the function $x \mapsto x (1 - \overline{\e} \sqrt{1/x})$ is increasing for $x \geq \overline{\e}^2$. Taking expectation with respect to $\mset$ concludes the proof of the lemma. \myqed
\end{proof}

\subsubsection{\Cref{lemma:S,lemma:concS}: Resource consumption scales with time}\hfill

We next show that resources consumption scales with time and the resource capacity. This is, by time $t$, the OSO algorithm utilizes approximately a fraction $\frac{t}{T}$ of the resource budget for each resource. We formalize this via the following definitions:

\begin{definition}[Occupancy vector]
    The occupancy vector of resources consumed by OSO at time $t$ is defined as follows, where $\overline{\bx}^t$ is the decision implemented by OSO at time $t$.
    \begin{equation}
        \bS^t = \sum_{\tau=1}^t (\bA^\tau)^\top \overline{\bx}^\tau \in \mathbb{R}^{d}
    \end{equation}
\end{definition}

\begin{definition}
    We denote by $\calH_t$ the $\sigma$-algebra generated by the history of the OSO algorithm up to time $t$, i.e., the demand realization $\bxi^\tau = (\br^\tau, \bA^\tau)$ and the sample path $\widetilde{\bxi}^{\tau+1:T}_\tau$ for $\tau \in \{1, \dots, t\}$. We can condition on $\calH_t$, giving the expectation operator $\condE{\calH_t}{\cdot}$ (denoted by $\condE{t}{\cdot}$ for simplicity). 
\end{definition}

We show that by time $t$, the algorithm utilizes approximately a fraction $\frac{t}{T}$ of the overall budget, i.e., $\bS^t$ is less than but close to $\frac{t}{T} B \bone$ componentwise. In fact, we prove the following stronger result, which will be later used to show that $\bS^t$ is concentrated around its mean:
\begin{lemma} 
    \label{lemma:S}
    For every $t \geq 1$, we have 
    \begin{align}
        \label{eq:expSeq0}
        \condE{t-1}{\bS^t} \leq \left(1 - \frac{1}{T-t+1}\right) \bS^{t-1} + \frac{B \bone}{T-t+1}.   
    \end{align}
    In particular, $\E{\bS^t} \leq \frac{t}{T} B \bone.$
\end{lemma}

\begin{proof}{Proof of \Cref{lemma:S}.}
By construction, OSO yields a feasible solution :
\begin{align}
    \label{eq:expSeq1}
    \bS^{t-1} 
    + (\bA^t)^\top \overline{\bx}^t 
    + (\widetilde{\bA}^{t+1}_{t})^\top \widetilde{\bx}^{t+1}
    + \dots
    + (\widetilde{\bA}^{T
}_{t})^\top \widetilde{\bx}^T
    \leq B \bone.   
\end{align}
Even conditioned on the history up to time $t-1$, the matrices $\bA^t, \widetilde{\bA}^{t+1}_{t}, \dots, \widetilde{\bA}^{T}_{t}$ are sampled i.i.d., and the solution $(\overline{\bx}^t = \widetilde{\bx}^t, \widetilde{\bx}^{t+1}, \dots, \widetilde{\bx}^T)$ is by assumption equivariant. Therefore, conditioned on the history up to time $t-1$, the sequence of vectors $
    \big\{ 
        \,
        (\bA^t)^\top \overline{\bx}^t, 
        \ (\widetilde{\bA}^{t+1}_t)^\top \widetilde{\bx}^{t+1}, 
        \ \dots, 
        \ (\widetilde{\bA}^{T}_t)^\top \widetilde{\bx}^{T}
        \,
    \big\}
$
is again an exchangeable sequence of random variables with equal expectation, conditioned on $\calH_{t-1}$: $\condE{t-1}{ (\bA^t)^\top \overline{\bx}^t } = \condE{t-1}{ (\widetilde{\bA}^{\tau}_t)^\top \widetilde{\bx}^{\tau} }$ for all $\tau > t$. From \Cref{eq:expSeq1}, we therefore have:
\begin{equation}
    \bS^{t-1} + (T - t + 1) \cdot \condE{t-1}{ (\bA^t)^\top \overline{\bx}^t } \leq B \bone.  
\end{equation}

We obtain inequality in \Cref{eq:expSeq0}:
\begin{align}
    \condE{t-1}{ \bS^t } 
    & = 
    \bS^{t-1} + \condE{t-1}{ (\bA^t)^\top \overline{\bx}^t }
    \\
    & \leq \bS^{t-1} + \frac{ B\bone - \bS^{t-1}}{T-t+1}
    \\
    & = \left(1 - \frac{1}{T-t+1}\right) \bS^{t-1} + \frac{B \bone}{T-t+1}.
\end{align}
   	
We now prove that $\E{\bS^t} \leq \frac{t}{T} B \bone$ by induction on $t$. It clearly holds for $t=0$ (where we define $\bS^0 = \bo$ by convention). Assuming that it holds for $t-1$, we obtain:
\begin{align}
    \E{\bS^t} 
    = \E{\condE{t-1}{\bS^t}} 
    &\leq \left(1 - \frac{1}{T - t+1}\right) \E{\bS^{t-1}} +  \frac{B \bone}{T - t+1}\\
    &\leq \left(1 - \frac{1}{T - t+1}\right) \frac{t-1}{T} \cdot B \bone + \frac{B \bone}{T - t+1}\\
    &= \frac{t}{T} B \bone,
\end{align} 
where the first inequality follows from inequality in \Cref{eq:expSeq0} and the next inequality follows by the induction hypothesis. This concludes the proof.  \myqed
\end{proof}
   
While this lemma guarantees that the occupation vector $\E{\bS^t}$ is at most $\frac{t}{T} B \bone$ in expectation, we need high-probability guarantees. We derive them in the next lemma.
   
\begin{lemma} \label{lemma:concS}
    For each $t \geq \frac{\overline{\e}^2 T}{4}$, we have:
    \begin{equation}
        \Prob{
            \bS^t 
            \leq \left(1+ \overline{\e} \sqrt{\frac{T}{t}}\right) \frac{t}{T} B \bone
        }
        \geq 1 - \frac{\overline{\e}}{2T}.
    \end{equation}
\end{lemma}

\begin{proof}{Proof of \Cref{lemma:concS}.}
    We show that for every component $k \in \{1, \dots, d\}$, $S^t_k$ is concentrated around its expected value, namely that $S^t_k \leq \Big( 1 + \overline{\e} \sqrt{\frac{T}{t}} \Big) \frac{t}{T} B$ with probability at least $1-\frac{\overline{\e}}{2dT}$. However, since the increments $ (\bA^1)^\top \overline{\bx}^1, \dots, (\bA^t)^\top \overline{\bx}^t$ are not independent (e.g., $\overline{\bx}^\tau$ depends on $\bA^1, \dots, \bA^\tau$), we cannot use standard concentration inequalities. Instead, we rely on a concentration result for ``self-centering'' sequences, given in \Cref{thm:centeringConc} (see \ref{app:concentration}).

    Let us denote $\alpha_t = 1- \frac{1}{T-t+1}$ and $\beta_t = \frac{B}{T-t+1}$, and let $Y_t$ be the solution to the recurrence relation $y_t = \alpha_t y_{t-1} + \beta_t$ and $y_0 = 0$. From \Cref{thm:centeringConc}, we obtain, for any $\gamma \in (0,1]$:
    \begin{align}
        \label{eq:concS}
        \Prob{
            S^t_k \geq (1+2\gamma) Y_t 
        } \leq \exp{-\gamma^2 Y_t}.    
    \end{align}
    
    As in \Cref{lemma:S}, we prove by induction on $t$ that $Y_t = \frac{t}{T} B$: this holds for $t=0$, and then:
    \begin{align}
        Y_t = \left(1 - \frac{1}{T-t+1}\right) Y_{t-1} + \frac{B}{T-t+1} = \left(1  - \frac{1}{T-t+1}\right) \frac{t-1}{T} B + \frac{B}{T-t+1} = \frac{t}{T} B.
    \end{align}
    From \Cref{eq:concS} with $\gamma = \frac{\overline{\e}}{2} \sqrt{T/t}$ ($\gamma\leq1$ by assumption), we derive:
    \begin{align}
        \Prob{
            S^t_k \geq \left(1+\overline{\e} \sqrt{\frac{T}{t}} \right) \frac{t}{T} B 
        } 
        \leq \exp{-\frac{\overline{\e}^2}{4}B}.
    \end{align}
    Recall that, by assumption, $B \geq \frac{{8}}{\overline{\e}^2} \log\left(\frac{2dT}{\overline{\e}}\right)\geq\frac{{4}}{\overline{\e}^2} \log\left(\frac{2dT}{\overline{\e}}\right)$. This yields:
    \begin{align}
        \Prob{
            S^t_k \geq \left(1+\overline{\e} \sqrt{\frac{T}{t}} \right) \frac{t}{T} B 
        } 
        \leq \exp{- \log (2dT/\overline{\e})} = \frac{\overline{\e}}{2d T}.
    \end{align}
    Taking a union bound over all $d$ coordinates, we obtain:
    \begin{align}
        \Prob{
            \bS^t \leq \left(1+\overline{\e} \sqrt{\frac{T}{t}} \right) \frac{t}{T} B\bone
        } \geq 1-\frac{\overline{\e}}{2T}.
    \end{align}
    This concludes the proof of the lemma.  \myqed
\end{proof}

\subsubsection{\Cref{lemma:val}: Reward of OSO algorithm bounded below.}\hfill 

We now bound the reward of the OSO algorithm at time $t$, i.e. $\br^t \cdot \overline{\bx}^t$. From \Cref{lemma:concS}, there is about $\left(1 - \frac{t-1}{T}\right) B$ of the budget left in each of the constraints, and so the remaining value should be $\left(1-\frac{t-1}{T}\right) \E{\OPT}$. 
Moreover, since there are $T-t+1$ variables in the remaining problem, we expect that $\overline{\bx}^t$ accrues a value of $\frac{1}{T-t+1} \left(1 - \frac{t-1}{T}\right) \E{\OPT}$, or $\frac{1}{T} \E{\OPT}$. 
This is formalized below. 

\begin{lemma} 
    \label{lemma:val}
    For every $t$ satisfying $\overline{\e}^2 T \leq t \leq (1 - 2 \overline{\e}) T$  we have:
    \begin{equation}
        \E{ \br^t \cdot \overline{\bx}^t} 
        \geq \left[
            1 
            - \overline{\e} \sqrt{\frac{T}{(1-\overline{\e}) T - t}} 
            - 2\overline{\e} \frac{\sqrt{Tt}}{T-t} 
            - \frac{1+\overline{\e}}{T - t + 1}
        \right] \frac{\E{\OPT}}{T}.
    \end{equation}
\end{lemma}

\begin{proof}{Proof of \Cref{lemma:val}.}
    Fix $t$ such that $\overline{\e}^2 T \leq t \leq (1 - 2 \overline{\e})T$, and consider the solution $(\overline{\bx}^t, \widetilde{\bx}^{t+1}, \dots,\widetilde{\bx}_{T})$ obtained by the OSO algorithm at this time. By definition, we have:
    \begin{align}
        \br^t \cdot \overline{\bx}^t
        + \sum_{\tau = t+1}^T \widetilde{\br}^{\tau}_t \cdot \widetilde{\bx}^{\tau}
        = \max_{\bx^t, \dots, \bx^T \in \Delta_I} \Set{
            \br^t \cdot\bx^t
            + \sum_{\tau = t+1}^T \widetilde{\br}^{\tau}_t \cdot \bx^{\tau}
            |
            (\bA^t)^\top \bx^t 
            + \sum_{\tau = t+1}^T (\widetilde{\bA}^\tau_t)^\top \bx^\tau
            \leq 
            B \bone - \bS^{t-1}
        }
    \end{align}
    	
    Conditioning on the history $\calH_{t-1}$ fixes the occupation vector $\bS^{t-1}$, hence the right-hand side of the resource constraint. The expected value of the stochastic program is equal to $\E{\OPT(T - t + 1, B \bone - \bS^{t-1})}$. It comes:
    \begin{align}
        \label{eq:valuePack}
        \condE{t-1}{
            \br^t \cdot \overline{\bx}^t
            + \sum_{\tau = t+1}^T \widetilde{\br}^{\tau}_t \cdot \widetilde{\bx}^{\tau}
        } = \E{\OPT(T - t + 1, B \bone - \bS^{t-1})}.  
    \end{align}
    
    As earlier, conditioned on $\calH_{t-1}$ the random variables $\big\{ \br^t \cdot \overline{\bx}^t, \widetilde{\br}^{t+1}_t \cdot \widetilde{\bx}^{t+1}, \dots, \widetilde{\br}^{T}_t \cdot \widetilde{\bx}^T\}$ form an exchangeable sequence and thus have the same conditional expectations. Thus, all the terms on the left-hand side of \Cref{eq:valuePack} have the same expectation. In particular,
    \begin{align}
        \condE{t-1}{
            \br^t \cdot \overline{\bx}^t
        } 
        = \frac{1}{T - t + 1} \cdot \E{\OPT(T - t + 1, B \bone - \bS^{t-1})}.
    \end{align}
    	
    Now, let $\gamma_t = \overline{\e} \sqrt{\frac{T}{t}}$. 
    From \Cref{lemma:concS}, we know that, with probability at least $1 - \frac{\overline{\e}}{2T}$, we have:
    \begin{equation}
        \bS^{t-1} \leq (1 + \gamma_{t-1}) \frac{t-1}{T} B \bone
        \implies B \bone-\bS^{t-1}\geq\left(1-(1 + \gamma_{t-1}) \frac{t-1}{T}\right)B\bone.
    \end{equation}
    Denote the above good event happening by $G$. When $G$ transpires, letting $\zeta_t = \min\big\{\frac{T-t+1}{T}, 1 - (1 + \gamma_{t-1}) \frac{t-1}{T}\big\} = 1 - (1+ \gamma_{t-1}) \frac{t-1}{T}$, we have from \Cref{lemma:sampleOpt} that:
    \begin{equation}
        \E{\OPT(T - t + 1, B \bone - \bS^{t-1})} \,\geq\, \left[\zeta_t \left(1 - \overline{\e} \sqrt{\frac{1}{\zeta_t}}\right) - \frac{1+\overline{\e}}{T}\right] \cdot \E{\OPT}.
    \end{equation}
    Note that the assumptions in \Cref{lemma:sampleOpt} are met because $t \leq (1-2\overline{\e})T$ and $\overline{\e} \in (0,1]$.
    Then:
    \begin{align}
        \nonumber
        \E{\br^t \cdot \overline{\bx}^t} 
        & \geq 
        \Emid{ \br^t \cdot \overline{\bx}^t }{G} \cdot \Prob{G} 
        \\
        \nonumber
        & \geq
        \left(1 - \frac{\overline{\e}}{2T}\right) 
        \left[
            \zeta_t \left(1 - \overline{\e} \sqrt{\frac{1}{\zeta_t}}\right) 
            - \frac{1+\overline{\e}}{T}
        \right] \frac{\E{\OPT}}{T - t + 1} 
        \\
        \label{eq:valPacking1}
        & \geq 
        \left[
            \left(1 - \frac{\overline{\e}}{T}\right) 
            \left(1 - \overline{\e} \sqrt{\frac{1}{\zeta_t}} \right) 
            \frac{\zeta_t}{T - t + 1} 
            - \frac{1+\overline{\e}}{T (T - t + 1)}
        \right]\E{\OPT}.
    \end{align}
        
    To further lower bound the right-hand side, using the definitions of $\zeta_t$ and $\gamma_{t-1}$ we have
    \begin{align}
        \frac{\zeta_t}{T - t + 1}
        = 
        \frac{(T-t+1) - \gamma_{t-1} (t-1)}{T (T - t +1)}
        = 
        \frac{1}{T}\left(1 - \overline{\e} \frac{\sqrt{T} \sqrt{t-1}}{T-t+1}\right)
        \geq 
        \frac{1}{T} \left(1- \overline{\e} \frac{\sqrt{T} \sqrt{t}}{T-t}\right)
    \end{align}
    
    Substituting into \Cref{eq:valPacking1} and using $(1-a)(1-b) \geq 1-a - b$ for $a,b \geq 0$, we get:
    \begin{align}
        \E{ \br^t \cdot \overline{\bx}^t} 
        & \geq \left[
            \left(1 - \frac{\overline{\e}}{T}\right) 
            \left(1 - \overline{\e} \sqrt{\frac{1}{\zeta_t}} \right) 
            \left(1 - \overline{\e} \frac{\sqrt{T} \sqrt{t}}{T-t}\right) 
            - \frac{1 + \overline{\e}}{T - t + 1}
        \right] \frac{\E{\OPT}}{T}
        \nonumber
        \\
        & \geq \left[
            1 - \frac{\overline{\e}}{T} -  \overline{\e} \sqrt{\frac{1}{\zeta_t}} 
            - \overline{\e} \frac{\sqrt{T} \sqrt{t}}{T-t} 
            - \frac{1+\overline{\e}}{T - t + 1}
        \right] \frac{\E{\OPT}}{T}
        \nonumber
        \\
        & \geq \left[
            1 - \overline{\e} \sqrt{\frac{1}{\zeta_t}} 
            - 2\overline{\e} \frac{\sqrt{T} \sqrt{t}}{T-t} 
            - \frac{1+\overline{\e}}{T - t + 1}
        \right] \frac{\E{\OPT}}{T}.
    \end{align}
    
    We can complete the lower bound of this right-hand side using 
    \begin{align}
        \zeta_t 
        \,=\, 
        \frac{(T - t + 1)}{T} - \overline{\e} \sqrt{\frac{t-1}{T}} 
        \,\geq\, 
        \frac{(1-\overline{\e}) T - t}{T}
    \end{align}
    We obtain:
    \begin{align}
        \E{\br^t\cdot\overline{\bx}^t} 
        & \geq \left[
            1 - \overline{\e} \sqrt{\frac{T}{(1-\overline{\e}) T - t}} 
            - 2\overline{\e} \frac{\sqrt{T} \sqrt{t}}{T-t} 
            - \frac{1+\overline{\e}}{T - t + 1}
        \right] \frac{\E{\OPT}}{T}.
    \end{align}
    This concludes the proof of the lemma. \myqed
\end{proof}

\subsubsection{Proof of \Cref{thm:resourceallocation}.}\hfill

Let $\ALG$ be the value achieved by the OSO algorithm. From \Cref{lemma:val}, we have:
\begin{align}
    \label{eq:valuePackFinal}
    \E{\ALG}
    & \geq 
    \sum_{t = \overline{\e}^2 T}^{(1-2\overline{\e}) T - 1} 
    \left(
        1 
        - \overline{\e} \sqrt{\frac{T}{(1-\overline{\e}) T - t}} 
        - 2\overline{\e} \frac{\sqrt{T} \sqrt{t}}{T-t} 
        - \frac{1+\overline{\e}}{T-t+1} 
    \right) \frac{\E{\OPT}}{T}. 
\end{align}

Since the function $\sqrt{\frac{T}{(1-\overline{\e}) T - t}}$ is increasing in $t$, we can use an integral to upper bound the sum:
\begin{align}
    \sum_{t = \overline{\e}^2 T}^{(1-2\overline{\e}) T - 1} 
    \sqrt{\frac{T}{(1-\overline{\e}) T - t}}
    & \,\leq\, 
    \int_0^{(1-2\overline{\e}) T} 
    \sqrt{\frac{T}{(1-\overline{\e}) T - t}} \,\d t
    \nonumber
    \\
    & \,=\, 
    \sqrt{T} \int_{\overline{\e}T}^{(1-\overline{\e}) T} \frac{1}{\sqrt{y}} \,\d y 
    \nonumber
    \\
    & \,=\, 
    2\sqrt{T} \left( \sqrt{(1-\overline{\e}) T} - \sqrt{\overline{\e}T} \right)
    \nonumber
    \\
    & \,\leq\, 
    2 T.
\end{align}
	
Similarly, since the function $\frac{\sqrt{t}}{T-t}$ is increasing in $t$, we have:
\begin{align}
    \sum_{t = \overline{\e}^2 T}^{(1-2\overline{\e}) T - 1} 
    \frac{\sqrt{T} \sqrt{t}}{T-t}
    \,=\, 
    \sum_{t = \overline{\e}^2 T}^{(1-2\overline{\e}) T - 1} 
    \frac{\sqrt{t/T}}{1-(t/T)}
    \,\leq\,
    \int_0^{(1-2\overline{\e})T} \frac{\sqrt{t/T}}{1-t/T} \,\d t
    \,=\, 
    T \int_0^{1-2\overline{\e}} \frac{\sqrt{y}}{1-y} \,\d y.
\end{align}

Therefore,
\begin{align}
    \sum_{t = \overline{\e}^2 T}^{(1-2\overline{\e}) T - 1} 
    \frac{\sqrt{T} \sqrt{t}}{T-t}
    & \,\leq\, 
    T \cdot \left[
        - 2 \sqrt{y} 
        + \log \left(\frac{1 + \sqrt{y}}{1-\sqrt{y}} \right) 
    \right] ~\Bigg|_0^{1-2\overline{\e}}
    \nonumber
    \\
    & \,\leq\,
    T \log \left(\frac{1 + \sqrt{1-2\overline{\e}}}{1-\sqrt{1-2\overline{\e}}} \right)
    \,\leq\, 
    T \log \left(\frac{2}{1 - \sqrt{1 - 2 \overline{\e}}} \right)
    \nonumber
    \\
    & \,\leq\, 
    T \log(2/\overline{\e}),
\end{align}
where the last inequality uses the fact that $\sqrt{1 + x} \leq 1 + \frac{x}{2}$ for all $x \in [-1,\infty)$.
    
Finally, the third negative term can be bounded as
\begin{align}
    \sum_{t = \overline{\e}^2 T}^{(1-2\overline{\e}) T - 1} 
    \frac{1 + \overline{\e}}{T - t + 1}
    & 
    \,\leq\, 
    \int_{\overline{\e}^2 T}^{(1-2\overline{\e}) T} 
    \frac{1 + \overline{\e}}{T - t + 1} \,\d t
    \,=\, 
    \int_{1+2 \overline{\e} T }^{1+(1-\overline{\e}^2) T} 
    \frac{1 + \overline{\e}}{y} \,\d y
    \nonumber
    \\
    & 
    \,=\, 
    \left(1+\overline{\e}\right) 
    \log(y) \Bigg|_{1+2 \overline{\e} T }^{1+(1-\overline{\e}^2) T}
    \,\leq\, 
    \left(1+\overline{\e}\right) \log (1/\overline{\e})
\end{align}

Combining these bounds into \Cref{eq:valuePackFinal}, we conclude:
\begin{align}
    \E{\ALG} 
    & 
    \,\geq\, 
    \left(
        1 - 2\overline{\e} 
        - 2 \overline{\e} \log (2/\overline{\e}) 
        - \frac{1+\overline{\e}}{T} \log(1/\overline{\e})
    \right) \E{\OPT} 
    \nonumber
    \\
    &
    \,\geq\, 
    (1 - 8 \overline{\e} \log(2/\overline{\e})) \E{\OPT}
    \nonumber
    \\
    &
    \,\geq\, 
    (1 - \e) \E{\OPT}.
\end{align}
The second inequality uses the assumption that $T \geq \frac{1}{\overline{\e}}$ (otherwise, the facts that $B \geq \frac{1}{\e}$ and the matrices $\bA^t$ have entries in $[0,1]$ make all constraints redundant and the problem becomes trivial). The third inequality comes from \Cref{eq:eps_bar2}. This concludes the proof of \Cref{thm:resourceallocation}.
\myqed

\subsubsection{Concentration Inequalities}\hfill
\label{app:concentration}

In this section, we collect and show concentration inequalities that are used in the proof of \Cref{thm:resourceallocation}. These include a concentration inequality for exchangeable sequences (\Cref{cor:exBernstein}), and a concentration inequality for affine stochastic processes (\Cref{thm:centeringConc}).

We make use of the following result from \cite{van1997weak}:
\begin{theorem}[Theorem 2.14.19 in \cite{van1997weak}] 
    \label{thm:conc-hyp} 
    Let $A = \{a_1, \dots, a_n\}$ be a set of real numbers in $[0,1]$. 
    Let $S$ be a random subset of $A$ of size $s$ and let $A_S = \sum_{i \in S} a_i$. 
    Setting $\overline{a} = \frac{1}{n} \sum_{i=1}^n a_i$ and $\sigma^2 = \frac{1}{n}\sum_{i=1}^n (a_i - \overline{a})^2$, we have, for every $\tau > 0$:
    \begin{align}
        \Prob{|A_S - s \overline{a}| \geq \tau} 
        \leq 2 \exp{ - \frac{\tau^2}{2 s \sigma^2 + \tau} }.
    \end{align}
\end{theorem}

A sequence $X_1, \dots,X_n$ of random variables is \emph{exchangeable} if its distribution is permutation invariant, i.e., the distribution of the vector $(X_{\pi(1)}, X_{\pi(2)}, \dots, X_{\pi(n)})$ is the same for all permutations $\pi$ of $\{1, \dots, n\}$. The following result is the main result of this section. 

\begin{corollary}
    \label{cor:exBernstein}
    Let $X_1, \dots, X_n$ be an exchangeable sequence of random variables, i.i.d. according to distribution $\calD$, in the interval $[0,1]$. Assume that $\sum_{i=1}^n X_i \leq M$ with probability 1. Then for every $s \in \{1, \dots, n\}$ and $\tau >0$, we have:
    \begin{align}
    \Prob{ X_1 + \dots + X_s \geq \frac{s M}{n} + \tau } 
    \,\leq\, 
    2 \exp{ 
        -\min\left\{\frac{\tau^2 n}{8 s M}, \frac{\tau}{2}\right\} 
    }.
    \end{align}
\end{corollary}

\begin{proof}{Proof of \Cref{cor:exBernstein}.}
    We prove the result in the case where the distribution $\calD$ is discrete. The general case following from standard approximation arguments.
        
    Consider a set $A = \{a_1, \dots,a_n\}$ of $n$ values in $[0,1]$, and define $\overline{a} = \frac{1}{n} \sum_{i=1}^n a_i$ and $\sigma^2 = \frac{1}{n} \sum_{i=1}^n (a_i - \overline{a})^2$. Condition on the \emph{set} $\{X_1, \dots, X_n\}$ being equal to $A$, leaving their \emph{order} free. Under this conditioning, $X_1, \dots, X_s$ is just a random subset of size $s$ from $A$, and $\overline{a} = \frac{1}{n} \sum_{i=1}^n X_i \leq \frac{M}{n}$. From \Cref{thm:conc-hyp}, we get:
    \begin{align}
        \Pmid{ 
            X_1 + \dots + X_s \geq \frac{s M}{n} + \tau
        }{
            \{X_1, \dots, X_n\} = A
        } 
        \,\leq\, 
        2 \exp{ -\frac{\tau^2}{2s \sigma^2 +\tau} }. 
        \label{eq:exBernstein}
    \end{align}	
    
    Since the $a_i$'s belong to the interval $[0,1]$, the variance term can be bounded as \begin{equation}
        \sigma^2
        \,=\, 
        \frac{1}{n} \sum_{i=1}^n (a_i - \overline{a})^2 
        \,\leq\,
        \frac{1}{n} \sum_{i=1}^n |a_i - \overline{a}| 
        \,\leq\, 
        \frac{1}{n} \left( 
            \sum_{i=1}^n |a_i| 
            + \sum_{i=1}^n |\overline{a}| 
        \right) 
        \,=\, 2 \overline{a} 
        \,\leq\, \frac{2 M}{n}.
    \end{equation}
    
    Further using the inequality $\frac{1}{a+b} \geq \frac{1}{2} \min\{\frac{1}{a}, \frac{1}{b}\}$ for non-negative $a,b$, we obtain:
    \begin{align}
        \exp{ - \frac{\tau^2}{2s \sigma^2 +\tau} }
        \,\leq\, 
        \exp{ - \frac{\tau^2}{4s M/n +\tau} } 
        \,\leq\, 
        \exp{ - \frac{1}{2} \min \left\{
            \frac{\tau^2n}{4sM} , \tau
        \right\} }.
    \end{align}
    
    Taking the expectation of \Cref{eq:exBernstein} over all possible sets $A$ completes the proof. \myqed
\end{proof}

\begin{theorem}
    \label{thm:centeringConc}
    Consider a sequence $X_1, \dots, X_T$ of (possibly dependent) random variables in $[0,1]$ adapted to a filtration $\calH_1, \dots, \calH_T$. 
    Define the partial sums $S_t = X_1 + \dots + X_t$ for $t \in \{0,1, \dots, T\}$ (with $S_0 = 0$). 
    Furthermore, suppose that there are sequences $\alpha_1, \dots, \alpha_T \in [0,1]$ and $\beta_1, \dots, \beta_T \geq 0$ such that $\Emid{S_t}{\calH_{t-1}} \leq \alpha_t S_{t-1} + \beta_t$ for all $t$. 
    Then for every $\gamma \in (0,1]$, we have:
	\begin{align}
		\Prob{ S_T \geq (1+2\gamma) Y_T } \leq \exp{-\gamma^2 Y_T},
    \end{align}
    where $Y_1, \dots, Y_T$ is the solution to the recursion $y_t = \alpha_t y_{t-1} + \beta_t$ (with $y_0 = 0$).
\end{theorem}

We make use of a lemma on the moment-generating function of affine transformations of $S_t$.

\begin{lemma}
    \label{lemma:mgf_affine}
    Consider a function $f(x) = ax + b$ with $a \in [0,1]$. Under the assumptions from \Cref{thm:centeringConc}, for all $\gamma \in (0,1]$ we have
    \begin{align}
        \E{ \exp{ \gamma f(S_t) } } 
        \leq 
        \E{ \exp{ \gamma f(\alpha_t S_{t-1} + (1+\gamma) \beta_t) } }.
    \end{align}
\end{lemma}

\begin{proof}{Proof of \Cref{lemma:mgf_affine}.}
    Since conditioning on $\calH_{t-1}$ fixes the sum $S_{t-1}$, we observe that:
    \begin{align}
        \E{\exp{\gamma f(S_t) } } 
        = \E{ \exp{\gamma (a S_t + b) } } 
        = \E{
            \exp{\gamma (a S_{t-1} + b) }  
            \cdot \Emid{ \exp{\gamma a X_t } }{ \calH_{t-1} } 
        }.
        \label{eq:aff0}
    \end{align}
    
    We get:
    \begin{alignat}{2}
        \Emid{ \exp{\gamma a X_t } }{ \calH_{t-1} } 
        & \,\leq\, 
        \Emid{ 1 + \gamma a X_t + \gamma^2 a^2 X^2_t }{ \calH_{t-1} } 
        &\quad& \text{(since $e^x \leq 1 + x + x^2$ for $x \in [0,1]$)}
        \nonumber
        \\
        & \,\leq\, 
        \Emid{ 1 + (\gamma + \gamma^2) a X_t }{ \calH_{t-1} } 
        \nonumber
        \\
        & \,=\, 
        1 + (\gamma + \gamma^2) a \cdot \Emid{X_t}{\calH_{t-1}} 
        \nonumber
        \\
        & \,\leq\, 
        \exp{ (\gamma + \gamma^2) a \cdot \Emid{X_t}{\calH_{t-1}} }.
        &\quad& \text{(since $1 + x \leq e^x$)}
        \label{eq:aff1} 
    \end{alignat}
    
    Since $S_t = S_{t-1} + X_t$, the assumption $\Emid{S_t}{\calH_{t-1}} \leq \alpha_t S_{t-1} + \beta_t$ implies $\Emid{X_t}{\calH_{t-1}} \leq (\alpha_t - 1) S_{t-1} + \beta_t$. 
    Applying this to \Cref{eq:aff1} and $\gamma^2 a (\alpha_t - 1) S_t \leq 0$, we get:
    \begin{align}
        \Emid{ \exp{ \gamma a X_t } }{ \calH_{t-1} } 
        &\leq \exp{ (\gamma + \gamma^2) a ((\alpha_t - 1) S_{t-1} + \beta_t) }
        \nonumber\\
        &\leq \exp{ \gamma a ((\alpha_t - 1) S_{t-1} + \beta_t) + \gamma^2 a \beta_t }.
    \end{align}
    
    From \Cref{eq:aff0}, it comes:
    \begin{align}
        \E{ \exp{ \gamma f(S_t) } }
        &\leq \E{ \exp{ \gamma (a (\alpha_t S_{t-1} + (1+\gamma) \beta_t) + b) } } 
        \nonumber\\
        &= \E{ \exp{ \gamma f(\alpha_t S_{t-1} + (1+\gamma) \beta_t) } }.
    \end{align}
    
    This concludes the proof of the lemma.  \myqed
\end{proof}

\begin{proof}{Proof of \Cref{thm:centeringConc}.}
Define the affine function $f_t(x) = \alpha_t x + (1 + \gamma) \beta_t$, so that \Cref{lemma:mgf_affine} can be expressed as $\E{\exp{ \gamma f(S_t) } } \leq \E{ \exp{ \gamma f(f_t(S_{t-1})) } }$. Applying it repeatedly gives:
\begin{align}
    \E{\exp{ \gamma S_T } } 
    & \,\leq\, 
    \E{ \exp{ \gamma f_T(S_{T-1}) } } 
    \,\leq\, 
    \E{ \exp{ \gamma f_T(f_{T-1}(S_{T-2})) } }
    \\
    & \,\leq\, 
    \dots 
    \,\leq\, 
    \E{ \exp{ \gamma f_T(f_{T-1}(\dots f_2(f_1(0)))) } }.
\end{align}

We can still apply \Cref{lemma:mgf_affine} because the composed function $f_T \circ f_{T-1} \circ \dots \circ f_t$ is still affine of the form $ax + b$ with $a \in [0,1]$ and $b \geq 0$ (indeed, $a = \alpha_T \dots \alpha_t \in [0,1]$ and $b$ is obtained by taking products and sums of the $\alpha_t$'s and $\beta_t$'s, which are all non-negative).

Moreover, we prove by induction on $t=1, \dots, T$ that the composition of these functions satisfies:
\begin{align}
    f_t(f_{t-1}(\dots f_2(f_1(0)) )) = (1+\gamma) Y_t.
\end{align}

For $t=0$, we have $f_1(0) = (1+\gamma) \beta_1 = (1 + \gamma) Y_1$. For $t\geq1$, we have:
\begin{align}
    f_t(f_{t-1}(\dots f_2(f_1(0)))) 
    =
    f_t((1+\gamma) Y_{t-1}) 
    = 
    (1+\gamma) \left[ \alpha_t Y_{t-1} + \beta_t\right] 
    = 
    (1+\gamma) Y_t.
\end{align}

Thus, we get the moment-generating function upper bound $\E{\exp{ \gamma S_T } } \leq \exp{ \gamma (1+\gamma) Y_T } $. Finally, applying Markov's inequality we get:
\begin{align}
    \Prob{ \Big. S_T \geq (1+2\gamma) Y_T } 
    & \,=\, 
    \Prob{ \Big. \exp{\gamma S_T} \geq \exp{\gamma (1+2\gamma) Y_T} } 
    \nonumber
    \\
    & \,\leq\, 
    \frac{\E{ \exp{\gamma S_T} }}{\exp{\gamma (1+2\gamma) Y_T}} 
    \nonumber
    \\
    & \,\leq\, 
    \frac{\exp{\gamma (1+\gamma) Y_T}}{\exp{\gamma (1+2\gamma) Y_T}} 
    \nonumber
    \\
    &\,=\,
    \exp{-\gamma^2 Y_T}.
\end{align}
This concludes the proof of the theorem.\myqed
\end{proof}

\subsubsection{Proof of \Cref{prop:myopic}.}\hfill

Consider the online resource allocation problem with $T$ items of size 1, one supply, and one resource in quantity $\sqrt{T}$, i.e., $|\calJ|=1$, $|\calK|=1$, $A^t_{11}=1$ for all $t\in\{1, \dots, T\}$, and $b_1=\sqrt{T}$. Assume that item values $r^t_j$ are equal to $1$ with probability $1/\sqrt{T}$ and to a small value $\varphi > 0$ with probability $1-1/\sqrt{T}$. Since there are on average $\sqrt{T}$ items of value $1$, it can be shown that $\E{\OPT} \approx \sqrt{T}$. Per Theorem 1, for any $\varepsilon>0$, the OSO algorithm achieves a value within a multiplicative factor of $1-\varepsilon$ of the $\sqrt{T}$ optimum for large enough $T$. However, a myopic decision-making rule would always assign items to the supply until all $\sqrt{T}$ resources have been consumed, leading to an expected value of $\sqrt{T} \left(1 \cdot \frac{1}{\sqrt{T}} + \varphi \cdot \left(1-\frac{1}{\sqrt{T}}\right)\right)\,=\, \varphi \sqrt{T} + 1 - \varphi \,\approx\, \varphi \sqrt{T}$. Therefore, a myopic approach achieves a fraction $\varphi$ of the optimal value, which can be made arbitrarily small. \myqed

\subsubsection{Proof of \Cref{prop:CE}.}\hfill

Consider the online resource allocation problem with one supply bin ($m=1$), $d > 2$ resources, and $T > d^2$ time periods. All rewards are known and equal to $r^t_{1}=1$. Each resource $k\in\calK$ has capacity $b_k = \sqrt{T}$. The assignment of an incoming item into the supply bin consumes one unit of one resource  with probability $1/\sqrt{T}$ and one unit of each resource with probability $1-d/\sqrt{T}$. Formally, we define the following probability distribution with support over $d+1$ possible realizations; for simplicity, we encode it via a random variable $\xi_t$ in $\{0, \dots, d\}$.
\begin{equation}
    \left(A^t_{1k}\right)_{k\in\calK} = \begin{cases}
        \be_1 & \text{with probability } 1 / \sqrt{T}\ \text{(encoded via $\xi_t=1$)} \\
        \vdots \\
        \be_d & \text{with probability } 1 / \sqrt{T}\ \text{(encoded via $\xi_t=d$)} \\
        \bone & \text{with probability } 1 - d / \sqrt{T}\ \text{(encoded via $\xi_t=0$)}
    \end{cases}
\end{equation}

The mean-based certainty-equivalent (CE) algorithm for the multi-dimensional knapsack problem solves the following problem at each iteration, and implements the current-period solution $\overline{x}^t := x^t$.
\begin{alignat}{3}
    \max \quad 
    & \sum_{\tau=1}^{t-1} r^{\tau} \overline{x}^{\tau}
    + r^{t} x^{t}
    + \sum_{\tau=t+1}^{T} \E{r^{\tau}} x^{\tau} 
    \label{eq:CEProp2}
    \\
    \text{s.t.} 
    \quad 
    & \sum_{\tau=1}^{t-1} A^{\tau}_{1k} \overline{x}^{\tau}
    + A^{t}_{1k} x^{t}
    + \sum_{\tau=t+1}^{T} \E{A^{\tau}_{1k}} x^{\tau}
    \leq b_k,
    &\quad& \forall \ k \in \calK 
    \\
    & x^{\tau} \in \{0, 1\}
    &\quad& \forall \ \tau \in \{t, \dots, T\}
\end{alignat}

\underline{Hindsight-optimal solution.}
Let $N_\ell = \sum_{t=1}^T \mathbbm{1}(\xi_t = \ell)$ characterize the number of time periods with realization $\ell\in\{0, \dots, d\}$. Then, $(N_1, \dots, N_d, N_0)$ follows a multinomial distribution with sum $T$ and parameters $\left(\frac{1}{\sqrt{T}}, \dots, \frac{1}{\sqrt{T}}, 1 - \frac{d}{\sqrt{T}}\right)$. The hindsight-optimal policy can be formulated via decision variable $G_\ell$ characterizing the number of times an item of type $\ell$ is chosen. Then, the hindsight-optimal solution, denoted by $\OPT(\bxi)$, is obtained from the following optimization problem:
\begin{alignat*}{2}
    \OPT(\bxi) = \max \quad 
    & \sum_{j=0}^d G_\ell
    \\
    \st \quad 
    & G_\ell \leq N_\ell 
    &\quad& \forall \ \ell \in \{0, \dots, d\}
    \\
    & G_\ell + G_0 \leq \sqrt{T}
    &\quad& \forall \ \ell \in \{1, \dots, d\}
\end{alignat*}
Since items of type $0$ involve higher resource consumption, they get de-prioritized in favor of all other item types. The optimal solution is therefore given by:
\begin{align*}
	& G_\ell = \min \{ N_\ell, \sqrt{T} \},\ \forall \ \ell \in \{1, \dots, d\}	\\
	& G_0 = \sqrt{T} - \max_{\ell \in \{1, \dots, d\}} G_\ell
\end{align*}
We can bound $\E{\OPT(\bxi)}$ as follows, for a small enough $\varepsilon>0$ and a large enough $T$:
\begin{alignat*}{2}
    \E{\OPT(\bxi)} 
    & = \sum_{j=0}^d \E{G_\ell}
    \\
    & \geq \sum_{j=1}^d \E{G_\ell}
    \\
    & = d \cdot \left( 
        \Emid{N_1 }{ N_1 \leq \sqrt{T}} \cdot \Prob{ N_1 \leq \sqrt{T} }
        + \sqrt{T} \cdot \Prob { N_1 \geq \sqrt{T} }
    \right)
    \quad\text{(by symmetry)}
    \\
    & \geq d \cdot \left( 
        \frac{1}{2} \left( 
        \sqrt{T} - (\sqrt{T} - 1)^{1/2} \frac{1/\sqrt{2\pi}}{1/2} 
        \right)
        + \frac{1}{2} \sqrt{T}
    \right)-\varepsilon
    \\
    & = d \sqrt{T} - d \frac{1}{\sqrt{2\pi}} (\sqrt{T}-1)^{1/2}-\varepsilon
    \\
    & \geq d \sqrt{T} - d T^{1/4}
\end{alignat*}

The critical step lies in the second inequality. Since $(N_1, \dots, N_d, N_0)$ follows a multinomial distribution with sum $T$ and parameters $\left(\frac{1}{\sqrt{T}}, \dots, \frac{1}{\sqrt{T}}, 1 - \frac{d}{\sqrt{T}}\right)$, $N_1$ follows a binomial distribution with $T$ trials and success probability $\frac{1}{\sqrt{T}}$. Therefore, its mean is $\sqrt{T}$ and its variance is $\sqrt{T}-1$. By the central limit theorem, we therefore know that $N_1$ converges to a normal distribution with mean $\sqrt{T}$ and variance $\sqrt{T}-1$ as $T$ grow large. 
Thus, $\Prob{ N_1 \leq \sqrt{T} }\approx0.5$ and $\Emid{ N_1 }{ N_1 \leq \sqrt{T}} \approx \sqrt{T} - (\sqrt{T} - 1)^{1/2} \frac{1/\sqrt{2\pi}}{1/2}$. This proves that $\E{\OPT(\bxi)}\geq d \sqrt{T} - d T^{1/4}$.

\underline{OSO solution.}
Per Theorem 1, the single-sample OSO algorithm achieves a value close to $\E{\OPT(\bxi)}$, the hindsight optimum, for large enough $T$. Specifically, the single-sample OSO algorithm obtains a value of at least $(1 - \varepsilon_{d,T,B}) \cdot \E{\OPT(\bxi)}$, with $\varepsilon$ scaling approximately as $\calO\left(\sqrt{\frac{\log(dT)}{B}}\right)$.

\underline{CE solution.} The expected resource utilization in each period, denoted by $\overline{A}$, is equal to:
\begin{alignat}{2}
    \overline{A}=\E{A_{1k}^t} 
    & = \frac{1}{\sqrt{T}} + \left(1 - \frac{d}{\sqrt{T}} \right)
    = 1 - \frac{d-1}{\sqrt{T}}<1,\quad\forall k\in\calK
\end{alignat}

Consider a decision epoch $t\in\{1, \dots, T\}$ where the capacity of all resources is equal to $B$. Assume that $\xi_t=0$, i.e., that the incoming item consumes one unit of each resource. Since the mean item consumes $\overline{A} < 1$ unit of each resource while yielding the same reward, the CE solution will reject the incoming item as long as the remaining time horizon is long enough. Similarly, if $\xi_t=j\in\{1, \dots, d\}$, the incoming item consumes one unit of resource $j$ whereas the mean item consumes $\overline{A}<1$ unit of each resource while yielding the same reward. Again, the CE solution will serve as many copies of the mean item as possible rather than the incoming item, and it will therefore reject the incoming item as long as the remaining time horizon is long enough.

Therefore, the CE solution will not serve any incoming item starting from $t = 1$ for the first $T - \sqrt{T}/\overline{A}$ (since the capacities of all resources remain equal). In turn, the CE algorithm will accept the remaining $\sqrt{T}/\overline{A}$ items up to capacity. The CE objective value, referred to as $\text{CE}$, therefore achieves an expected value of $\E{\text{CE}(\bxi)}\leq\sqrt{T}/\overline{A}$.

Next, we prove that the CE solution achieves a competitive ratio of at most $1/d$. Suppose for the sake of contradiction that its competitive ratio is $\frac{1}{d} + \varepsilon$ for some $\varepsilon > 0$. Then the following holds:
\begin{alignat}{2}
    &\quad& 
    \E{\text{CE}(\bxi)}
    & \,\geq\, 
    \left(\frac{1}{d} + \varepsilon \right) \E{\OPT(\bxi)} - \alpha 
    \nonumber
    \\
    \implies
    &\quad&
    \frac{\sqrt{T}}{1 - \frac{d-1}{\sqrt{T}}}
    & \,\geq\, 
    \left(\frac{1}{d} + \varepsilon \right) \cdot \left( d \sqrt{T} - d T^{1/4} \right) - \alpha
    \nonumber
    \\
    \implies 
    &\quad& 
    T 
    & \,\geq\, \left( \sqrt{T} - (d - 1) \right) 
    \cdot \left( 
        (1 + \varepsilon d) \sqrt{T} - (1 + \varepsilon d) T^{1/4} - \alpha
    \right)
    \nonumber
    \\
    \implies
    &\quad& 
    0 
    & \,\geq\, 
    \varepsilon d T 
    - (1 + \varepsilon d) T^{3/4} 
    - (\alpha + (d-1)(1+\varepsilon d)) \sqrt{T}
    + (d-1)(1+\varepsilon d) T^{1/4} 
    + (d-1) \alpha
    \nonumber
\end{alignat}
which is a contradiction for large enough $T$. 

Therefore, CE achieves a competitive ratio of at most $1 / d$, and OSO can achieve a competitive ratio of $1 - \varepsilon_{d,T,B}$ with $\varepsilon$ scaling as $\calO \left(\sqrt{\frac{\log(dt)}{B}}\right)$, This proves that OSO can achieve unbounded (multiplicative) benefits over CE as the number of resources $d$ grows infinitely.\myqed

\subsection{Computational Results}
\label{app:resourcealloc_comp}

\subsubsection{Need for Resolving Heuristics}\hfill
\label{app:resourcealloc_comp_exact}

We first show that the online resource allocation problem (\Cref{eq:mssip}) remains intractable with off-the-shelf stochastic programming and dynamic programming methods even in moderately-sized instances, thus motivating the need for efficient resolving heuristics such as the OSO algorithm studied in this paper. We consider an online multidimensional knapsack problem over $T$ periods with $|\calJ| = 1$ supply node, and $|\calK| = d$ resources of capacity $B$. Items arrive one at a time, and the resource consumption variables $A^t_{1k}$ are binary (equal to 0 or 1 with probability 0.5) and independent. The complexity of the problem is driven by the number of items $T$, the number of resources $d$, and the capacity $B$.

\begin{table}[h!]
    \setlength{\tabcolsep}{4pt}
    \centering
    \small
    \begin{tabular}{
        c
        S[table-format=2] S[table-format=2] S[table-format=2]
        *{6}{S[table-format=1.2e-1]}
    }   
        \toprule
        Metric & {$T$} & {$B$} & {$d$} 
        & {SP} & {DP} & {OSO-1} & {OSO-5} & {OSO-10} & {OSO-20}
        \\ 
        \midrule 
        {Obj.} 
        & 4 & 2 & 2 
        & \multicolumn{1}{c}{\tablenum[table-format=1.2]{3.34}} 
        & \multicolumn{1}{c}{\tablenum[table-format=2.2]{3.34}} 
        & \multicolumn{1}{c}{\tablenum[table-format=2.2]{3.34}}
        & \multicolumn{1}{c}{\tablenum[table-format=2.2]{3.33}}
        & \multicolumn{1}{c}{\tablenum[table-format=2.2]{3.33}}
        & \multicolumn{1}{c}{\tablenum[table-format=2.2]{3.34}} \\ 
        & 6 & 3 & 2 
        & \multicolumn{1}{c}{\tablenum[table-format=1.2]{5.1}} 
        & \multicolumn{1}{c}{\tablenum[table-format=2.2]{5.1}} 
        & \multicolumn{1}{c}{\tablenum[table-format=2.2]{5.11}}
        & \multicolumn{1}{c}{\tablenum[table-format=2.2]{5.09}}
        & \multicolumn{1}{c}{\tablenum[table-format=2.2]{5.12}}
        & \multicolumn{1}{c}{\tablenum[table-format=2.2]{5.09}} \\ 
        & 8 & 4 & 2 
        & \multicolumn{1}{c}{\tablenum[table-format=1.2]{6.94}} 
        & \multicolumn{1}{c}{\tablenum[table-format=2.2]{6.94}}
        & \multicolumn{1}{c}{\tablenum[table-format=2.2]{6.96}}
        & \multicolumn{1}{c}{\tablenum[table-format=2.2]{6.9}}
        & \multicolumn{1}{c}{\tablenum[table-format=2.2]{6.96}}
        & \multicolumn{1}{c}{\tablenum[table-format=2.2]{6.94}} \\ 
        & 10 & 5 & 2 
        & \multicolumn{1}{c}{\tablenum[table-format=1.2]{8.91}} 
        & \multicolumn{1}{c}{\tablenum[table-format=2.2]{8.91}}
        & \multicolumn{1}{c}{\tablenum[table-format=2.2]{8.89}}
        & \multicolumn{1}{c}{\tablenum[table-format=2.2]{8.84}}
        & \multicolumn{1}{c}{\tablenum[table-format=2.2]{8.91}}
        & \multicolumn{1}{c}{\tablenum[table-format=2.2]{8.9}} \\ 
        & 12 & 6 & 2 & {---} 
        & \multicolumn{1}{c}{\tablenum[table-format=2.2]{10.74}}
        & \multicolumn{1}{c}{\tablenum[table-format=2.2]{10.73}}
        & \multicolumn{1}{c}{\tablenum[table-format=2.2]{10.62}}
        & \multicolumn{1}{c}{\tablenum[table-format=2.2]{10.72}}
        & \multicolumn{1}{c}{\tablenum[table-format=2.2]{10.74}} \\ 
        & 14 & 7 & 2 & {---} 
        & \multicolumn{1}{c}{\tablenum[table-format=2.2]{12.64}}
        & \multicolumn{1}{c}{\tablenum[table-format=2.2]{12.62}}
        & \multicolumn{1}{c}{\tablenum[table-format=2.2]{12.55}}
        & \multicolumn{1}{c}{\tablenum[table-format=2.2]{12.63}}
        & \multicolumn{1}{c}{\tablenum[table-format=2.2]{12.64}} \\ 
        & 16 & 8 & 2 & {---} 
        & \multicolumn{1}{c}{\tablenum[table-format=2.2]{14.56}}
        & \multicolumn{1}{c}{\tablenum[table-format=2.2]{14.52}}
        & \multicolumn{1}{c}{\tablenum[table-format=2.2]{14.41}}
        & \multicolumn{1}{c}{\tablenum[table-format=2.2]{14.55}}
        & \multicolumn{1}{c}{\tablenum[table-format=2.2]{14.55}} \\ 
        & 18 & 9 & 2 & {---} 
        & \multicolumn{1}{c}{\tablenum[table-format=2.2]{16.45}}
        & \multicolumn{1}{c}{\tablenum[table-format=2.2]{16.37}}
        & \multicolumn{1}{c}{\tablenum[table-format=2.2]{16.2}}
        & \multicolumn{1}{c}{\tablenum[table-format=2.2]{16.44}}
        & \multicolumn{1}{c}{\tablenum[table-format=2.2]{16.39}} \\ 
        & 20 & 10 & 2 & {---} 
        & \multicolumn{1}{c}{\tablenum[table-format=2.2]{18.41}}
        & \multicolumn{1}{c}{\tablenum[table-format=2.2]{18.36}}
        & \multicolumn{1}{c}{\tablenum[table-format=2.2]{18.24}}
        & \multicolumn{1}{c}{\tablenum[table-format=2.2]{18.38}}
        & \multicolumn{1}{c}{\tablenum[table-format=2.2]{18.37}} \\ 
        & 20 & 10 & 3 & {---} 
        & \multicolumn{1}{c}{\tablenum[table-format=2.2]{17.96}}
        & \multicolumn{1}{c}{\tablenum[table-format=2.2]{17.92}}
        & \multicolumn{1}{c}{\tablenum[table-format=2.2]{17.75}}
        & \multicolumn{1}{c}{\tablenum[table-format=2.2]{17.9}}
        & \multicolumn{1}{c}{\tablenum[table-format=2.2]{17.92}} \\ 
        & 20 & 10 & 4 & {---}
        & \multicolumn{1}{c}{\tablenum[table-format=2.2]{17.64}}
        & \multicolumn{1}{c}{\tablenum[table-format=2.2]{17.5}}
        & \multicolumn{1}{c}{\tablenum[table-format=2.2]{17.43}}
        & \multicolumn{1}{c}{\tablenum[table-format=2.2]{17.57}}
        & \multicolumn{1}{c}{\tablenum[table-format=2.2]{17.59}} \\ 
        & 20 & 10 & 5 & {---} 
        & \multicolumn{1}{c}{\tablenum[table-format=2.2]{17.34}}
        & \multicolumn{1}{c}{\tablenum[table-format=2.2]{17.16}}
        & \multicolumn{1}{c}{\tablenum[table-format=2.2]{17.11}}
        & \multicolumn{1}{c}{\tablenum[table-format=2.2]{17.3}}
        & \multicolumn{1}{c}{\tablenum[table-format=2.2]{17.35}} \\ 
        & 20 & 10 & 6 & {---} & {---}
        & \multicolumn{1}{c}{\tablenum[table-format=2.2]{16.84}}
        & \multicolumn{1}{c}{\tablenum[table-format=2.2]{16.81}}
        & \multicolumn{1}{c}{\tablenum[table-format=2.2]{17.09}}
        & \multicolumn{1}{c}{\tablenum[table-format=2.2]{17.09}} \\ 
        \midrule
        {Time (s)}
        & 4 & 2 & 2 & 2.65e-03 & 1.93e-04 & 3.13e-03 & 3.00e-03 & 3.35e-03 & 3.84e-03 \\ 
        & 6 & 3 & 2 & 4.87e-02 & 5.46e-04 & 2.96e-03 & 4.39e-03 & 4.36e-03 & 4.64e-03 \\ 
        & 8 & 4 & 2 & 1.04e+00 & 1.16e-03 & 4.27e-03 & 4.90e-03 & 5.34e-03 & 6.86e-03 \\ 
        & 10 & 5 & 2 & 3.59e+01 & 2.08e-03 & 5.35e-03 & 6.55e-03 & 6.80e-03 & 9.27e-03 \\ 
        & 12 & 6 & 2 & {---} & 3.47e-03 & 6.24e-03 & 7.26e-03 & 8.75e-03 & 1.16e-02 \\ 
        & 14 & 7 & 2 & {---} & 5.35e-03 & 9.88e-03 & 9.54e-03 & 1.14e-02 & 1.47e-02 \\ 
        & 16 & 8 & 2 & {---} & 7.89e-03 & 8.72e-03 & 1.37e-02 & 1.39e-02 & 1.79e-02 \\ 
        & 18 & 9 & 2 & {---} & 1.10e-02 & 1.05e-02 & 1.70e-02 & 1.69e-02 & 2.23e-02 \\ 
        & 20 & 10 & 2 & {---} & 1.50e-02 & 1.11e-02 & 1.71e-02 & 1.87e-02 & 2.66e-02 \\ 
        & 20 & 10 & 3 & {---} & 5.19e-01 & 1.17e-02 & 2.36e-02 & 2.15e-02 & 3.25e-02 \\ 
        & 20 & 10 & 4 & {---} & 1.30e+01 & 1.16e-02 & 1.59e-02 & 2.19e-02 & 3.74e-02 \\ 
        & 20 & 10 & 5 & {---} & 3.35e+02 & 1.43e-02 & 2.66e-02 & 4.27e-02 & 4.47e-02 \\ 
        & 20 & 10 & 6 & {---} & {---} & 1.26e-02 & 2.03e-02 & 2.84e-02 & 4.34e-02 \\ 
        \bottomrule
    \end{tabular}
    \begin{tablenotes}
	\item {---}: Instances which did not complete within a 1-hour time limit.
    \end{tablenotes}
    \caption{Performance comparison in the online multi-dimensional knapsack problem. ``SP'': Stochastic programming. ``DP'': Dynamic programming. ``OSO-$S$'': OSO algorithm with $S$ sample paths per iteration. Solutions evaluated on 100 instances.}
    \label{tab:exact}
\end{table}

\Cref{tab:exact} compares the objective values and computational times of the OSO algorithm against (i) a multi-stage stochastic programming formulation based on a scenario-tree representation (SP), and (ii) the policy computed via dynamic programming (DP). Note that the scenario-tree representation leads to highly intractable integer optimization instances even for small problems. The full scenario tree involves $2^d$ scenarios at each time period, hence $\calO(2^{dT})$ overall nodes over the multi-stage horizon, $\calO(2^{dT})$ integer decision variables and $\calO(2^{d(T-1)})$ constraints. The stochastic programming model becomes intractable very quickly, with as few as 12 time periods, 1 supply node, 2 resources, and binary uncertainty; in comparison, in the next section, we will solve instances with up to 100 time periods, 10 supply nodes, 20 resources, and continuous uncertainty. Whereas these results rely on an exhaustive scenario tree, they also suggest that stochastic programming remains intractable even with small-sample representations of uncertainty. For example, with $d=2$, the scenario tree involves 4 possible realizations at each time period, and leads to an intractable formulation with $T=12$ and $|\calJ|=1$. Even small-sample scenario-tree approximations based, for example, on sample average approximation or scenario reduction, would result in similarly intractable stochastic programming formulations.

The dynamic programming algorithm is more scalable but still terminates up to 4-5 orders of magnitude slower than OSO, and times out in comparatively small instances (e.g., 20 time periods, 1 supply node, and 6 resources). This again stems from the exponential growth in problem size, with $\calO(T (2B)^d)$ possible states. In comparison, the OSO algorithm is very computationally efficient, terminating in fractions of a second on these simple examples. Furthermore, the OSO solutions are very close to optimal. In low-dimensional problem instances, even the single-sample variant of OSO leads to virtually identical solutions as the DP algorithm (variations are due to the randomness associated with the 100 out-of-sample scenarios). When the number of resources $d$ grows larger, the OSO algorithm benefits from more sample paths.

Next, we compare the OSO algorithm to the other resolving heuristics and the perfect-information benchmark.

\subsubsection{Comparison on online resource allocation problems}\hfill

We provide additional computational results to complement the results from \Cref{ssec:computation} for the multi-dimensional knapsack problem, the online generalized assignment problem, and the general online resource allocation problem. Recall that these problems involve very high-dimensional discrete allocation problems with up to 50 time periods, 10 supply nodes, and 50 resources; or up to 100 time periods, 10 supply nodes, and 20 resources. All uncertain resource consumption parameters $A^t_{jk}$ follow a bimodal distribution parametrized by $\psi$, described in the main text and shown in \Cref{fig:bimodal_pdf}. At each iteration, the integer program is solved with Gurobi 11 and Julia 1.10, with 2 threads. For the multidimensional knapsack and online generalized assignment problems, each IP was solved with termination criteria of either 60 seconds or a relative gap of 0.01\%. For the more challenging online resource allocation instances, we used a termination criteria of either 100 seconds per iteration or a relative gap of 1\%. The capacity parameter $B$ was chosen to vary over a range which is not too small (such that all algorithms cannot serve demand) nor too large (such that all algorithms can trivially serve all items).

\begin{figure}[h!]
    \centering
    \includegraphics[width=0.7\linewidth]{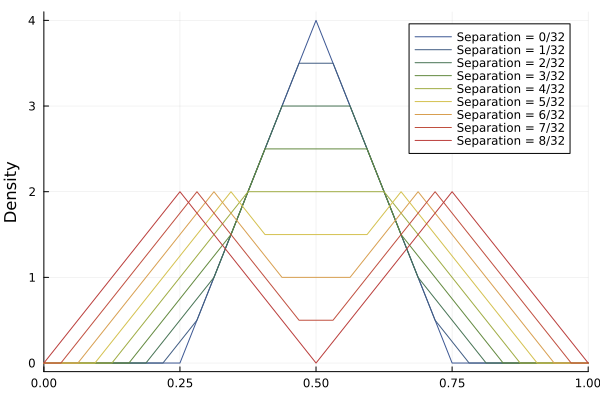}
    \caption{Bimodal distribution with separation parameter $\psi \in [0, 0.25]$.}
    \label{fig:bimodal_pdf}
\end{figure}

\Cref{tab:knapsack_app,tab:assignment_fixedC_app,tab:resourcealloc_app} compare the algorithms' solutions for each problem; \Cref{fig:knapsack_groupedbar_obj_time_app} to \Cref{fig:resource_alloc_groupedbar_obj_time_app} show them visually in absolute terms, without normalization. Throughout, the myopic policy induces a significant loss as compared to the hindsight-optimal solution, of up to 30\% for multidimensional knapsack, 37\% for online generalized assignment, and 50\% for online resource allocation. The CE benchmark improves upon the myopic solution in all settings except the multidimensional knapsack with bimodal resource consumption. Then, single-sample OSO yields significant improvements in solution quality over the CE benchmark, (up to 39\% for online resource allocation). These improvements are stronger when the distribution is more bimodal (and hence the mean is less representative of a sample) and when capacity is smaller. Finally, multi-sample OSO can yield additional improvements but these become marginally smaller as the number of sample paths increases. Although the single-sample and small-sample OSO methods involve longer computational times, these methods still yield high-quality solutions within the time limit.

\begin{table}
    \centering
    {\small
    \begin{tabular}{
        S[table-format=1.2] S[table-format=1.1] c
        S[table-format=3.2] S[table-format=3.2] S[table-format=3.2] r S[table-format=3.2] c
        S[table-format=1.3] S[table-format=1.3] S[table-format=1.3] S[table-format=4.3]
    }
        \toprule
        & & & \multicolumn{5}{c}{Objective (\% of perfect-info)} & & \multicolumn{4}{c}{Computation time (s)}
        \\
        \cmidrule(lr){4-8} \cmidrule(lr){10-13}
        {$\psi$} & {$B$} 
        & & {Myopic} & {CE} & {OSO-1} & {(over CE)} & {OSO-5}
        & & {Myopic} & {CE} & {OSO-1} & {OSO-5}
        \\ 
        \midrule
        0.0 & 0.1 
        & & 90.00 & 87.9 & 94.36 & {$+7.3$\%} & 94.74 
        & & 0.064 & 0.082 & 0.150 & 19.1 \\ 
        0.0 & 0.2 
        & & 92.97 & 90.98 & 95.80 & {$+5.3$\%} & 96.59 
        & & 0.055 & 0.078 & 0.388 & 339 \\ 
        0.0 & 0.3 
        & & 93.31 & 91.99 & 96.53 & {$+4.9$\%} & 96.67 
        & & 0.056 & 0.084 & 0.372 & 145 \\ 
        0.0 & 0.4 
        & & 96.94 & 95.94 & 98.08 & {$+2.2$\%} & 98.28 
        & & 0.055 & 0.089 & 0.219 & 39.0 \\ 
        \midrule
        0.25 & 0.1 
        & & 70.32 & 71.45 & 88.20 & {$+23.4$\%} & 91.15 
        & & 0.056 & 0.082 & 1.02 & 1150 \\ 
        0.25 & 0.2 
        & & 88.81 & 81.22 & 94.60 & {$+16.5$\%} & 95.55 
        & & 0.056 & 0.076 & 0.748 & 790 \\ 
        0.25 & 0.3 
        & & 94.83 & 95.25 & 96.21 & {$+1.0$\%} & 97.60 
        & & 0.062 & 0.077 & 0.113 & 1.20 \\ 
        0.25 & 0.4 
        & & 100.0 & 100.0 & 100.0 & {$+0.0$\%} & 100.0 
        & & 0.070 & 0.075 & 0.079 & 0.191 \\ 
        \bottomrule
    \end{tabular}
    }
    \caption{Objective relative to the hindsight optimum (in percent) and computation time (in seconds) for the online multi-dimensional knapsack problem. ``OSO-1'', ``OSO-5'': OSO algorithm with 1, 5 sample paths per iteration.}
    \label{tab:knapsack_app}
\end{table}

\begin{figure}
\centering
\subfloat[\small Objective (unimodal distribution, $\psi = 0$)]{\includegraphics[width=0.49\textwidth]{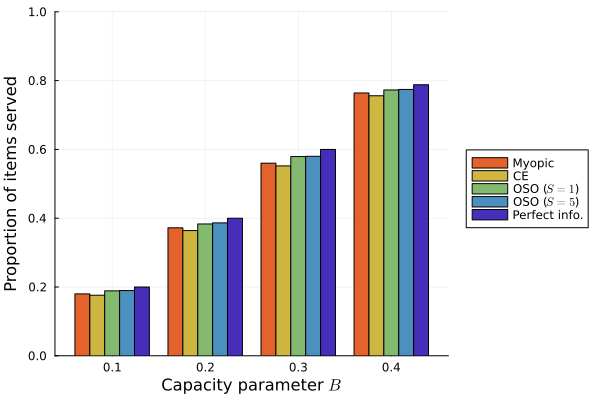}
\label{fig:knapsack_unimodal_groupedbar_obj_app}
}%
\subfloat[\small Computation time (unimodal distribution, $\psi = 0$)]{\includegraphics[width=0.49\textwidth]{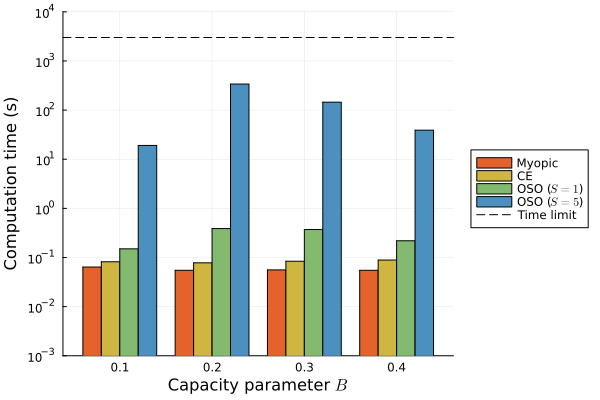}
\label{fig:knapsack_unimodal_groupedbar_time_app}
}%
\newline
\subfloat[\small Objective (bimodal distribution, $\psi = 1/4$)]{\includegraphics[width=0.49\textwidth]{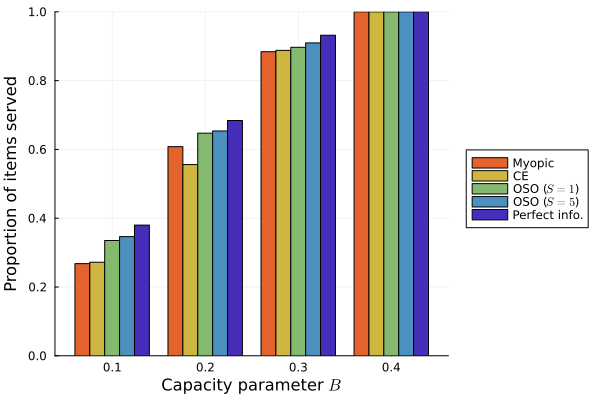}
\label{fig:knapsack_bimodal_groupedbar_obj_app}
}%
\subfloat[\small Computation time (bimodal distribution, $\psi = 1/4$)]{\includegraphics[width=0.49\textwidth]{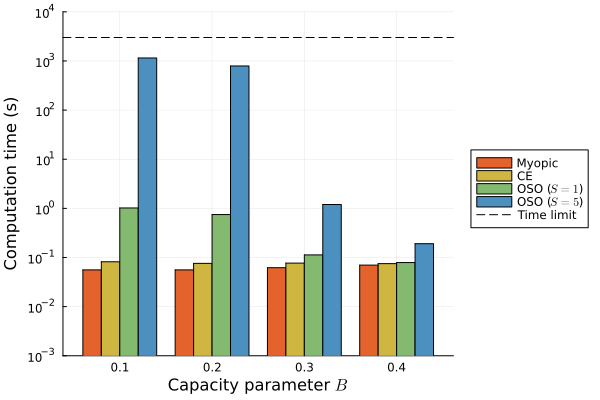}
\label{fig:knapsack_bimodal_groupedbar_time_app}
}%
\caption{Normalized objectives and computation times for the online multidimensional knapsack problem.} 
\label{fig:knapsack_groupedbar_obj_time_app}
\end{figure}

\begin{table}
    \centering
    {
    \footnotesize
    \begin{tabular}{
        S[table-format=1.2] S[table-format=1.1] 
        S[table-format=2.2] S[table-format=2.2] S[table-format=2.2] S[table-format=2.2] S[table-format=2.2] S[table-format=2.2] 
        S[table-format=1.3] S[table-format=1.3] S[table-format=3.3] S[table-format=4.2] S[table-format=4.1] S[table-format=4.1]}
        \toprule
        & & \multicolumn{6}{c}{Objective (\% of perfect-info)} & \multicolumn{6}{c}{Computation time (s)}
        \\
        \cmidrule(lr){3-8} \cmidrule(lr){9-14}
        {$\psi$} & {$B$} 
        & {Myopic} & {CE} & {OSO-1} & {OSO-5} & {OSO-10} & {OSO-20}
        & {Myopic} & {CE} & {OSO-1} & {OSO-5} & {OSO-10} & {OSO-20}
        \\ 
        \midrule
        0.0 & 0.2 
        & 74.38 & 95.93 & 93.58 & 94.17 & 95.98 & 95.37 
        & 0.077 & 0.221 & 0.575 & 6.44 & 14.8 & 35.8 \\ 
        0.0 & 0.3 & 77.84 & 96.64 & 96.12 & 95.56 & 96.38 & 96.09 
        & 0.076 & 0.228 & 4.22 & 53.0 & 164 & 502 \\ 
        0.0 & 0.4 & 80.39 & 95.98 & 95.49 & 96.69 & 97.69 & 97.8 
        & 0.059 & 0.238 & 17.3 & 188 & 682 & 1200 \\ 
        0.0 & 0.5 & 84.35 & 91.57 & 94.86 & 97.84 & 98.24 & 98.48 
        & 0.069 & 0.229 & 21.5 & 45.3 & 118 & 264 \\ 
        \midrule
        0.25 & 0.2 
        & 68.22 & 78.63 & 84.94 & 87.37 & 88.12 & 91.41 
        & 0.080 & 0.230 & 7.03 & 76.7 & 200 & 535 \\ 
        0.25 & 0.3 
        & 63.24 & 68.52 & 81.49 & 86.23 & 86.76 & 87.27 
        & 0.071 & 0.252 & 213 & 1030 & 1310 & 1560 \\ 
        0.25 & 0.4 
        & 69.11 & 72.3 & 84.79 & 88.47 & 90.07 & 89.99 
        & 0.059 & 0.226 & 841 & 1620 & 1850 & 1980 \\ 
        0.25 & 0.5 
        & 85.97 & 88.78 & 96.87 & 98.07 & 98.24 & 99.60 
        & 0.060 & 0.228 & 0.562 & 10.1 & 11.7 & 23.1 \\ 
        \bottomrule
    \end{tabular}
    }
    \caption{Objective relative to the hindsight optimum (in percent) and computation time (in seconds) for the online generalized assignment problem. ``OSO-$S$'': OSO algorithm with $S$ sample paths per iteration.}
    \label{tab:assignment_fixedC_app}
\end{table}

\begin{figure}
\centering
\subfloat[\small Objective (unimodal distribution, $\psi = 0$)]{\includegraphics[width=0.49\textwidth]{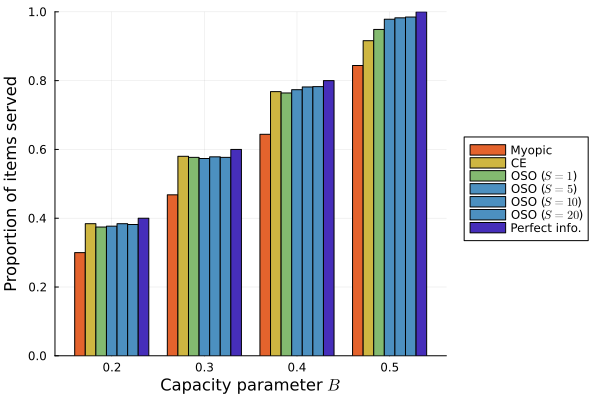}
\label{fig:assignment_fixedC_unimodal_groupedbar_obj_app}
}%
\subfloat[\small Computation time (unimodal distribution, $\psi = 0$)]{\includegraphics[width=0.49\textwidth]{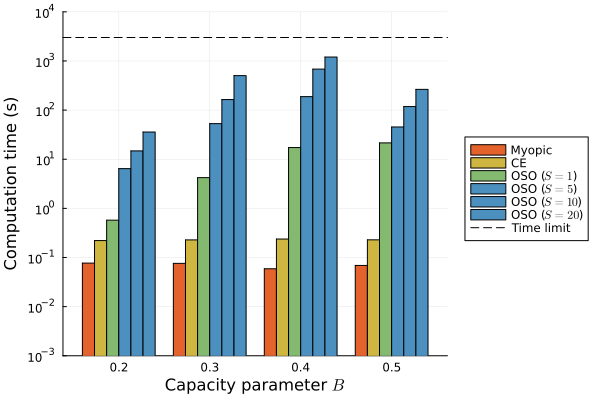}
\label{fig:assignment_fixedC_unimodal_groupedbar_time_app}
}%
\newline
\subfloat[\small Objective (bimodal distribution, $\psi = 1/4$)]{\includegraphics[width=0.49\textwidth]{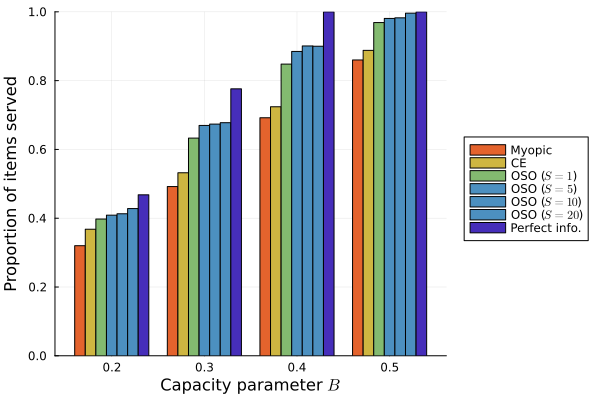}
\label{fig:assignment_fixedC_bimodal_groupedbar_obj_app}
}%
\subfloat[\small Computation time (bimodal distribution, $\psi = 1/4$)]{\includegraphics[width=0.49\textwidth]{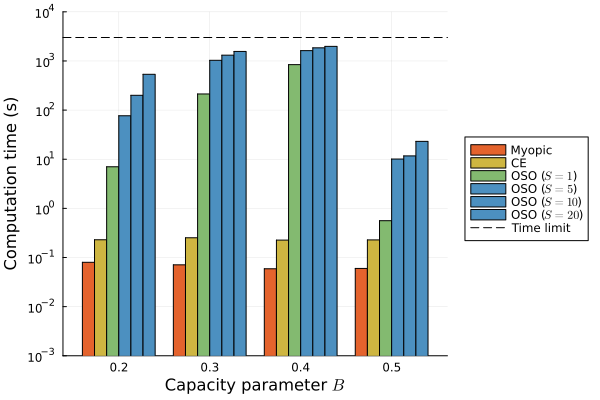}
\label{fig:assignment_fixedC_bimodal_groupedbar_time_app}
}%
\caption{Normalized objectives and computation times for the online generalized assignment problem.} 
\label{fig:assignment_fixedC_groupedbar_obj_time_app}
\end{figure}

\begin{table}
    \centering
    {
    \small
    \begin{tabular}{
        S[table-format=1.2] S[table-format=1.1] c
        S[table-format=2.2] S[table-format=2.2] r S[table-format=2.2] r S[table-format=2.2] c
        S[table-format=1.3] S[table-format=1.3] S[table-format=4.1] S[table-format=4.0]}
        \toprule
        & & & \multicolumn{6}{c}{Objective (\% of perfect-info)} & & \multicolumn{4}{c}{Computation time (s)}
        \\
        \cmidrule(lr){4-9} \cmidrule(lr){11-14}
        {$\psi$} & {$B$} 
        & & {Myopic} & {CE} & {(over Myo.)} & {OSO-1} & {(over CE)} & {OSO-5}
        & & {Myopic} & {CE} & {OSO-1} & {OSO-5}
        \\ 
        \midrule
        0.0 & 0.1 
        & & 78.80 & 84.76 & {$+7.6$\%} & 87.48 & {$+3.2$\%} & 90.68 
        & & 0.039 & 0.600 & 87.7 & 5770 \\ 
        0.0 & 0.2 
        & & 77.32 & 81.89 & {$+5.9$\%} & 91.77 & {$+12.1$\%} & 93.66 
        & & 0.040 & 0.701 & 442 & 7890 \\ 
        0.0 & 0.3 
        & & 80.27 & 83.33 & {$+3.8$\%} & 93.78 & {$+12.5$\%} & 95.11 
        & & 0.042 & 0.869 & 686 & 8310 \\ 
        \midrule
        0.25 & 0.1 
        & & 50.72 & 60.71 & {$+19.7$\%} & 84.33 & {$+38.9$\%} & 88.44 
        & & 0.040 & 0.637 & 150 & 7440 \\ 
        0.25 & 0.2 
        & & 54.11 & 66.00 & {$+22.0$\%} & 89.61 & {$+35.7$\%} & 93.24 
        & & 0.039 & 1.20 & 670 & 8490 \\ 
        0.25 & 0.3 
        & & 60.87 & 73.75 & {$+21.2$\%} & 92.71 & {$+25.7$\%} & 95.50
        & & 0.042 & 1.42 & 1890 & 6590 \\ 
        \bottomrule
    \end{tabular}
    }
    \caption{Objective relative to the hindsight optimum (in percent) and computation time (in seconds) for the online resource allocation problem. ``OSO-1'', ``OSO-5'': OSO algorithm with 1, 5 sample paths per iteration.}
    \label{tab:resourcealloc_app}
\end{table}

\begin{figure}
\centering
\subfloat[\small Objective (unimodal distribution, $\psi = 0$)]{\includegraphics[width=0.49\textwidth]{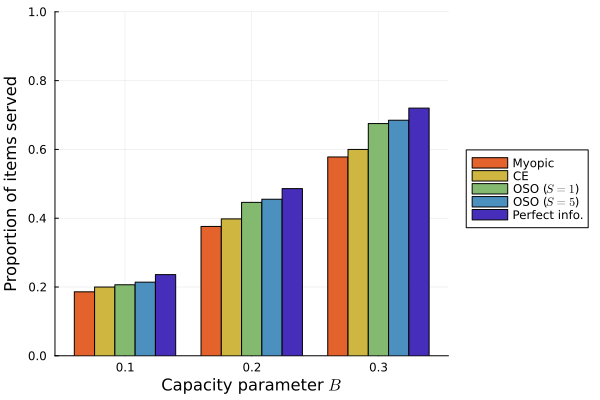}
\label{fig:resource_alloc_unimodal_groupedbar_obj_app}
}%
\subfloat[\small Computation time (unimodal distribution, $\psi = 0$)]{\includegraphics[width=0.49\textwidth]{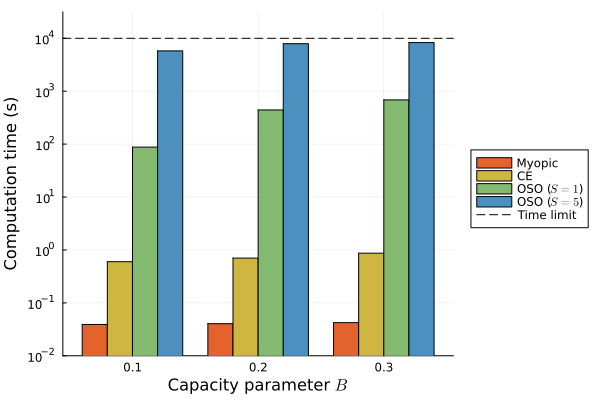}
\label{fig:resource_alloc_unimodal_groupedbar_time_app}
}%
\newline
\subfloat[\small Objective (bimodal distribution, $\psi = 1/4$)]{\includegraphics[width=0.49\textwidth]{figures/resource_alloc/groupedbar_objective_capacity_100_0.25.png}
\label{fig:resource_alloc_bimodal_groupedbar_obj_app}
}%
\subfloat[\small Computation time (bimodal distribution, $\psi = 1/4$)]{\includegraphics[width=0.49\textwidth]{figures/resource_alloc/groupedbar_time_capacity_100_0.25.png}
\label{fig:resource_alloc_bimodal_groupedbar_time_app}
}%
\caption{Normalized objectives and computation times for the online resource allocation problem.} 
\label{fig:resource_alloc_groupedbar_obj_time_app}
\end{figure}
\clearpage
\section{Online Batched Bin-packing} 
\label{app:binpacking}

\subsection{Problem Statement and MSSIP Formulation}

\begin{definition}[Online batched bin packing]
    Items arrive over $T$ time periods. All bins have capacity $B$. At each time period $t$, a batch $\calI^t$ of $q$ items are revealed. Each item $i \in \calI^t$ has size $V^t_i \in \{0, 1, \dots, B\}$. The objective is to pack items in as few bins as possible.
\end{definition}

We use the flow-based formulation from \cite{valerio1999exact}; this formulation exhibits much stronger scalability than other formulation with a looser linear relaxation, in our instances. This formulation relies on a network representation with node set: $\calN=\{0,\dots,B\}$, packing arcs: $\{(i,j):0\leq i<j\leq B\}$, and loss arcs corresponding to wasted capacity $\{(i,i+1):0\leq i<B\}$. A packing is characterized by a path from node 0 to node $B$. \Cref{fig:binpacking_flow} shows an example.

\begin{figure}[h!]
    \centering
    \begin{tikzpicture}
        \node[circle,thick,draw=black,fill=black,inner sep=0pt,minimum size=6pt] (0) at (0,0){};
        \node[circle,thick,draw=black,fill=black,inner sep=0pt,minimum size=6pt] (1) at (1.5,0){};
        \node[circle,thick,draw=black,fill=black,inner sep=0pt,minimum size=6pt] (2) at (3,0){};
        \node[circle,thick,draw=black,fill=black,inner sep=0pt,minimum size=6pt] (3) at (4.5,0){};
        \node[circle,thick,draw=black,fill=black,inner sep=0pt,minimum size=6pt] (4) at (6,0){};
        \node[circle,thick,draw=black,fill=black,inner sep=0pt,minimum size=6pt] (5) at (7.5,0){};
        \node[] at (8,.3) {$B=5$};
        \draw[->, thick, dotted] (0) -- (1);
        \draw[->, thick, dotted] (1) -- (2);
        \draw[->, thick, dotted] (2) -- (3);
        \draw[->, thick, dotted] (3) -- (4);
        \draw[->, thick] (4) -- (5);
        \draw[->, thick, color=myred] (0) to [bend left=45] (2);
        \draw[->, thick, dotted, color=myred] (1) to [bend left=45] (3);
        \draw[->, thick, color=myred] (2) to [bend left=45] (4);
        \draw[->, thick, dotted, color=myred] (3) to [bend left=45] (5);
        \draw[->, thick, dotted, color=myblue] (0) to [bend right=45] (3);
        \draw[->, thick, dotted, color=myblue] (1) to [bend right=45] (4);
        \draw[->, thick, dotted, color=myblue] (2) to [bend right=45] (5);
    \end{tikzpicture}
    \caption{Flow-based bin packing representation. Red (resp. blue) arcs denote items of size 2 (resp. 3). The bin capacity is 5. The solution in solid lines packs two items of size 2.}
    \label{fig:binpacking_flow}
\end{figure}

We define the following decision variables:
\begin{align*}
    x_{ij}
    & \,=\, 
    \text{number of items of size $j-i$ placed in any bin starting in ``position'' $i$}
    \\
    w_{j}
    & \,=\,
    \text{number of loss arcs used across all bins starting in ``position'' $j$}
    \\
    z
    & \,=\,
    \text{number of bins opened}
\end{align*}

The offline bin-packing formulation for a set of items $\calI$ is given as follows. 
\Cref{eq:binpacking_flow_objective} minimizes the number of bins. \Cref{eq:binpacking_flow_flowbalance} defines packing solutions as a path from the source to the sink, and \Cref{eq:binpacking_flow_counts} ensures that all items are placed in a bin (denoting $c_s(\calI)$ as the number of size-$s$ items in $\calI$).
\begin{subequations}
\label{eq:binpacking_flow}
\begin{alignat}{3}
    \label{eq:binpacking_flow_objective}
    \min \quad
    & 
    z
    \\
    \st \quad 
    & \sum_{i=0}^{j-1} x_{ij}
    - \sum_{k=j+1}^{B} x_{jk}
    = \begin{cases}
        - z + w_{j} & j = 0 
        \\
        - w_{j-1} + w_{j} & j \neq 0, B
        \\
        - w_{j-1} + z & j = B
    \end{cases}
    \label{eq:binpacking_flow_flowbalance}
    \\
    & \sum_{j=0}^{B-s} x_{j,j+s}
    = c_s(\calI)
    \quad \forall \ s =1, \dots, B
    \label{eq:binpacking_flow_counts}
    \\
    & z, \bx, \bw \text{ nonnegative integer}
    \label{eq:binpacking_flow_domain}
\end{alignat}
\end{subequations}

\begin{definition}
    Let $\calI$ be a set of items. We denote by $\calF(\calI)$ the feasible set of \Cref{eq:binpacking_flow}, projected on the $\bx$ and $z$ variables:
    \begin{alignat}{2}
        \calF(\calI) := \Set{
            \bx \in \mathbb{Z}_+^{B(B+1)/2}, 
            z \in \mathbb{Z}_+
            |
            \exists \ \bw \in \mathbb{Z}_+^B: \ \text{\Cref{eq:binpacking_flow_flowbalance,eq:binpacking_flow_counts}}
        }
    \end{alignat}
\end{definition}

In the online problem, items come in batch $\calI^t$ at time $t \in \calT$. We denote the uncertainty in batch $t$ by $\bXi^t = \bV^t$, with realized batches $\calI^t$ and sampled batches $\widetilde{\calI}^t$; we denote by $\bx^t, z^t$ the decision variables at time $t \in \calT$. Let also $\calI^{1:t-1} := \calI^1 \cup \dots \cup \calI^{t-1}$ denote the set of all past items, and by $\bx^{1:t-1} := \sum_{\tau=1}^{t-1} \bx^\tau$ and $z^{1:t-1} := \sum_{\tau=1}^{t-1} z^\tau$ the cumulative past decisions (vector of utilization and number of bins). We can express the per-period objective function and feasible set for the current decisions $(\bx^t, z^t)$ as:
\begin{alignat}{3}
    \label{eq:mssip_binpacking_flow_objective}
    f^t(\bx^t, z^t, \calI^{1:t}) 
    & := z^t
    \\
    \label{eq:mssip_binpacking_flow_feasibleset}
    \calF_t(\bx^{1:t-1}, z^{1:t-1}, \calI^{1:t})
    & := \Set{
        \bx^t \in \mathbb{Z}_+^{B(B+1)/2},
        \ z^t \in \mathbb{Z}_+
        |
        (\bx^{1:t-1} + \bx^t, 
        z^{1:t-1} + z^t) \in \calF(\calI^{1:t})
    }
\end{alignat}
Therefore, the multi-stage stochastic programming formulation can be expressed as follows:
\begin{alignat}{2}
\label{eq:mssip_binpacking_flow}
    \mathbb{E}_{\bXi^{1:T}}
    \biggl[
        \max_{(\bx^1, z^1) \in \calF_1(\calI^{1})} \biggl\{ 
            & f^1(\bx^1, z^1, \calI^{1})
            + 
            \mathbb{E}_{\bXi^{2:T}}
            \biggl[
                \max_{(\bx^2, z^2) \in \calF_2(\bx^{1}, z^{1}, \calI^{1:2})} \biggl\{ 
                    f^2(\bx^2, z^2, \calI^{1:2})
                    + \dots
                    \\
                    & + \condE{\bXi^{T}}{
                        \max_{(\bx^T, z^T) \in \calF_T(\bx^{1:T-1}, z^{1:T-1}, \calI^{1:T})} 
                        \Bigl\{ f^T(\bx^T, z^T, \calI^{1:T})\Bigr\}
                    } 
                \dots
                \biggr\}
            \biggr]
        \biggr\} 
    \biggr]
    \nonumber
\end{alignat}

The single-sample OSO algorithm for the online batched bin packing problem is given in \Cref{alg:OSO_binpacking}. We consider a variant where all uncertainty realizations are sampled at the beginning of the horizon, for ease of theoretical analysis. We define the following problem at time $t$:
\begin{definition}
    We denote by $\IPbp{t}{\calI^t, \widetilde{\calI}^{t+1:T}}{\overline{\bx}^{1:t-1}, \overline{z}^{1:t-1}}$ the following integer program which is solved at time $t$ of the OSO algorithm, giving optimal solution $(\widetilde{\bx}^t, \widetilde{\bx}^{t+1:T}, \widetilde{z}^t, \widetilde{z}^{t+1:T})$:
    \begin{subequations}
    \label{eq:binpacking_flow_OSO}
    \begin{alignat}{3}
        \label{eq:binpacking_flow_OSO_obj}
        \min \quad
        & 
        \overline{z}^{1:t-1} + z^t + z^{t+1:T}
        \\
        \label{eq:binpacking_flow_OSO_nowfeas}
        \st \quad 
        & 
        (\bx^t, z^t) \in \calF_t(\overline{\bx}^{1:t-1}, \overline{z}^{1:t-1}, \calI^{1:t})
        \\
        \label{eq:binpacking_flow_OSO_nextfeas}
        & 
        (\bx^\tau, z^\tau) \in \calF_t(
            \overline{\bx}^{1:t-1} + \bx^{t:\tau-1}, 
            \overline{z}^{1:t-1} + z^{t:\tau-1},
            \calI^{1:t} \cup \widetilde{\calI}^{t+1:\tau}
        )
        &\quad& 
        \forall \ \tau \in \{t+1, \dots, T\}
    \end{alignat}
    \end{subequations}
    We call $(\widetilde{\bx}^t, \dots, \widetilde{\bx}^{T})$ an optimal extension of $\overline{\bx}^{1:t-1}$ given $(\calI^t, \widetilde{\calI}^{t+1:T})$. 
    We also denote by $\OPTbp{t}{\calI^t, \widetilde{\calI}^{t+1:T}}{\overline{\bx}^{1:t-1}, \overline{z}^{1:t-1}}$ the corresponding optimal objective value, which is the minimum number of bins that can hold the items $\calI^{1:t} \cup \widetilde{\calI}^{t+1:T}$.
\end{definition}

\begin{algorithm}[h!]
    \caption{Single-sample OSO for the online batched bin packing problem.}
    \label{alg:OSO_binpacking}
    \begin{algorithmic} 
    \raggedright
    \item[] \textbf{Sample:} Sample batches $\widetilde{\calI}^1, \dots, \widetilde{\calI}^T$ from the common distribution $\calD$.
    \item Repeat, for $t \in \{1, \dots, T\}$:
    \begin{itemize}
    \item[] \textbf{Observe:} Observe true batch of items $\calI^t$.
    \item[] \textbf{Optimize:} Solve $\IPbp{t}{\calI^t, \widetilde{\calI}^{t+1:T}}{\overline{\bx}^{1:t-1}, \overline{z}^{1:t-1}}$, with optimal solution $(\widetilde{\bx}^t, \widetilde{\bx}^{t+1:T}, \widetilde{z}^t, \widetilde{z}^{t+1:T})$.
    \item[] \textbf{Implement:} Implement $\overline{\bx}^t = \widetilde{\bx}^t$ and $\overline{z}^t = \widetilde{z}^t$, discarding $\widetilde{\bx}^{t+1:T}$ and $\widetilde{z}^{t+1:T}$.
    \end{itemize}
\end{algorithmic}
\end{algorithm}

\subsection{Proof of \Cref{thm:binpacking}.}
\label{app:binpacking_proof}

The proof proceeds by tracking the evolution of $\OPTbp{t}{\widetilde{\calI}^t, \widetilde{\calI}^{t+1:T}}{\overline{\bx}^{1:t-1}, \overline{z}^{1:t-1}}$, that is, of the cost estimate given that the algorithm has already made decisions $\overline{\bx}^{1:t-1}$ using $\overline{z}^{1:t-1}$ bins. When $t=1$, this cost is $\OPTbp{1}{\widetilde{\calI}^1, \widetilde{\calI}^{2:T}}{\bo, 0}$, which is equal to the true optimum in expectation (since $\calI^t$ and $\widetilde{\calI}^t$ are sampled from the same distribution). At time $t = T+1$, this cost is $\OPTbp{T+1}{\varnothing, \varnothing}{\overline{\bx}^{1:T}, \overline{z}^{1:T}}$, which is the total cost of the OSO algorithm over the true instance. Therefore, to prove \Cref{thm:binpacking}, it suffices to bound:
\begin{align}
    \label{eq:costDiff}
    \OPTbp{t+1}{
        \widetilde{\calI}^{t+1}, \widetilde{\calI}^{t+2:T}
    }{
        \overline{\bx}^{1:t}, \overline{z}^{1:t}
    }
    - \OPTbp{t}{
        \widetilde{\calI}^{t}, \widetilde{\calI}^{t+1:T}
    }{
        \overline{\bx}^{1:t-1}, \overline{z}^{1:t-1}
    }.
\end{align}

Consider a fixed round $t$. 
Since $(\widetilde{\bx}^t, \widetilde{\bx}^{t+1:T}, \widetilde{z}^t, \widetilde{z}^{t+1:T})$ is optimal for $\IPbp{t}{\calI^t, \widetilde{\calI}^{t+1:T}}{\overline{\bx}^{1:t-1}, \overline{z}^{1:t-1}}$, and $\overline{\bx}^{t} := \widetilde{\bx}^{t}$, 
$(\overline{\bx}^t, \widetilde{\bx}^{t+1}, \dots, \widetilde{\bx}^{T})$ is an optimal extension of $\overline{\bx}^{1:t-1}$ given items $(\calI^t, \widetilde{\calI}^{t+1:T})$. Therefore, $(\widetilde{\bx}^{t+1}, \dots, \widetilde{\bx}^{T})$ is an optimal extension of $\overline{\bx}^{1:t}$ given items $(\widetilde{\calI}^{t+1}, \widetilde{\calI}^{t+2:T})$. That is:
\begin{align} 
    \label{obs:opt}
    \OPTbp{t}{
        \calI^t, \widetilde{\calI}^{t+1:T}
    }{
        \overline{\bx}^{1:t-1}, \overline{z}^{1:t-1}
    }
    \,=\,
    \OPTbp{t+1}{
        \widetilde{\calI}^{t+1}, \widetilde{\calI}^{t+2:T}
    }{
        \overline{\bx}^{1:t}, \overline{z}^{1:t}
    }
\end{align}
Thus, upper bounding the difference in \Cref{eq:costDiff} is equivalent to upper bounding 
\begin{align}
   \label{eq:costDiff2}
    \OPTbp{t}{
        \calI^t, \widetilde{\calI}^{t+1:T}
    }{
        \overline{\bx}^{1:t-1}, \overline{z}^{1:t-1}
    }
    - \OPTbp{t}{
        \widetilde{\calI}^{t}, \widetilde{\calI}^{t+1:T}
    }{
        \overline{\bx}^{1:t-1}, \overline{z}^{1:t-1}
    }.
\end{align}
That is, given $\overline{\bx}^{1:t-1}$, we need to show that the total cost is not significantly impacted whether the next batch is $\calI^t$ (the actual one) or $\widetilde{\calI}^t$ (the sampled one). We leverage ``coupling'' between the item sizes in $\calI^t$ and $\widetilde{\calI}^t$ to design an assignment for $\calI^t$ from an assignment for $\widetilde{\calI}^t$. We make use of \emph{monotone matchings}, which match two values if the latter is at least as large as the former. 

\begin{definition}[Monotone matching] 
    Given two sequences  $a_1, \dots, a_n \in \mathbb{R}$ and $b_1, \dots, b_n \in \mathbb{R}$, a \emph{monotone matching} $\pi$ from the $a_\ell$'s to the $b_\ell$'s is an injective function from a subset $L \in \{1, \dots, n\}$ to $\{1, \dots, n\}$ such that $a_\ell \leq b_{\pi(\ell)}$ for all $\ell \in L$. We say that $a_\ell$ is \emph{matched} to $b_{\pi(\ell)}$ if $\ell \in L$, and \emph{unmatched} otherwise. 
\end{definition}

Intuitively, thinking of  $a_1, \dots, a_n$ and $b_1, \dots, b_n$ as sequences of item sizes, a monotone matching indicates that, if item $b_{\pi(\ell)}$ is assigned to some bin, then we can replace it with item $a_\ell$ without violating the bin's capacity.
In other words, we can use an assignment of the items $b_1, \dots, b_n$ to come up with an assignment of the matched items in $a_1, \dots, a_n$ using the same bins. \cite{rhee1993lineii} showed that if the two sequences are i.i.d. from the same distribution, then almost all items can be matched using a monotone matching.

\begin{theorem}[Monotone Matching Theorem~\citep{rhee1993lineii}]
    \label{thm:rheeTal}
    Define independent random variables $A_1, \dots, A_n$ and $B_1, \dots, B_n$ from distribution $\calD$ over $[0,1]$. There is a constant $c$ such that with probability at least $1 - \exp{-c \log^{3/2} n}$ there is a monotone matching $\pi$ of the $A_\ell$'s to the $B_\ell$'s where at most $c \sqrt{n} \log^{3/4} n$ of the $A_\ell$'s are unmatched.
\end{theorem}

Using this result we can upper bound the difference given in \Cref{eq:costDiff2} as follows. 

\begin{lemma} 
    \label{lemma:exchange}
    There is a constant $c$ such that with probability at least $1 - \exp{-c\log^{3/2} q}$,
    \begin{equation}
        \OPTbp{t}{
            \calI^t, \widetilde{\calI}^{t+1:T}
        }{
            \overline{\bx}^{1:t-1}, \overline{z}^{1:t-1}
        } 
        \,-\, 
        \OPTbp{t}{
            \widetilde{\calI}^t, \widetilde{\calI}^{t+1:T}
        }{
            \overline{\bx}^{1:t-1}, \overline{z}^{1:t-1}
        } 
        \,\leq\, 
        c \sqrt{q} \log^{3/4} q.
    \end{equation}
\end{lemma}

\begin{proof}{Proof of \Cref{lemma:exchange}.}
The strategy will be to start with the packing solution for $\IPbp{t}{
    \widetilde{\calI}^t, \widetilde{\calI}^{t+1:T}
}{
    \overline{\bx}^{1:t-1}, \overline{z}^{1:t-1}
}$ and construct a ``good enough'' packing solution for $\IPbp{t}{
    \calI^t, \widetilde{\calI}^{t+1:T}
}{
    \overline{\bx}^{1:t-1}, \overline{z}^{1:t-1}
}$ i.e., one that uses at most $c \sqrt{q} \log^{3/4} q$ more bins. 
Then the optimal solution $\OPTbp{t}{
    \calI^t, \widetilde{\calI}^{t+1:T}
}{
    \overline{\bx}^{1:t-1}, \overline{z}^{1:t-1}
}$ will also use at most $c \sqrt{q} \log^{3/4} q$ more bins than $\OPTbp{t}{
    \widetilde{\calI}^t, \widetilde{\calI}^{t+1:T}
}{
    \overline{\bx}^{1:t-1}, \overline{z}^{1:t-1}
}$, completing the proof.

Let $V^t_1, \dots, V^t_q$ be the realized item sizes in the $t$-th batch, and $\widetilde{V}^t_1, \dots, \widetilde{V}^t_q$ be the sampled item sizes in the $t$-th batch. Let $(\widetilde{\bx}^t, \widetilde{\bx}^{t+1:T}, \widetilde{z}^t, \widetilde{z}^{t+1:T})$ be the optimal solution for $\IPbp{t}{
    \widetilde{\calI}^t, \widetilde{\calI}^{t+1:T}
}{
    \overline{\bx}^{1:t-1}, \overline{z}^{1:t-1}
}$ given the previous assignments $\overline{\bx}^{1:t-1}$, with auxillary variables $\bw^{\text{now}}$ and $\bw^{\text{next}}$ corresponding to \Cref{eq:binpacking_flow_OSO_nowfeas,eq:binpacking_flow_OSO_nextfeas}.
We use a monotone matching to construct a feasible solution for $\IPbp{t}{
    \calI^t, \widetilde{\calI}^{t+1:T}
}{
    \overline{\bx}^{1:t-1}, \overline{z}^{1:t-1}
}$. 

Let $\pi$ be a monotone matching from $\{V^t_1 / B, \dots, V^t_q / B\}$, the normalized true item sizes in batch $t$, to $\{\widetilde{V}^t_1 / B, \dots, \widetilde{V}^t_q / B\}$, the normalized sampled item sizes in batch $t$, given by \Cref{thm:rheeTal}. This implies that for all $\ell \in \{1, \dots, q\}$ matched by $\pi$, we have $V^t_\ell \leq \widetilde{V}^t_{\pi(\ell)}$. 
We construct a new solution $(\widehat{\bx}^t, \widehat{z}^t, \widehat{\bw}^{\text{now}})$ from $(\widetilde{\bx}^t, \widetilde{z}^t, {\bw}^{\text{now}})$ according to \Cref{alg:exchange}. The overall idea is that for matched elements in $\{V^t_1, \dots, V^t_q\}$, we can replace sampled items with true items that are not larger than those sampled items. True items that are unmatched are assigned to a new bin each.
\begin{algorithm}[h!]
    \caption{Algorithm for constructing a solution $(\widehat{\bx}^t, \widehat{z}^t, \widehat{\bw}^{\text{now}})$ from $(\widetilde{\bx}^t, \widetilde{z}^t, \bw^{\text{now}})$.}
    \label{alg:exchange}
    \begin{algorithmic} 
    \State \textbf{Initialization:} Define $\pi$ as a matching between item sizes in $\calI^t$ and item sizes in $\widetilde{\calI}^t$.
    \State Define $\widehat{\bx}^t \leftarrow \widetilde{\bx}^t$, $\widehat{z}^t \leftarrow \widetilde{z}^t$ and $\widehat{\bw}^{\text{now}} \leftarrow \bw^{\text{now}}$.
    \State Define $L, M \subset \{1, \dots, q\}$ as domain, range of $\pi$.
    \State Define $L^c \leftarrow \{1, \dots, q\} \setminus L$ and $M^c \leftarrow \{1, \dots, q\} \setminus M$.
    \For{$\ell \in L$}
        \Comment{Replace item $\widetilde{V}^t_{\pi(\ell)}$ with item $V^t_{\ell}$ in the same bin}
        \State Define $m := \pi(\ell)$.
        \State Define $s := V^t_\ell$ and $\widetilde{s} := \widetilde{V}^t_{m}$; $V^t_\ell \leq \widetilde{V}^t_{m}$, so $s \leq \tilde{s}$.
        \State Define $j$ as the minimum such that $\widehat{x}^t_{j,j+\widetilde{s}} \geq 1$.
        \State Replace packing arc $(j, j + \widetilde{s})$ with a packing arc $(j, j + s)$ and empty arcs from $s$ to $\widetilde{s}$:
        \begin{subequations}
        \begin{alignat}{2}
            \widehat{x}^t_{j,j+\widetilde{s}}
            & \,\leftarrow\, \widehat{x}^t_{j,j+\widetilde{s}} \ - 1
            \\
            \widehat{x}^t_{j,j+s}
            & \,\leftarrow\, \widehat{x}^t_{j,j+s} \ + 1
            \\
            \widehat{w}^{\text{now}}_{k} 
            & \,\leftarrow\, \widehat{w}^{\text{now}}_{k} \ + 1
            &\quad& \forall \ k \in \{s, \dots, \widetilde{s} - 1\}
        \end{alignat}
        \end{subequations}
    \EndFor
    \For{$\ell \in L^c$}
        \Comment{Replace any remaining $\widetilde{V}^t_{m}$ with item $V^t_{\ell}$ in its own bin}
        \State Pick any $m \in M^c$; update $M^c \leftarrow M^c \setminus \{m\}$.
        \State Define $s := V^t_\ell$ and $\widetilde{s} := \widetilde{V}^t_{m}$.
        \State Define $j$ as the minimum such that $\widehat{x}^t_{j,j+\widetilde{s}} \geq 1$.
        \State Replace packing arc $(j, j + \widetilde{s})$ with empty arcs from $j$ to $j + \widetilde{s}$;
        \begin{subequations}
        \begin{alignat}{2}
            \widehat{x}^t_{j,j+\widetilde{s}}
            & \,\leftarrow\, \widehat{x}^t_{j,j+\widetilde{s}} \ - 1
            \\
            \widehat{w}^{\text{now}}_{k} 
            & \,\leftarrow\, \widehat{w}^{\text{now}}_{k} \ + 1
            &\quad& \forall \ k \in \{j, \dots, j + \widetilde{s} - 1\}
        \end{alignat}
        \end{subequations}
        \State Add flow on packing arc $(0, s)$ and empty arcs for the rest of the box:
        \begin{subequations}
        \begin{alignat}{2}
            \widehat{x}^t_{0,s}
            & \,\leftarrow\, \widehat{x}^t_{0,s} \ + 1
            \\
            \widehat{z}^t 
            & \,\leftarrow\, \widehat{z}^t \ + 1
            \\
            \widehat{w}^{\text{now}}_{k} 
            & \,\leftarrow\, \widehat{w}^{\text{now}}_{k} \ + 1
            &\quad& \forall \ k \in \{s, \dots, B - 1\}
        \end{alignat}
        \end{subequations}
    \EndFor
    \State \textbf{return} $(\widehat{\bx}^t, \widehat{\bw}^{\text{now}})$
\end{algorithmic}
\end{algorithm}

We first verify that $(\widehat{\bx}^t, \widetilde{\bx}^{t+1:T}, \widehat{z}^t, \widetilde{z}^{t+1:T})$ is a solution that packs batches $\calI^t, \widetilde{\calI}^{t+1:T}$ given the history, i.e. is a feasible solution to $\IPbp{t}{
    \calI^t, \widetilde{\calI}^{t+1:T}
}{
    \overline{\bx}^{1:t-1}, \overline{z}^{1:t-1}
}$. We show this in two parts:

\begin{enumerate}
    \item \underline{\Cref{eq:binpacking_flow_OSO_nowfeas}}. We claim that $\widehat{\bw}^{\text{now}}$ would certify that $(\overline{\bx}^{1:t-1} + \widehat{\bx}^t, \overline{z}^{1:t-1} + \widehat{z}^t)$ belongs in $\calF(\calI^{1:t-1} \cup \calI^t)$. First note that $\bw^{\text{now}}$ would certify that $(\overline{\bx}^{1:t-1} + \widetilde{\bx}^t, \overline{z}^{1:t-1} + \widetilde{z}^t)$ belongs in $\calF(\calI^{1:t-1} \cup \widetilde{\calI}^t)$. Next, at each iteration of either for loop in \Cref{alg:exchange}, flow conservation (\Cref{eq:binpacking_flow_flowbalance}) is maintained at all nodes $j \neq 0, B$; flow balance is also maintained at $j = 0, B$ regardless of whether a new bin is opened for unmatched items in $\calI^t$. Hence flow conservation remains satisfied for $(\overline{\bx}^{1:t-1} + \widehat{\bx}^t, \overline{z}^{1:t-1} + \widehat{z}^t, \widehat{\bw}^{\text{now}})$.

    Also, \Cref{eq:binpacking_flow_counts} is satisfied. This is because $\sum_{j=0}^{B-s} \widetilde{x}^t_{j,j+s} = c_s(\widetilde{\calI}^t)$ (since $\sum_{j=0}^{B-s} \overline{x}^{1:t-1}_{j,j+s} + \widetilde{x}^{t}_{j,j+s} = c_s(\calI^{1:t-1} \cup \widetilde{\calI}^t)$ and $\sum_{j=0}^{B-s} \overline{x}^{1:t-1}_{j,j+s} = c_s(\calI^{1:t-1})$), and at each iteration of either for loop in \Cref{alg:exchange} one item in $\widetilde{\calI}^t$ is swapped out for one item in $\calI^t$. Therefore, at termination, $\sum_{j=0}^{B-s} \widehat{x}^t_{j,j+s} = c_s(\calI^t)$, and $\sum_{j=0}^{B-s} \overline{x}^{1:t-1}_{j,j+s} + \widehat{x}^{t}_{j,j+s} = c_s(\calI^{1:t-1} \cup \calI^t)$. Hence, $(\widehat{\bx}^t, \widehat{z}^t)$ satisfies \Cref{eq:binpacking_flow_OSO_nowfeas}: $(\overline{\bx}^{1:t-1} + \widehat{\bx}^t, \overline{z}^{1:t-1} + \widehat{z}^t) \in \calF(\calI^{1:t-1} \cup \calI^t)$.

    \item \underline{\Cref{eq:binpacking_flow_OSO_nextfeas}}. We claim that $\widehat{\bw}^{\text{now}} + \bw^{\text{next}} -\bw^{\text{now}}$ would certify that $(\overline{\bx}^{1:t-1} + \widehat{\bx}^t + \widetilde{\bx}^{t+1:T}, \overline{z}^{1:t-1} + \widehat{z}^t + \widetilde{z}^{t+1:T})$ belongs in $\calF(\calI^{1:t-1} \cup \calI^t \cup \widetilde{\calI}^{t+1:T})$. Firstly it is positive, since $\widehat{\bw}^{\text{now}}$ starts from $\bw^{\text{now}}$ and is never decremented in \Cref{alg:exchange}. Next, \Cref{eq:binpacking_flow_flowbalance} is satisfied for:
    \begin{itemize}
        \item $\left(
            \overline{\bx}^{1:t-1}
            + \widetilde{\bx}^{t},
            \overline{z}^{1:t-1}
            + \widetilde{z}^{t},
            {\bw}^{\text{now}}
        \right)$ and $\left(
            \overline{\bx}^{1:t-1}
            + \widetilde{\bx}^{t},
            + \widetilde{\bx}^{t+1:T},
            \overline{z}^{1:t-1}
            + \widetilde{z}^{t},
            + \widetilde{z}^{t+1:T},
            {\bw}^{\text{next}}
        \right)$, and therfore the difference $\left(
            \widetilde{\bx}^{t+1:T},
            \widetilde{z}^{t+1:T},
            {\bw}^{\text{next}} - {\bw}^{\text{now}}
        \right)$;
        \item $\left(
            \overline{\bx}^{1:t-1}
            + \widehat{\bx}^{t},
            \overline{z}^{1:t-1}
            + \widehat{z}^{t},
            \widehat{\bw}^{\text{now}}
        \right)$ from above;
        \item and therefore the sum $\left(
            \overline{\bx}^{1:t-1}
            + \widehat{\bx}^{t},
            + \widetilde{\bx}^{t+1:T},
            \overline{z}^{1:t-1}
            + \widehat{z}^{t},
            + \widetilde{z}^{t+1:T},
            \widehat{\bw}^{\text{now}}
            + {\bw}^{\text{next}} 
            - {\bw}^{\text{now}}
        \right)$.
    \end{itemize} Finally, \Cref{eq:binpacking_flow_counts} is satisfied because $\widetilde{\bx}^{t+1:T}$ satisfies the counts of $\widetilde{\calI}^{t+1:T}$ for each item size.
\end{enumerate}

We next evaluate the quality of the constructed solution $(\widehat{\bx}^t, \widetilde{\bx}^{t+1:T}, \widehat{z}^t, \widetilde{z}^{t+1:T})$. This is by definition equal to $\overline{z}^{1:t-1} + \widehat{z}^t + \widetilde{z}^{t+1:T}$. $\widehat{z}^t$ is equal to $\widetilde{z}^t$ plus the number of unmatched elements in $\pi$, which from \Cref{thm:rheeTal} is at most $c \sqrt{q} \log^{3/4} q$ with probability at least $1 - \exp{-c \log^{3/2} q}$. Therefore, with probability at least $1 - \exp{-c \log^{3/2} q}$, we have:
\begin{alignat}{2}
    \OPTbp{t}{
        \calI^t, \widetilde{\calI}^{t+1:T}
    }{
        \overline{\bx}^{1:t-1}, \overline{z}^{1:t-1}
    } 
    & \leq 
    \overline{z}^{1:t-1} + \widehat{z}^t + \widetilde{z}^{t+1:T}
    \\
    & \leq 
    \overline{z}^{1:t-1} + \widetilde{z}^t + \widetilde{z}^{t+1:T} 
    + c \sqrt{q} \log^{3/4} q
    \\
    & = \OPTbp{t}{
        \widetilde{\calI}^t, \widetilde{\calI}^{t+1:T}
    }{
        \overline{\bx}^{1:t-1}, \overline{z}^{1:t-1}
    } 
    + c \sqrt{q} \log^{3/4} q
\end{alignat}
which concludes the lemma.\myqed
\end{proof}

\subsubsection*{Proof of \Cref{thm:binpacking}.}

By taking a union bound, \Cref{lemma:exchange} holds for all $t \in \{1, \dots, T\}$ with probability at least $1 - T\, \exp{-c \log^{3/2} q}$. Under such event, using \Cref{obs:opt} we have
\begin{align}
    \OPTbp{t+1}{
        \widetilde{\calI}^{t+1}, \widetilde{\calI}^{t+2:T}
    }{
        \overline{\bx}^{1:t}, 
        \overline{z}^{1:t}
    }
    & \,=\, 
    \OPTbp{t}{
        \calI^t, \widetilde{\calI}^{t+1:T}
    }{
        \overline{\bx}^{1:t-1}, 
        \overline{z}^{1:t-1}
    }
    \\
    & \,\leq\,
    \OPTbp{t}{
        \widetilde{\calI}^t, \widetilde{\calI}^{t+1:T}
    }{
        \overline{\bx}^{1:t-1}, 
        \overline{z}^{1:t-1}
    }
    + c \sqrt{q} \log^{3/4} q
\end{align}
Telescoping this inequality over $t \in \{1, \dots, T\}$, we obtain:
\begin{align}
    \OPTbp{T+1}{
        \varnothing, 
        \varnothing
    }{
        \overline{\bx}^{1:T}, 
        \overline{z}^{1:T}
    }
    & \,\leq\,
    \OPTbp{1}{
        \widetilde{\calI}^1, \widetilde{\calI}^{2:T}
    }{
        \bo, 
        0
    }
    + c T \sqrt{q} \log^{3/4} q
\end{align}

Recall that $\OPTbp{T+1}{
    \varnothing, 
    \varnothing
}{
    \overline{\bx}^{1:T}, 
    \overline{z}^{1:T}
}$ is the cost of our OSO algorithm on the true items, denoted by $\ALG(\calI^{1:T})$, and that 
$\OPTbp{1}{
    \widetilde{\calI}^1, \widetilde{\calI}^{2:T}
}{
    \bo, 
    0
}$ is the offline optimum for the sampled items, denoted by $\OPT(\widetilde{\calI}^{1:T})$. Thus, with probability at least $1 - T\, \exp{-c \log^{3/2} q}$, we have:
\begin{equation}
    \ALG(\calI^{1:T}) 
    \,\leq\, 
    \OPT(\widetilde{\calI}^{1:T}) 
    + c T \sqrt{q} \log^{3/4} q.
\end{equation}

Let $G$ be the event that this inequality holds, and let $G^c$ be its complement. Note that $G$ is a random event that depends on both $\calI^{1:T}$ (the problem instance) and $\widetilde{\calI}^{1:T}$ (the sampling procedure). Taking expectations over the true item sizes, we obtain:
\begin{equation}
    \condEmid{\calI^{1:T}}{\ALG(\calI^{1:T})}{G}
    \,\leq\, 
    \OPT(\widetilde{\calI}^{1:T})
    + c T \sqrt{q} \log^{3/4} q.
\end{equation}

Therefore, since all items can fit in $n$ bins, we have:
\begin{align}
    \condE{\calI^{1:T}}{\ALG(\calI^{1:T})}
    & = 
    \condEmid{\calI^{1:T}}{\ALG(\calI^{1:T})}{G} \Prob{G}
    + \condEmid{\calI^{1:T}}{\ALG(\calI^{1:T})}{G^c} (1-\Prob{G})
    \\
    & \leq 
    \condEmid{\calI^{1:T}}{\ALG(\calI^{1:T})}{G} \cdot 1
    + n T \exp{-c \log^{3/2} q}
    \\
    & \leq 
    \OPT(\widetilde{\calI}^{1:T})
    + c T \sqrt{q} \log^{3/4} q
    + n T \exp{-c \log^{3/2} q}
\end{align}
and taking a further expectation over the sampled item sizes gives:
\begin{align}
    \E{\ALG(\calI^{1:T})}
    & \leq 
    \E{\OPT(\widetilde{\calI}^{1:T})}
    + c T \sqrt{q} \log^{3/4} q
    + n T \exp{-c \log^{3/2} q}
    \\
    & = 
    \E{\OPT({\calI}^{1:T})}
    + c T \sqrt{q} \log^{3/4} q
    + n T \exp{-c \log^{3/2} q}
\end{align}
Finally, the assumption on $q$ guarantees that $n \exp{-c \log^{3/2} q}\leq \sqrt{q} \log^{3/4} q$, so the expected cost is at most $\E{\OPT({\calI}^{1:T})} +  \calO(T \sqrt{q}\, \log^{3/4} q)$. We conclude by leveraging the fact that $T = \frac{n}{q}$. \myqed

\subsection{Computational Results}
\label{app:binpacking_comp}

We compare the OSO algorithm to the resolving heuristics benchmarks for the online batched bin packing problem. The OSO algorithm provisions for a problem-specific regularizer; we add the following regularizer to the OSO problem solved at time $t$ to encourage packing items into fuller bins:
\begin{equation}
    \Psi(\bx^t) 
    = \frac{1}{|\calI^t|} \sum_{i=0}^{B-1} \sum_{j=i+1}^B 
    \left( 1 - \frac{j^2}{B^2} \right) x^t_{ij}
\end{equation}
This regularizer is similar to other approaches for online bin-packing. \cite{gupta_interior-point-based_2020} use a penalty of the form $\exp{- \varepsilon_t N_t(h)}$, where $N_t(h)$ denotes the number of bins filled to $h$ at time $t$ to discourage actions which deplete bins of levels with small $N(h)$. Instead, since we investigate bin packing instances with relatively large bins and fewer items, $N(h)$ is often 1 or 0 in our context; due to our batched setting, our regularizer prioritizes packing items into fuller bins. 

The problem that OSO solves at each iteration $t \in \calT$ is therefore given by:
\begin{subequations}
\label{eq:binpacking_flow_OSO_Omega}
\begin{alignat}{3}
   \label{eq:binpacking_flow_OSO_obj_Omega}
   \min \quad
   & 
   \overline{z}^{1:t-1} + z^t + z^{t+1:T} + \Psi(\bx^t)
   \\
   \label{eq:binpacking_flow_OSO_nowfeas_Omega}
   \st \quad 
   & 
   (\bx^t, z^t) \in \calF_t(\overline{\bx}^{1:t-1}, \overline{z}^{1:t-1}, \calI^{1:t})
   \\
   \label{eq:binpacking_flow_OSO_nextfeas_Omega}
   & 
   (\bx^\tau, z^\tau) \in \calF_t(
       \overline{\bx}^{1:t-1} + \bx^{t:\tau-1}, 
       \overline{z}^{1:t-1} + z^{t:\tau-1},
       \calI^{1:t} \cup \widetilde{\calI}^{t+1:\tau}
   )
   &\quad& 
   \forall \ \tau \in \{t+1, \dots, T\}
\end{alignat}
\end{subequations}

We construct online batched bin packing instances with $T$ time periods, each with a batch of $|\calI^t| = q$ items. Item sizes are drawn from a uniform distribution over $\{0, 1, \dots, 100\}$ with $B = 100$. For each combination of parameters, we generate 10 random instances and, for each one, we run OSO 5 times. We impose a time limit of 100 seconds for the integer program solved at each iteration. \Cref{tab:binpacking_flow_obj_time} reports the objective values and computation times.

\begin{table}
    \centering
    {
    \setlength{\tabcolsep}{4pt}
    \begin{tabular}{
        S
        *{2}{S[table-format=2]}
        *{4}{S[table-format=+1.3\%,retain-explicit-plus=true]}
        S[table-format=2.3]
        S[table-format=3.3]
        S[table-format=3.2]
        S[table-format=4.1]
    }
        \toprule
        & & & 
        \multicolumn{4}{c}{Objective increase}
        & 
        \multicolumn{4}{c}{Computation time (s)}
        \\
        \cmidrule(lr){4-7}
        \cmidrule(lr){8-11}
        {With $\Psi(\cdot)$?}
        & {$T$} & {$q$}
        & {Myopic} & {CE} & {OSO-1} & {OSO-5}
        & {Myopic} & {CE} & {OSO-1} & {OSO-5}
        \\
        \midrule
        {\xmark} & 16 & 64 & +4.794\% & +4.490\% & +3.634\% & +0.933\% & 2.811 & 4.780 & 11.93 & 393.7 \\ 
        {\xmark} & 32 & 32 & +6.693\% & +6.440\% & +4.942\% & +1.153\% & 3.405 & 6.956 & 14.46 & 572.5 \\ 
        {\xmark} & 64 & 16 & +6.627\% & +8.198\% & +5.536\% & +1.506\% & 4.434 & 8.428 & 20.25 & 398.3 \\ 
        \midrule
        {\cmark} & 16 & 64 & +1.439\% & +1.363\% & +1.318\% & +0.715\% & 52.56 & 74.63 & 188.0 & 437.5 \\ 
        {\cmark} & 32 & 32 & +1.811\% & +2.007\% & +1.915\% & +1.212\% & 73.60 & 108.4 & 211.8 & 1026 \\ 
        {\cmark} & 64 & 16 & +2.096\% & +2.154\% & +2.141\% & +1.420\% & 5.898 & 69.80 & 159.8 & 850.6 \\ 
        \bottomrule 
    \end{tabular}
    }
    \caption{Geometric mean of percentage increase in bins opened over hindsight-optimal benchmark, and mean computational time in seconds. Averages taken over 10 random instances, and 5 random samples per instance for OSO. ``OSO-1'', ``OSO-5'': OSO algorithm with 1, 5 sample paths per iteration.}
    \label{tab:binpacking_flow_obj_time}
\end{table}

These results confirm our insights from the online resource allocation problem, showing that OSO improves upon the myopic and CE solutions, and that performance improves with more samples per iteration at the cost of longer computation times. Without the regularizer, the CE solution barely improves upon the myopic solution with large and medium batch sizes, and actually leads to deteriorated solutions with small batch sizes; in comparison, OSO-1 consistently returns higher-quality solutions than both benchmarks, with reductions in wasted capacity around 1--2 percentage points. Increasing the number of sample paths can achieve further cost improvements, albeit with one to two orders of magnitude increases in computational times. Moreover, these results show the impact of the regularizer in the online batched bin packing problem. Adding the regularizer can result in solution improvements across the board, at the cost of longer computational times. Still, the OSO-5 solution without the regularizer outperforms all benchmarks with the regularizer, further demonstrating the benefits of the sampling approach at the core of the OSO algorithm. Altogether, these results highlight the role of sampling and re-optimization to manage online arrivals in batched bin packing.
\section{Rack Placement}
\label{app:rackplacement}

Here, we provide details on our experimental setup for the rack placement problem, and the modeling modifications made during the deployment process.

\subsection{Experimental Setup}
\label{app:rackplacement_comp_setup}

We build a simulated datacenter with two identical rooms. Each room has 4 top-level UPS devices; each UPS device is connected to 6 mid-level PDU devices, and each PDU device is connected to 3 leaf-level PSU devices. 
The regular capacity of each mid-level PDU and leaf-level PSU is respectively 20\% and 60\% of their parent's capacity. 
The top-level UPS devices have regular capacities equal to 75\% of their failover capacities, while the regular capacities of the PDU and PSU devices is 50\% of their failover capacities. Each room has 36 rows, each with 20 tiles. Each row is connected to two PSU devices in the same room with different parent PDU and UPS devices.

The reward is identical across demand requests, set to $r_i = 200$, while the number of racks $n_i$, power requirement per rack $\rho_i$, and cooling requirement per rack $\gamma_i$ for demand request $i$ are constructed from empirical distributions. 

In computational experiments, the perfect-information benchmark was solved with a 1-hour time limit. The resolving heuristics (OSO and the myopic and CE benchmarks) were solved with a 300-second limit per iteration and a 1\% optimality gap. We considered instances with a total of 150 items, thus 150, 30, or 15 time periods (corresponding to batch sizes of 1, 5, 10 respectively).

\subsection{Modeling Modifications in Production} 
\label{app:rackplacement_production_objectives}

We detail the modifications we made to our rack placement model discussed in \Cref{sec:real} to closely align recommendations with real-world considerations and preferences from data center managers. These modifications come in the form of the following secondary objectives. Throughout, we applied a small weight to these objectives to retain the primary goal of maximizing data center utilization.
\begin{itemize}
\item[--] {\it Room minimization.} 
This objective incentivizes compact data center configurations to reduce operational overhead. The room of row $r\in \calR$ is denoted by $\text{room}(r)\in \mathcal{M}$. 
We add a new binary variable $w_{im}^t$ denoting if room $m$ contains demand $i \in \calI_t$.
We add a term $- \sum_{i \in \calI_t} \sum_{m \in \calM} \lambda_m w_{im}^t$ to the objective to penalize placements in emptier rooms, where $\lambda_m$ is larger for emptier rooms. 
We link this variable to the placement decisions $\by^t$:
\begin{align}
    \label{eq:room_min}
    y_{ir}^t \leq w_{i,\text{room}(r)}^t
    &&
    \ \forall \ i \in \calI_t, 
    \ \forall \ r \in \calR
\end{align}

\item[--] {\it Row minimization.} 
This objective also incentivizes compact configurations to reduce operational overhead. 
We add a new binary variable $z_r^t$ indicating whether row $r \in \calR$ contains racks from the current demand batch $\calI_t$. 
We add a term $-\sum_{r \in \calR} \theta_r z_r^t$ to the objective where $\theta_r$ is larger for rows $r \in \calR$ with fewer placed racks. 
We add the following linking constraints:
\begin{align}
    \label{eq:resourceProfileSizeWithG}
    y_{ir}^t & \leq z_r^t
    &&
    \ \forall \ i \in \calI_t,
    \ \forall \ r \in \calR
\end{align}

\item[--] {\it Tile group minimization.} This objective incentivizes placing multi-rack reservations on identical tile groups to facilitate customer service down the road and, again, to reduce overhead---in practice, tiles belonging to the same tile groups are located in the same part of the row, so this objective promotes contiguity. We add a binary variable $v_{ij}^t$ denoting if tile group $j$ is used by demand $i \in \calI_t$, penalizing it by a parameter $\tau$. We add the following linking constraint:
\begin{align}
    \label{eq:tilegroup_min}
    x_{ij}^t \leq n_i \cdot v_{ij}^t
    &&
    \ \forall \ i \in \calI_t, 
    \ \forall \ j \in \calJ
\end{align}

\item[--] {\it Power balance.} This objective encourages balanced power loads to avoid overloads and mitigate maintenance operations. Let $\calP_m^{UPS}$ store top-level power devices in room $m \in \calM$. The power devices in $\calP^{UPS}_m$ share a load of $\sum_{p \in \calP^{UPS}_m} P_{p}$; all capacities $P_p$ are identical in practice; and there are $\binom{|\calP^{UPS}_m|}{2}$ pairs of distinct power devices in room $m$. Accordingly, if all pairs of distinct power devices share the same load, each pair will power a load of $\frac{1}{\binom{|\calP^{UPS}_m|}{2}}\sum_{p' \in \calP^{UPS}_m} P_{p'}$. We refer to this quantity as the pair-wise balanced load. We first minimize the surplus load $\Phi^t \in \R_+$ for any pair of top-level UPS devices $(p,q)$ as the difference between the total load allocated to power devices $p$ and $q$ and the pair-wise balanced load:
\begin{align}
    \label{eq:powerSurplus}
    \Phi^t & \geq 
    \sum_{\tau=1}^{t-1} 
    \sum_{j \in \calJ_p \cap \calJ_{p'}} 
    \sum_{i \in \calI_\tau}
    \rho_i \overline{x}^{\tau}_{ij}
    + 
    \sum_{i \in \calI_t}
    \sum_{j \in \calJ_p \cap \calJ_{p'}} 
    \rho_i x^{t}_{ij}
    - 
    \frac{1}{\binom{|\calP^{UPS}_m|}{2}} 
    \sum_{q \in \calP^{UPS}_m} P_{q},
    && 
    \ \forall \ m \in \calM, 
    \ \forall \ p, p' \in \calP^{UPS}_{m}
\end{align}

Second, we minimize the largest power load difference across all pairs of top-level UPS devices. This is written as  $\Gamma_{\text{U}}^t - \Gamma_{\text{L}}^t$, where $\Gamma_{\text{U}}^t, \Gamma_{\text{L}}^t \in \R_+$ are defined as follows:
\begin{align}
    \label{eq:powerBalance1}
    \Gamma_{\text{U}}^t \geq 
    \sum_{\tau=1}^{t-1} 
    \sum_{i \in \calI_\tau}
    \sum_{j \in \calJ_p \cap \calJ_{p'}} 
    \rho_i \overline{x}^{\tau}_{ij}
    + 
    \sum_{i \in \calI_t}
    \sum_{j \in \calJ_p \cap \calJ_{p'}} 
    \rho_i x^{t}_{ij}
    ,
    &&
    \ \forall \ m \in \calM, 
    \ \forall \ p, p' \in \calP^{UPS}_{m}
    \\
    \label{eq:powerBalance2}
    \Gamma_{\text{L}}^t \leq 
    \sum_{\tau=1}^{t-1} 
    \sum_{i \in \calI_\tau}
    \sum_{j \in \calJ_p \cap \calJ_{p'}} 
    \rho_i \overline{x}^{\tau}_{ij}
    + 
    \sum_{i \in \calI_t}
    \sum_{j \in \calJ_p \cap \calJ_{p'}} 
    \rho_i x^{t}_{ij}
    ,
    &&
    \ \forall \ m \in \calM, 
    \ \forall \ p, p' \in \calP^{UPS}_{m}
\end{align}
\end{itemize}

The modified objective function at time $t$ is then given by 
\begin{align}
    f_t(\bx^t, \by^t, \bxi^{1:t})
    - \sum_{i \in \calI_t} \sum_{m \in \calM} \lambda_m w_{im}^t
    - \tau \sum_{i \in \calI_t} \sum_{j \in \calJ} v_{ij}^t
    - \sum_{r \in \calR} \theta_r z_r^t
    - \alpha \Phi^t 
    - \beta (\Gamma_{\text{U}}^t - \Gamma_{\text{L}}^t)
\end{align}

\subsection{Production Setting}
\label{app:rackplacement_production_params}

In our rack placement algorithm deployed in production, the room minimization parameter $\lambda_m$ is equal to 40, 3 and 0, for rooms up to 0\%, 20\%, 100\% full respectively at the start of batch $\calI_t$. The row minimization parameter $\theta_r$ is equal to 2, 1, and 0 for rows up to 0\%, 50\%, and 100\% full respectively at the start of batch $\calI_t$. We have objective weights of $\tau = 1$, $\alpha = 10^{-3}$, $\beta = 10^{-5}$ for the tile group minimization parameter, the power surplus parameter, and the power load difference parameter. Since reward parameters are set to $r_i=200$, these objectives remain secondary objectives in the model.

One difference between the generic multi-stage stochastic optimization framework considered in this paper and the rack placement problem is that the latter does not evolve in a well-specified finite horizon $T$; rather, the horizon terminates when the data center can no longer accommodate incoming requests. Accordingly, we define a moving horizon: at each time period $t$, we sample requests for $k_t$ periods, where $k_t$ is determined so that future requests fill all non-empty rooms. We also prioritize incoming demands over sampled demands via a corresponding weight in the objective function---in particular, we ensure that the current requests are placed if they can be placed.

\section{Propensity Score Matching Analysis.}
\label{app:PSM}

We use propensity score matching (PSM) to corroborate our OLS regression estimates reported in \Cref{sec:empirical}. Specifically, we partition data centers into high- and low-adoption categories; we use a logistic regression model with the seven control variables to obtain each data center’s propensity score; and we then match each high-adoption data center to its neighbors within the low-adoption population, allowing for replacement \citep{rosenbaum1983central}. For robustness, we repeat the procedure with two thresholds between high- and low-adoption data centers (60\% and 45\%) and with 1, 2 and 4 neighbors within the low-adoption population for each high-adoption data center. The propensity score model has an area under the curve of $0.75$ with a 60\% cutoff and of $0.71$ with a 45\% cutoff, which satisfy the target threshold of $0.70$ \citep{defond2017client}.

\Cref{fig:controls_PSM} reports the distribution of the four continuous control variables across the high-adoption population and the matched low-adoption population. The visualization suggests that the two matched groups are highly balanced. To corroborate this observation, \Cref{tab:balance_PSM} shows that PSM mitigates the differences between the high- and low-adoption data centers across control variables. Specifically, the table indicates a standardized bias around $0.1$, a variance ratio below 2, and a Kolmogorov-Smirnov statistic around $0.1$--$0.3$, all of which are reflective of balanced distributions \citep{stuart2013prognostic,rubin2001using}. In turn, despite slight remaining disparities due to the small samples and the inherent variability in the control variables, the PSM method matches high-adoption data centers to ``similar’’ low-adoption data centers per the control variables.

\begin{figure}[h!]
\centering
\small
\subfloat[Demand.]{\includegraphics[width=0.48\textwidth]{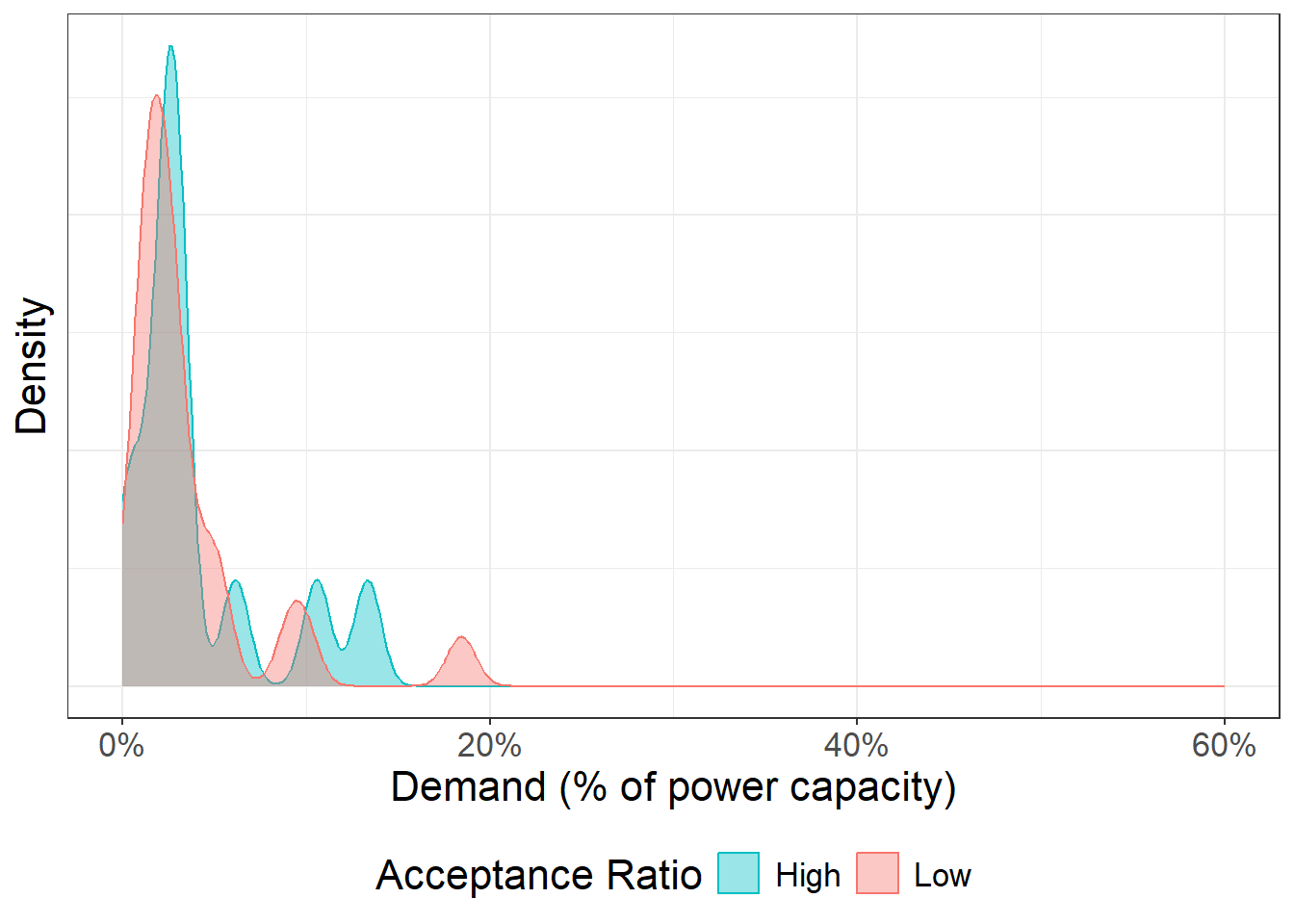}} 
\hspace{0.1 cm}
\subfloat[Initial utilization.]{\includegraphics[width=0.48\textwidth]{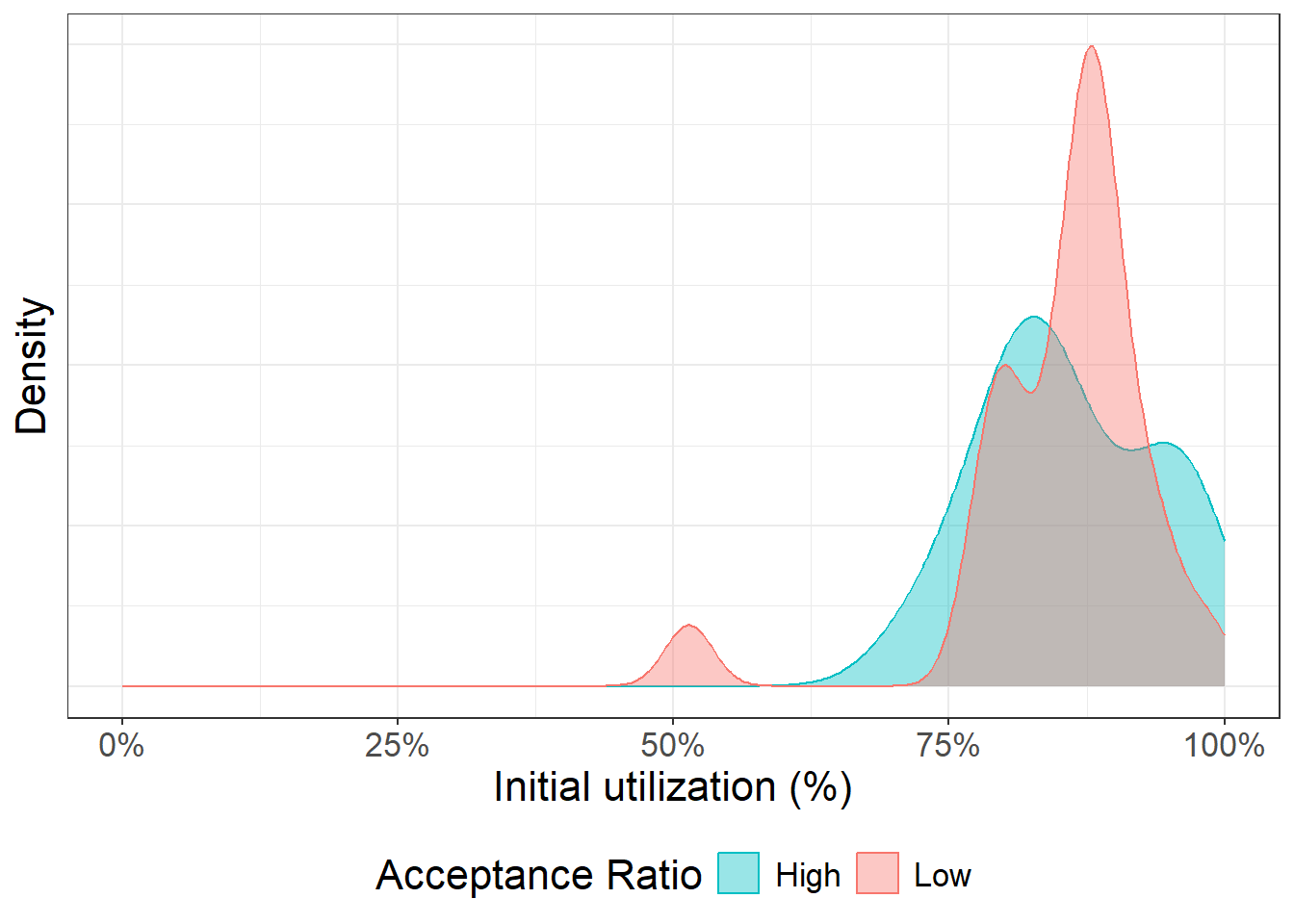}} 
\hspace{0.1 cm}
\subfloat[Initial power stranding.]{\includegraphics[width=0.48\textwidth]{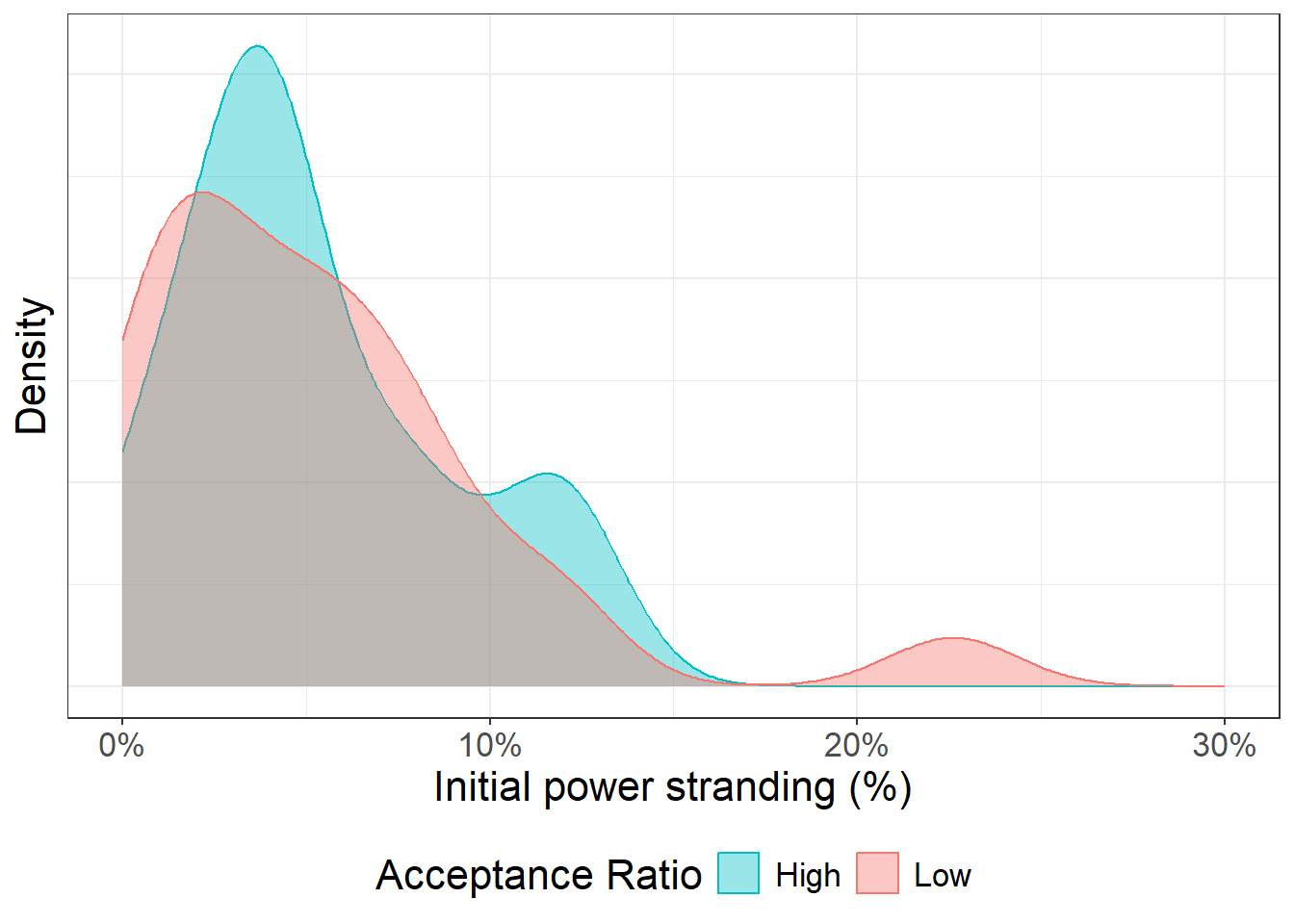}} 
\hspace{0.1 cm}
\subfloat[Power capacity.]{\includegraphics[width=0.48\textwidth]{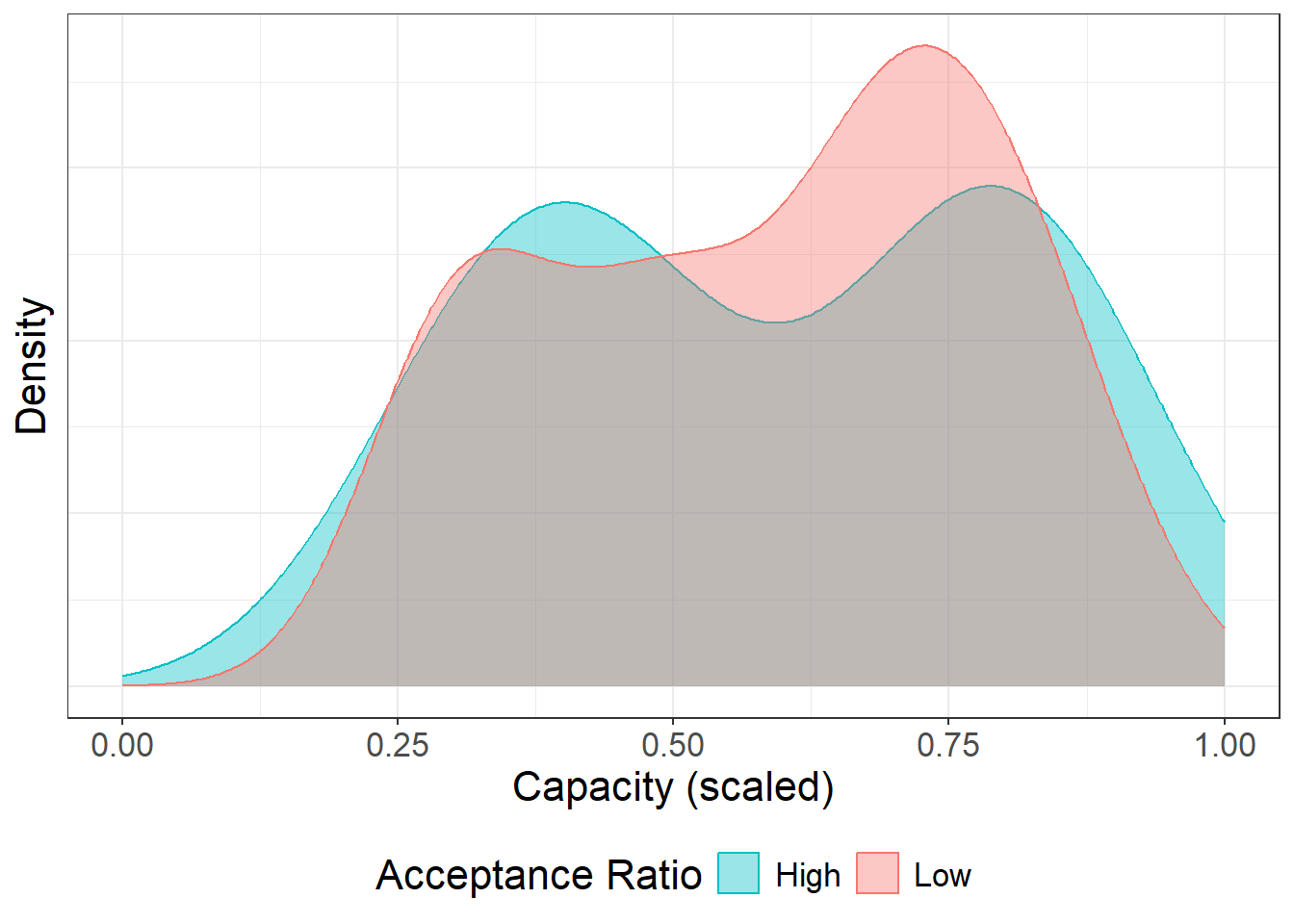}} 
\caption{Distribution of continuous control variables across data centers after PSM (using a 60\% threshold and four neighbors).} \label{fig:controls_PSM}
\vspace{-12pt}
\end{figure}

\begin{table}[h!]
    \centering
    {
    \footnotesize
    \setlength{\tabcolsep}{4pt}
    \renewcommand{\arraystretch}{1}
    \caption{Distributional balance after PSM (using 60\% and 45\% thresholds and four neighbors).}
    \label{tab:balance_PSM}
    \begin{tabular}{
        S[table-column-width=5cm,table-text-alignment=left]
        *{8}{S[table-format=-1.4]}
    }
        \toprule
    	& \multicolumn{4}{c}{Threshold: {60\%}} 
        & \multicolumn{4}{c}{Threshold: {45\%}}\\
	\cmidrule(lr){2-5}\cmidrule(lr){6-9}
	& {SB} & {VR} & {KS} & {($p$-value)} 
    & {SB} & {VR} & {KS} & {($p$-value)} \\
	\midrule
	{Demand} 
    & 0.1321 & 1.0353 & 0.2292 & (0.5745) 
    & -0.0392 & 1.2156 & 0.1563 & (0.6274) \\
	{Initial utilization} 
    & 0.0877 & 0.8033 & 0.2708 & (0.3745) 
    & 0.0828 & 1.9774 & 0.2292 & (0.1962)  \\
	{Initial power stranding} 
    & -0.0084 & 0.5659 & 0.2083 & (0.6821) 
    & -0.1315 & 0.4573 & 0.1771 & (0.4661)\\
	{IT capacity} 
    & 0.0221 & 1.2828 & 0.1667 & (0.8572) 
    & -0.0859 & 1.1487 & 0.1771 & (0.4321) \\
	{Rooms} 
    & 0.1522 & 1.4485 & 0.1458 & (0.6246) 
    & -0.1010 & 1.4169 & 0.1667 & (0.2466)\\
	{``Flex'' architecture} 
    & 0.0000 & 1.0682 & 0.0000 & (1.0000) 
    & 0.1251 & 1.0490 & 0.0625 & (0.6496)  \\
	{Location (US = 1, Europe = 0)}
    & -0.1549 & 1.3147 & 0.0625 & (0.6915) 
    & -0.1877 & 1.2238 & 0.0833 & (0.4435) \\
    \bottomrule
    \end{tabular}
    \begin{tablenotes}\vspace{-3pt}
    \item Threshold: boundary used to differentiate high-adoption data centers vs. low-adoption data centers.\vspace{-3pt}
	\item Standardized bias (SB): difference between means, divided by the pooled standard deviation.\vspace{-3pt}
	\item Variance ratio (VR): ratio between the sample variance of the high- and low-adoption data centers.\vspace{-3pt}
	\item Kolmogorov-Smirnov (KS) statistic: measure of the distance between cumulative distribution functions.
    \end{tablenotes}
    }
\end{table}

Finally, we use the PSM dataset to corroborate our previous findings. \Cref{fig:treatment_PSM} shows the distribution of the outcome variable partitioned between the high-adoption data centers and the matched population, using $60$\% and $45$\% thresholds. The qualitative observations echo those from \Cref{fig:treatment}, in that the distribution shifts to the left, leading to a lower average increase in power stranding among high-adoption data centers than low-adoption ones. Quantitative evidence confirms the impact of adoption on power stranding after controlling for covariates via PSM ($+1.03$\% vs. $+2.57$\% with a $60$\% threshold, and $1.65$\% vs. $3.71$\% with a 45\% threshold). The differences are of the same order of magnitude to the one found in the raw data (\Cref{fig:treatment}) and remain statistically significant---at the $5$\% level with a $60$\% threshold and at the $1$\% level with a $45$\% threshold. These results confirm that differences on power stranding are not merely due to the confounding effect of third variables, but can be attributed to differences in adoption of our algorithmic tool. 

\begin{figure}[h!]
\centering
\small
\subfloat[Threshold: 60\%.]{\includegraphics[width=0.48\textwidth]{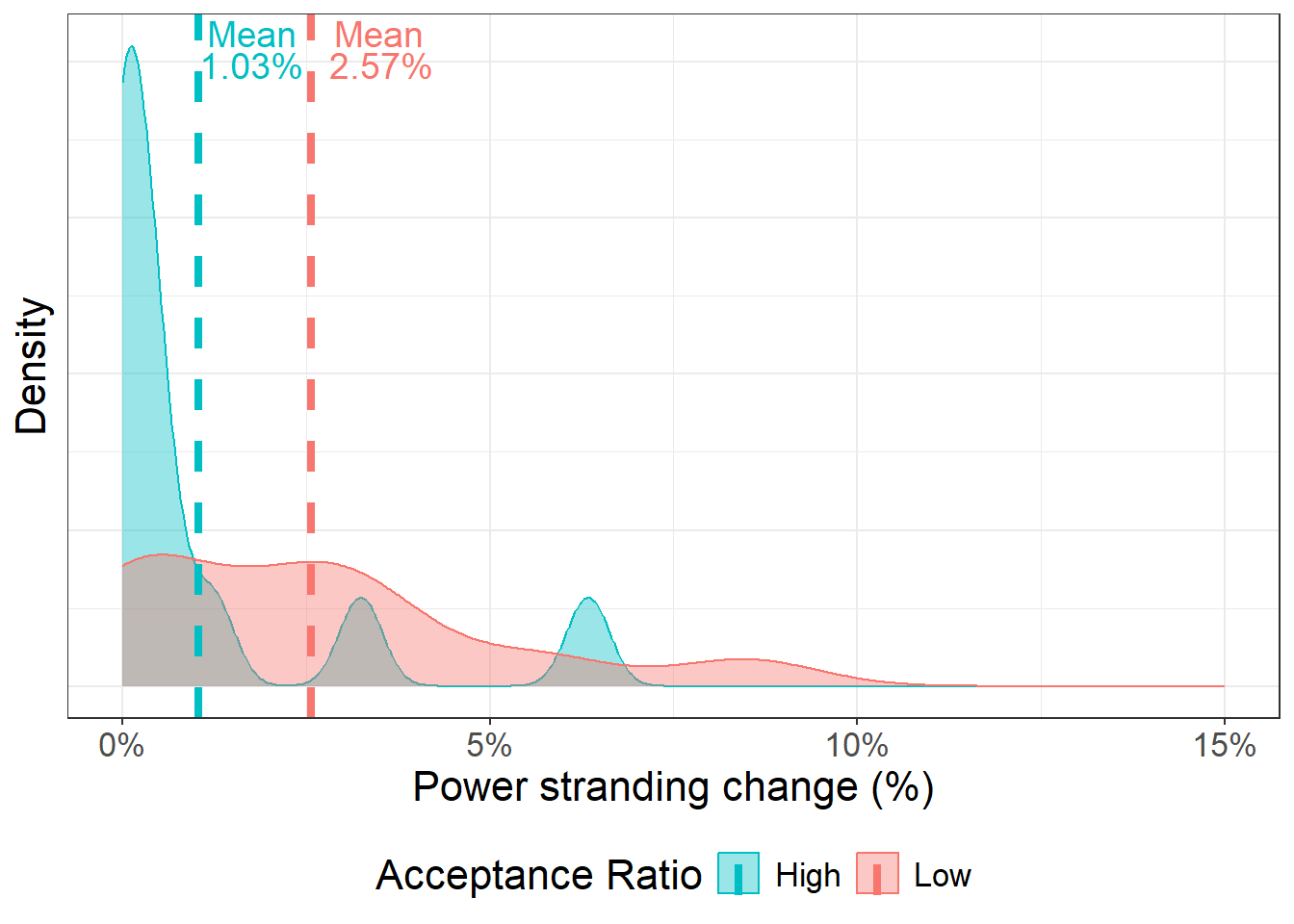}} 
\hspace{0.1 cm}
\subfloat[Threshold: 45\%.]{\includegraphics[width=0.48\textwidth]{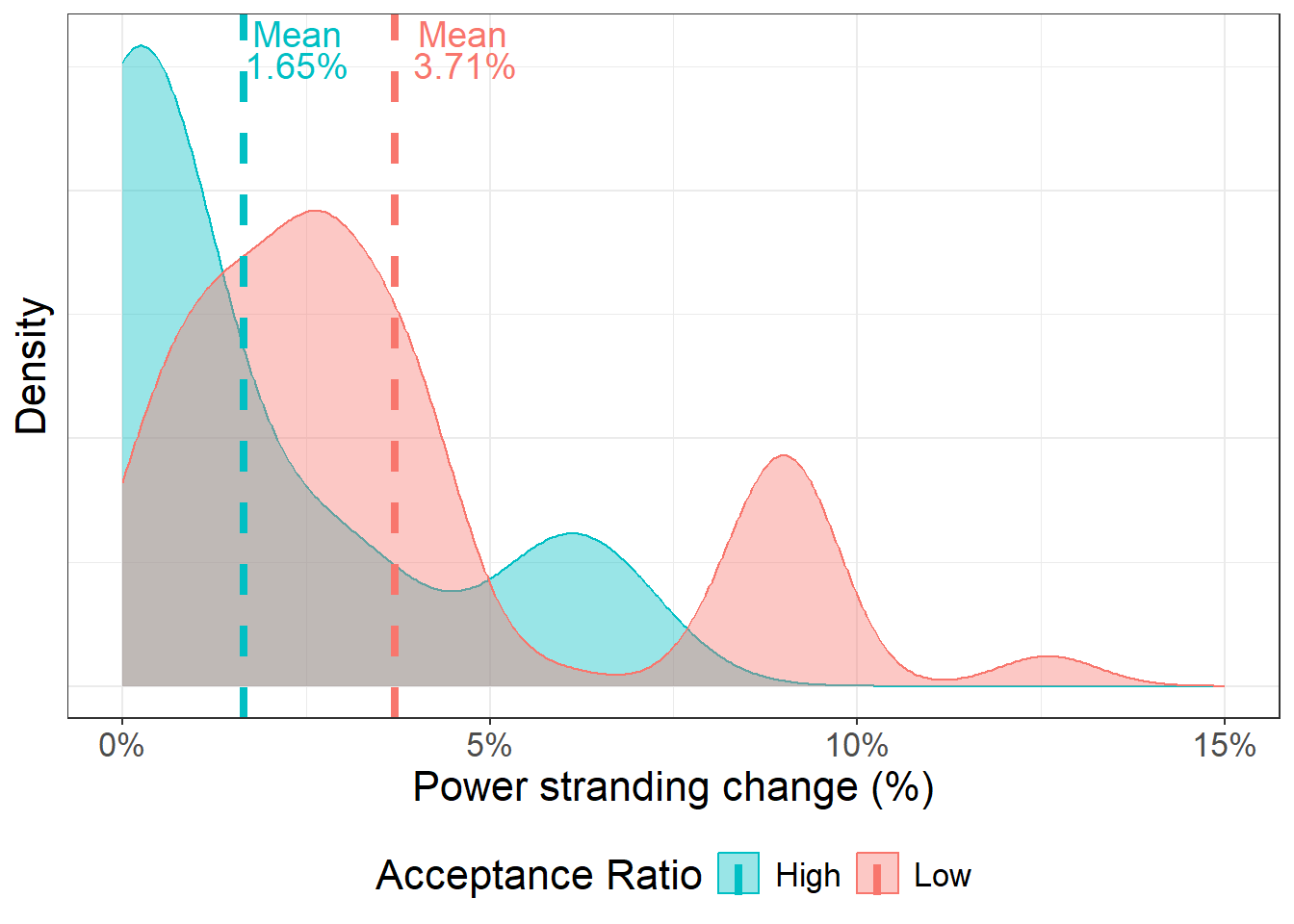}} 
\caption{Distribution of treatment variable across data centers after PSM (using 45\% and 60\% thresholds between high- and low-adoption data centers, and four low-adoption neighbors for each high-adoption data center), smoothed via Kernel Density Estimation from the \texttt{geom\_density} function in the \texttt{ggplot2} package in R.} \label{fig:treatment_PSM}
\vspace{-12pt}
\end{figure}

To conclude, \Cref{tab:regression_PSM} reports the regression estimates, with and without controls, using the eight PSM samples corresponding to the two adoption thresholds and the three matching neighborhoods along with a no-PSM baseline. By design, the no-PSM estimates without controls are identical to those from \Cref{fig:treatment}, and the PSM estimates with four estimates and without controls are identical to those from \Cref{fig:treatment_PSM}; the others provide robustness tests with one and two neighbors, and by adding the control variables in a PSM regression specification. However, the no-PSM estimates do not exactly coincide with those from \Cref{tab:regression} because of the different treatment variable---namely, we used a continuous measure of adoption between 0 and 1 in \Cref{tab:regression} versus a binary treatment variable separating high-adoption from low-adoption data centers in \Cref{tab:regression_PSM}. These results confirm that the impact of adoption on power stranding is negative across all specifications and statistically significant in the majority of cases---in 13 out of the 16 specifications. The magnitude of the coefficients ranges from $-1.1$\% to $-2.2$\%; this suggests that moving a data center from the low-adoption to the high-adoption category could reduce power stranding by 1 to 2 percentage points, which is also consistent with our baseline analysis.

\begin{table}[h!]
    \centering
    {
    \footnotesize
    \renewcommand{\arraystretch}{1}
    \setlength{\tabcolsep}{5pt}
    \caption{Regression estimates of the treatment effect across PSM specifications.}
    \label{tab:regression_PSM}
    \begin{tabular}{
        *{2}{S[table-text-alignment=left]}
        *{8}{S[table-format=-1.5,digit-group-size=5]}
    }
        \toprule
    	& 
        & \multicolumn{4}{c}{Threshold: {60\%}} 
        & \multicolumn{4}{c}{Threshold: {45\%}}\\
	\cmidrule(lr){3-6}\cmidrule(lr){7-10}
	{Controls?} & & {No PSM} & {1N} & {2N} & {4N} & {No PSM} & {1N} & {2N} & {4N} \\
    \midrule
    {No} & {Effect} 
    & -0.01988 & -0.01547 & -0.01213 & -0.01542
    & -0.01734 & -0.01243 & -0.01725 & -0.02064 
    \\
         & {$p$-value}
    & 0.01264\sym{**} & 0.09962\sym{*} & 0.1121 & 0.02887\sym{**}
    & 0.03854\sym{**} & 0.06167\sym{*} & 0.00702\sym{***} & 0.00069\sym{***}
    \\
    {Yes} & {Effect}
    & -0.02259 & -0.01196 & -0.01319 & -0.01601
    & -0.02037 & -0.01108 & -0.01802 & -0.02178 
    \\
          & {$p$-value}
    & 0.0304\sym{**} & 0.216 & 0.138 & 0.0441\sym{**}
    & 0.0220\sym{**} & 0.0880\sym{*} & 0.00572\sym{***} & 0.00032\sym{***}
    \\
    \bottomrule
    \end{tabular}
    \begin{tablenotes}
	\item 1N, 2N, 4N: Results after PSM using 1, 2, and 4 neighbors, respectively.
	\item \sym{*}, \sym{**}, and \sym{***} indicate significance levels of 10\%, 5\%, and 1\%.
    \end{tablenotes}
    }
\end{table}

\end{document}